\documentclass{article}

\newcommand{\W}{\mathcal{W}}
\newcommand{\T}{\mathcal{T}}

\newcommand{\cA}{\mathcal{A}}

\newcommand{\cL}{\mathcal{L}}

\newcommand{\cU}{\mathcal{U}}\newcommand{\cV}{\mathcal{V}}
\newcommand{\cW}{\mathcal{W}}

\newcommand{\Sym}{\mathbb{\mr{Sym}}}

\newcommand{\N}{\mathbb{N}}
\newcommand{\R}{\mathbb{R}}
\newcommand{\Q}{\mathbb{Q}}
\newcommand{\Z}{\mathbb{Z}}

\newcommand{\sS}{\mathfrak{S}}
\newcommand{\m}{\to}
\newcommand{\cd}{\mathrm{\bf D}}

\usepackage[vcentermath]{youngtab}
\newcommand{\Y}[1]{{\tiny\yng(#1)}}

\usepackage{url,graphicx,verbatim,amssymb,enumerate,stmaryrd, amsthm, mathrsfs, amsfonts, amsmath}
\usepackage{subcaption}
\usepackage[format=plain, labelfont=bf, textfont=it]{caption}
\usepackage[pagebackref,colorlinks,citecolor=blue,linkcolor=blue,urlcolor=blue,filecolor=blue]{hyperref}
\usepackage[dvipsnames, hyperref]{xcolor}
\usepackage{tikz} 
\usetikzlibrary{matrix,arrows}
\usepackage{microtype}
\usepackage{multicol}
\usepackage[all]{xy}

\usepackage[hmargin=3cm, vmargin=1.9cm]{geometry} %Shorten vertical and horizontal margins  
\usepackage{dsfont}
\usepackage{pinlabel} 

\newtheorem{theorem}{Theorem}[section]
\newtheorem{lemma}[theorem]{Lemma}
\newtheorem{proposition}[theorem]{Proposition}
\newtheorem{corollary}[theorem]{Corollary}
\newtheorem{conjecture}[theorem]{Conjecture}
\theoremstyle{definition}
\newtheorem{definition}[theorem]{Definition}
\newtheorem{example}[theorem]{Example}
\newtheorem{remark}[theorem]{Remark}

\newtheorem{convention}[theorem]{Convention}

\newtheorem{question}[theorem]{Question}

\newcommand{\mr}[1]{{\rm #1}}
\newcommand{\fS}{\mathfrak{S}}
\newcommand{\fB}{\mathfrak{B}}

\newcommand{\Hom}{\mathrm{Hom}}
\newcommand{\FI}{\mathrm{FI}}
\newcommand{\FB}{\mathrm{FB}}
\newcommand{\FIM}{\mathrm{FIM}}
\newcommand{\Inj}{\mathrm{Inj}}

\newcommand{\im}{\mathrm{im}}

\newcommand{\Tor}{\mathrm{Tor}}
\newcommand{\Top}{\mathcal{T}}
\newcommand{\Lie}{\mathcal L}
\newcommand{\Ind}{\mathrm{Ind}}
\newcommand{\Emb}{\mathrm{Emb}}

\newcommand{\Mod}{\mathrm{Mod}}
\newcommand{\Arc}{\mathrm{Arc}}
\newcommand{\up}{\mathrm{Ind}}

\title{Higher order representation stability and ordered configuration spaces of manifolds}
\author{Jeremy Miller\thanks{Jeremy Miller was supported in by NSF grant DMS-1709726.} \, and Jennifer C. H. Wilson}

\date{\today}

\begin{document}

\maketitle 

Using the language of twisted skew-commutative algebras, we define \emph{secondary representation stability},  a stability pattern in the {\it unstable} homology of spaces that are representation stable in the sense of Church--Ellenberg--Farb \cite{CEF}. We show that the rational homology of configuration spaces of ordered points in noncompact manifolds satisfies secondary representation stability. While representation stability for the homology of configuration spaces involves stabilizing by introducing a point ``near infinity,'' secondary representation stability involves stabilizing by introducing a pair of orbiting points -- an operation that relates homology groups in different homological degrees. This result can be thought of as a representation-theoretic analogue of \emph{secondary homological stability} in the sense of Galatius--Kupers--Randal-Williams \cite{GKRW1,GKRW2}. In the course of the proof we establish some additional results: we give a new characterization of the homology of the complex of injective words, and we give a new proof of integral representation stability for configuration spaces of noncompact manifolds, extending previous results to nonorientable manifolds.

 \setcounter{tocdepth}{4}
\setcounter{secnumdepth}{4}
 
\tableofcontents
\newpage
 
\section{Introduction} \label{intro}
 
The objective of this paper is to introduce the concept of \emph{secondary representation stability} and prove that this phenomenon is present in the homology of the ordered configuration spaces of a connected noncompact manifold. Church--Ellenberg--Farb \cite{CEF} proved that, in each fixed homological degree $i$, these homology groups are \emph{representation stable}: up to the action of the symmetric groups, the homology classes stabilize under the operation of adding a point ``near infinity." In this paper, we exhibit patterns between {\it unstable} homology groups in different homological degrees. We show that certain sequences of unstable rational homology groups stabilize under the new operation of adding {\it pairs} of points orbiting each other ``near infinity." We formalize this secondary representation stability phenomenon using the theory of twisted skew-commutative algebras.

\subsection{Stability for configuration spaces} \label{SectionIntroStabilization}

For a manifold $M$, let $F_k(M):=\{(m_1,\ldots,m_k) \, | \,  m_i \in M, \; m_i \neq m_j \text{ if } i \neq j\} \subseteq M^k$ be  the configuration space of $k$ distinct ordered points in $M$. The symmetric group $\fS_k$ acts on $F_k(M)$ by permuting the terms, and so induces a $\Z[\fS_k]$-module structure on the homology groups $H_i(F_k(M))$. Although these homology groups do not exhibit classical homological stability as $k$ increases, Church--Ellenberg--Farb \cite{Ch,CEF} showed that they do stabilize in a certain sense as $\fS_k$-representations. To  make this statement of \emph{representation stability} precise, we recall the definition of the stabilization map. 

Assume throughout that $M$ is a connected noncompact $n$-manifold with $n \geq 2$. Since $M$ is not compact, there is an embedding $e: M \sqcup \R^n \hookrightarrow M$ such that $e|_M$ is isotopic to the identity, as in Figure \ref{MEmbedding}. Such an embedding exists, for example, by Kupers--Miller \cite[Lemma 2.4]{kupersmillerimprov}. 
\begin{figure}[!ht]    \centering
\labellist
\Large \hair 0pt
\pinlabel {\fontsize{18}{40} $\R^n$} [c] at 106 40
\pinlabel {\fontsize{18}{40} $M$} [c] at 21 33
\pinlabel {\fontsize{18}{40} $M$} [c] at 302 40
\endlabellist
\begin{center}\scalebox{.5}{\includegraphics{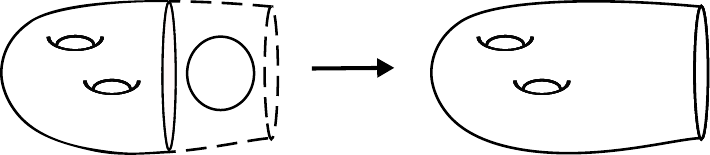}}\end{center}
\caption{The embedding $e:M \sqcup \R^n \hookrightarrow M$.}
\label{MEmbedding}
\end{figure}  
Using this embedding, we construct a map $$t:F_{k-1}(M) \m F_{k}(M)$$ which maps a configuration in $M$ to its image in $e(M)$, and then adds a point labeled by $k$ in $e(\R^n)$. This map is illustrated in Figure  \ref{StabilizationExample}.
\begin{figure}[!ht]    \centering
\begin{center}\scalebox{.5}{\includegraphics{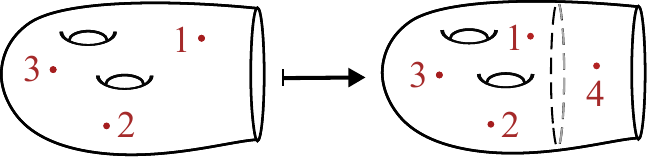}}\end{center}
\caption{The stabilization map $t:F_{3}(M) \m F_{4}(M)$.}
\label{StabilizationExample}
\end{figure}  

The following stability result is a consequence of work of  Church--Ellenberg--Farb \cite[Theorem 6.4.3]{CEF}.

\begin{theorem}[Church--Ellenberg--Farb {\cite[Theorem 6.4.3]{CEF}}] \label{CEFsurj}
Let $M$ be a connected, orientable, noncompact $n$-manifold with $n \geq 2$. For $i \leq \frac{k-1}{2}$,  
$$ \Z[\fS_{k}] \cdot t_*(H_i(F_{k-1}(M);\Z)) = H_i(F_{k}(M);\Z).$$ 
\end{theorem}

In this paper, we consider a higher order stabilization map, $t'$. Using the embedding $e$ we can also construct a map $F_{k-2}(M) \times F_2(\R^n) \m F_{k}(M)$ which places two points in $e(\R^n)$, labeled by $(k-1)$ and $k$. This induces a map $H_a(F_{k-2}(M)) \otimes H_b(F_2(\R^n)) \m H_{a+b}(F_{k}(M)).$ We then define the stabilization map 
$$t':H_{i-1}(F_{k-2}(M)) \m H_i(F_{k}(M))$$
 by pairing a class in $H_{i-1}(F_{k-2}(M))$ with the class in $H_1(F_2(\R^n))$ of the point labeled by $k$ orbiting the point labeled by $(k-1)$ counterclockwise, as in Figure \ref{SecondaryStabilizationSample}. This class is zero for $n \geq 3$, but is nonzero for $n=2$. Note that this operation is symmetric in $k$ and $(k-1)$. 
\begin{figure}[!ht]    \centering
\begin{center}\scalebox{.5}{\includegraphics{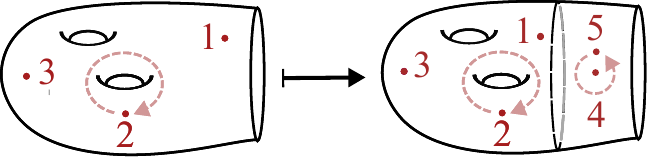}}\end{center}
\caption{The secondary stabilization map $t':H_{1}(F_{3}(M)) \m H_2(F_{5}(M))$.}
\label{SecondaryStabilizationSample}
\end{figure}
While the classical stabilization map $t_*$ raises the number of points by one and keeps homological degree constant, the map $t'$ increases the number of points by two and homological degree by one. 

With the definition of $t'$, we can state the following version of our main theorem, \emph{secondary representation stability} for the rational homology of configuration spaces. For this theorem we do not need to assume $M$ is orientable, but we assume that our manifold $M$ is finite type (that is, the homotopy type of a finite CW complex) to ensure that the rational homology groups of the configuration spaces are finite-dimensional. Let $\N_0$ denote the set of nonnegative integers.

\begin{theorem}\label{maindecat}
Let $M$ be a connected noncompact finite type $n$-manifold with $n \geq 2$. There is a function $r:\N_0 \m \N_0$ tending to infinity such that for $i \leq \frac{k-1}{2}+r(k)$,  $$\Q[\fS_{k}] \cdot \Big( t_*(H_i(F_{k-1}(M);\Q))+t'(H_{i-1}(F_{k-2}(M);\Q)) \Big) = H_i(F_{k}(M);\Q).$$
\end{theorem}

Up to the action of $\fS_k$, the homology group $H_i(F_{k}(M);\Q)$ is generated by the images of $t_*$ and $t'$ in a range. In other words, Theorem \ref{CEFsurj} says that when the homological degree $i$ is small enough relative to the number $k$ of points, the group $H_i(F_{k}(M);\Q)$ is spanned by classes where at least one point is stationary ``near infinity.''  Theorem \ref{maindecat} says that there is a larger range in which the homology group is spanned by classes where at least one point is stationary, or two points are orbiting each other ``near infinity.''

When dim$(M) \geq 3$, we will see that Theorem 1.2 implies an improved representation stability range for the groups $H_i(F_k(M); \Q)$. For 2-manifolds, however, this result is a novel form of stability among these homology groups.

\begin{remark}The idea to study homological degree-shifting stabilization maps originated with the work of Galatius--Kupers--Randal-Williams \cite{GKRW1,GKRW2}. Their work generalizes classical homological stability, whereas we generalize representation stability. 
See also Hepworth \cite[Theorem B and C]{HepworthEdge} for a related result. 
\end{remark}

\subsection{Categorical reformulation}\label{seccat}

In order to prove Theorem \ref{maindecat}, and interpret it within the broader field of representation stability, we will reformulate the result in terms of finite generation of a module over a certain enriched category (or equivalently as a module over a certain \emph{twisted skew-commutative algebra}). From this perspective, Theorem \ref{maindecat} becomes a structural algebraic result on the homology of configuration spaces. We now review elements of the theory of FI-modules.

\subsubsection*{FI-modules}

Let FI denote the category of finite sets and injective maps. An \emph{FI-module} (over a commutative unital ring $R$) is a covariant functor $\cV$ from FI to the category of $R$-modules. 

Given an FI-module $\cV$, we write $\cV_S$ to denote the image of $\cV$ on a set $S$, or for $k \in \N_0$ we let $\cV_k$ denote the value of $\cV$ on the standard set $[k]:=\{1,\ldots,k\}$ or $[0]:=\varnothing$. The endomorphisms End$_{\FI}([k]) \cong \fS_k$ induce an action of  $\fS_k$ on $\cV_k$. The FI-module structure on $\cV$ is completely determined by these $\fS_k$-actions and the maps $\cV_k \to \cV_{k+1}$ induced by the standard inclusions $[k] \subset [k+1]$. 

Given an FI-module $\cV$,  the \emph{minimal generators} $H_0^{\FI}(\cV)$ of $\cV$ are a sequence of $\fS_k$-representations that we think of as encoding the ``unstable" elements of $\cV$. In degree $k$, the $\fS_k$-representation $H_0^{\FI}(\cV)_k$ is defined to be the cokernel
\[ H_0^{\FI}(\cV)_k := \mathrm{cokernel}\left( \bigoplus_{a \in [k]} \cV_{ [k] \setminus \{a\} } \m \cV_k \right) \]
where  the maps are induced by the natural inclusions $[k] \setminus \{a\}  \hookrightarrow [k]$. Minimal generators should not necessarily be viewed as FI-module generators; in general they are a quotient and not a subobject. They do, however, give a lower bound on the size of a generating set. They are analogous to the \emph{indecomposable} elements of an algebra with respect to an augmentation.  onendomorphisms morphisms act by zero. 

We say that an FI-module $\cV$ is \emph{generated in degree $\leq d$} (or has \emph{generation degree} $\leq d$) if $$H_0^{\FI}(\cV)_k \cong 0 \qquad \text{ for $k>d.$}$$ 
We say that $\cV$ is \emph{finitely generated} if $\bigoplus_{k \geq 0} H_0^{\FI}(\cV)_k$ is finitely generated as an $R$-module. Finite generation is equivalent to the condition that there is a finite subset of $\bigoplus_{k \geq 0} \cV_k$ whose images under the FI morphisms generate $\bigoplus_{k \geq 0} \cV_k$ as an $R$-module.

The FI-modules central to this paper have additional structure: they are \emph{free} FI-modules in the sense of Definition \ref{DefnFreeModule}. A free FI-module $\cV$ admits natural splittings $ H_0^{\FI}(\cV)_k \hookrightarrow \cV_k$, and in this case the images of minimal generators under these splittings do give a canonical generating set for $\cV$. Free FI-modules are highly constrained;  all FI morphisms act by injective maps, and they are completely determined by their minimal generators (see Theorem \ref{4.1.5}, quoting \cite[Theorem 4.1.5]{CEF}). 

\subsubsection*{Stability in the homology of configuration spaces}

Given a noncompact manifold $M$ of dimension at least $2$ and $i \in \N_0$, the $i$th homology groups $\{H_i(F_k(M))\}_{k=0}^\infty$ of the configuration spaces have the structure of an FI-module, denoted $H_i(F(M))$, which we now describe. We take homology with coefficients in a fixed commutative, unital ring $R$ unless otherwise stated. Given a finite set $S$, let $F_S(M)$ denote the space of embeddings of $S$ into $M$. 

If $|S|=k$, a choice of bijection $S\cong [k]$  gives a homeomorphism $F_S(M) \cong F_k(M)$. Every injective map of sets $f: S \hookrightarrow T$ defines a map $\bar{f}: F_S (M) \to F_T (M)$, as in  Figure \ref{ModuleStructure}. 
\begin{figure}[!ht]    \centering
\labellist
\Large \hair 0pt
\pinlabel {\fontsize{20}{40} $f: S \hookrightarrow T$} [c] at 25 95
\pinlabel {\fontsize{20}{40} $\bar{f}$} [c] at 268 75
\endlabellist
\begin{center}\scalebox{.46}{\includegraphics{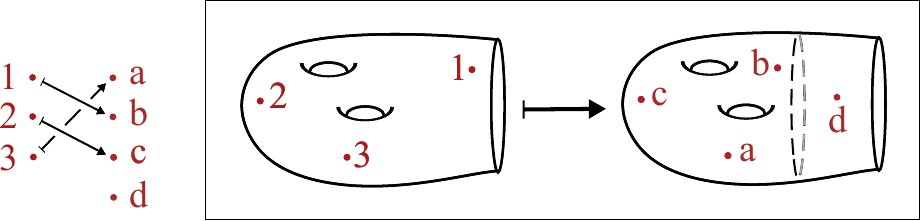}}\end{center}
\caption{The FI-module structure on $H_i(F(M);R)$.}
\label{ModuleStructure}
\end{figure}  
We use the injection $S \hookrightarrow T$ to relabel the configuration, and insert points labeled by the elements of $T\setminus f(S)$ in the image $e(\R^n)$ of the embedding $e$.

 Although the map $\bar{f}$ depends on many choices, up to homotopy it only depends on the isotopy class of the embedding $e$ and the injection $S \hookrightarrow T$, and so for a fixed choice of embedding we obtain a well-defined FI-module structure on the homology groups $H_i(F(M))$. In the language of FI-modules, Theorem \ref{CEFsurj} is the statement that $H_0^{\FI}(H_i(F(M))_S$ vanishes when $|S|>2i$. If $M$ has finite type then the FI-module $H_i(F(M))$ is finitely generated. For $k \geq 2i$, every homology class in $H_i(F_k(M))$ is an $R$-linear combination of homology classes of the form of Figure \ref{StableHomology}:  there are at most $2i$ points moving around $M$ in an $i$-parameter family,  and the remaining points remain fixed ``near infinity.''  
\begin{figure}[!ht]
\begin{center}\scalebox{.5}{\includegraphics{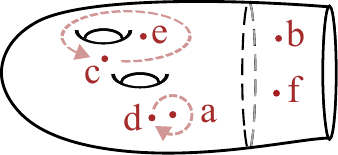}}\end{center}
\caption{A stable homology class in $H_{2}(F(M))_{\{a,b,c,d,e,f\}}$.}
\label{StableHomology}
\end{figure}

Church--Ellenberg--Farb showed that the homology groups of configuration spaces $H_i(F(M))$ are free FI-modules when $M$ is noncompact \cite[Definition 4.1.1 and Section 6.4]{CEF}.  The $\fS_k$-representations $H_0^{\FI} (H_i(F(M)))_k$ therefore determine all homology groups of $F_k(M)$; the objective of this paper is to achieve a better understanding of these groups.

\subsubsection*{Secondary representation stability}
 
In general there are no natural nonzero maps from $H_0^{\FI} (H_i(F(M)))_k$ to $H_0^{\FI}(H_i(F(M)))_{k+1}$. However, $t'$ induces a map $$H_0^{\FI}\Big(H_i(F(M))\Big)_k \longrightarrow H_0^{\FI}\Big(H_{i+1}(F(M))\Big)_{k+2}, $$ and our main result is a stability result with respect to this operation. 
 
Given $i \geq 0$ and a finite set $S$, let $\W_i^M(S)$ be the sequence of minimal generators $$\W_i^M(S):=H_0^{\FI}\left(H_{\frac{|S|+i}{2}} \left( F(M); R\right) \right)_S.$$ By convention, fractional homology groups are zero. Any injection $S \hookrightarrow T$ with $|T|-|S|=2$ induces a map $\W_i^M(S) \m \W_i^M(T)$ as shown in Figure \ref{WedgeModuleStructure}. 
\begin{figure}[!ht]    \centering
\labellist
\Large \hair 0pt
\pinlabel {\fontsize{20}{40} $g: S \hookrightarrow T$} [c] at  25 110
\pinlabel {\fontsize{20}{40} $g_*$} [c] at  268 75
\endlabellist
\begin{center}\scalebox{.46}{\includegraphics{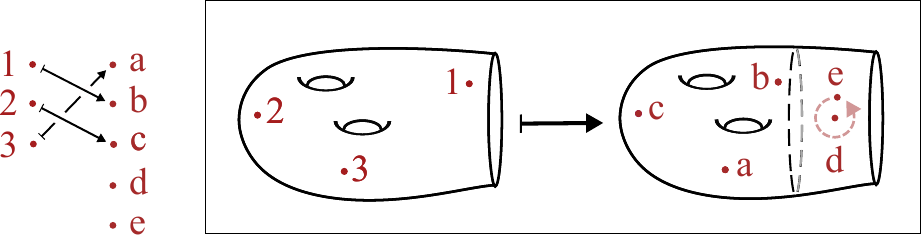}}\end{center}
\caption{Stabilization by orbiting points.}
\label{WedgeModuleStructure}
\end{figure}  

If $|T|-|S| = 2d$ for $d>1$, the data of the injection is not enough to define a map $\W_i^M(S) \m \W_i^M(T)$. In addition to the injection $f: S \hookrightarrow T$, we choose a \emph{perfect matching} on the complement $T\setminus f(S)$, that is, a partition of $T\setminus f(S)$ into $d$ sets of size $2$. This matching determines how the points will be paired. To specify the sign of the resultant homology class, we then choose an \emph{orientation} on the perfect matching (see Definition \ref{DefFIM+}). We define a stabilization map on the homology of $F_S(M)$ by introducing these $d$ pairs of orbiting points ``near infinity,'' as in Figure \ref{Match}.
\begin{figure}[!ht]    \centering
\labellist
\Large \hair 0pt
\pinlabel {\fontsize{20}{40} $G$} [c] at  33 146
\pinlabel {\fontsize{20}{40} $G_*$} [c] at  295 123
\endlabellist
\begin{center}\scalebox{.46}{\includegraphics{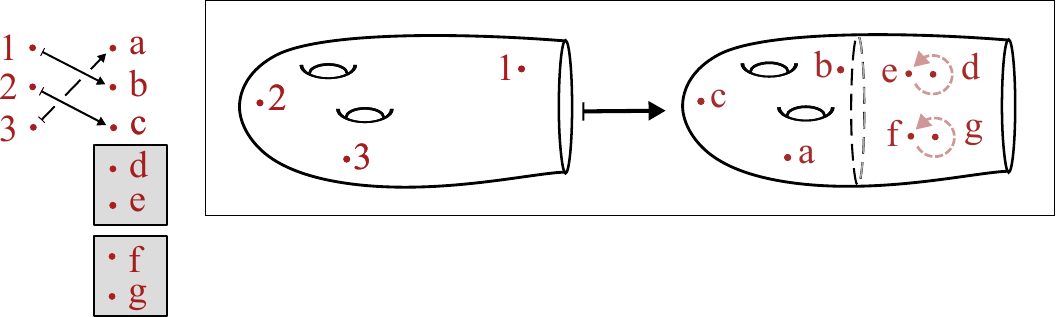}}\end{center}
\caption{The $\bigwedge \,(\Sym^2 R)$-module structure on $\W_i^M$.}
\label{Match}
\end{figure}  

These operations and the $\fS_k$-actions give the sequences $\W_i^M$ the structure of modules over the \emph{twisted skew-commutative algebra $\bigwedge \, (\Sym^2 R)$}, or, equivalently, a module over the enriched category FIM$^+$ of Definition \ref{DefFIM+}. See work of Sam--Snowden and Nagpal--Sam--Snowden \cite{SStcas, SS1, SSN1, NSSskew} and Section \ref{FIreview} for more information on twisted (skew-)commutative algebras. In this language, Theorem \ref{maindecat} can be formulated as follows.

\begin{theorem} \label{maintheorem}
If $R$ is a field of characteristic zero and $M$ is a connected noncompact manifold of finite type and dimension at least two, then for each $i\geq 0$ the sequence of minimal generators $$\W_i^M(k)=H_0^{\FI}\left(H_{\frac{k+i}{2}} \left( F(M); R\right) \right)_k$$  is finitely generated as a $\bigwedge \,(\Sym^2 R)$-module.
\end{theorem}

We call this finite generation result \emph{secondary representation stability}.  This implies that there is some number $N_i$ such that for any $k$ the minimal generators $H_0^{\FI} (H_{\frac{i+k}{2}}(F(M)))_k$ are spanned by classes of the form given in Figure \ref{SecondaryStableHomology}, where all but at most $N_i$ many points move in orbiting pairs ``near infinity.''
\begin{figure}[!ht]    \centering
\begin{center}\scalebox{.5}{\includegraphics{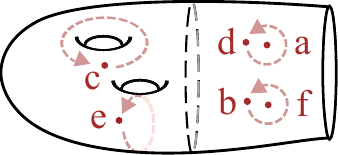}}\end{center}
\caption{A secondary stable class in $H_0^{\FI}\left( H_4(F(M)) \right)_{\{a,b,c,d,e,f\}}$.}
\label{SecondaryStableHomology}
\end{figure}  
For connected, noncompact surfaces, representation stability  is shown graphically in Figure \ref{RepStabilityGraph}, and secondary representation stability in Figure \ref{SecondaryStabilityGraph}.
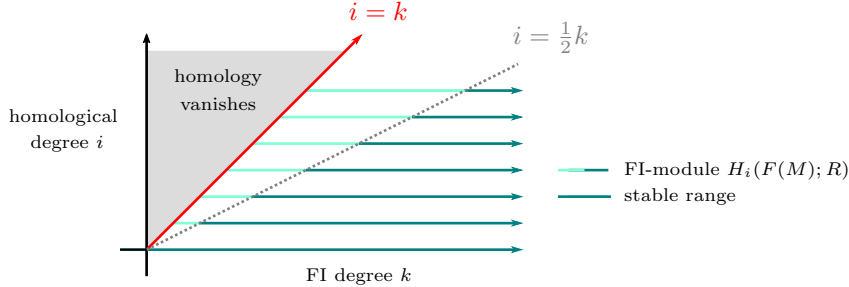
\begin{figure}[!ht]    \centering
\labellist
\Large \hair 0pt
\pinlabel {\scriptsize homological } [c] at -20 60
\pinlabel {\scriptsize degree $i$} [c] at -20 50
\pinlabel {\scriptsize FI degree $k$} [c] at 90 0
\pinlabel {\fontsize{10}{10}{ \color{red} $i=k$}} [c] at 95 100
\pinlabel {\fontsize{10}{10} { \color{gray} $i=\frac12 k$ }} [c] at 160 90
\pinlabel {\scriptsize homology} [c] at 37 75
\pinlabel {\scriptsize vanishes} [c] at 37 65
\pinlabel {\scriptsize FI-module $H_i(F(M); R)$} [l] at 190 40
\pinlabel {\scriptsize stable range } [l] at 190 30
\pinlabel {\scriptsize } [l] at 190 20
\endlabellist
\begin{center}\scalebox{1}{\includegraphics{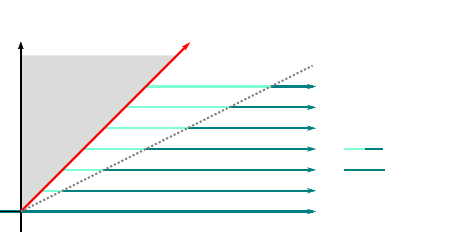}}\end{center}
\caption{The FI-modules $H_i(F(M); R)$ for a noncompact connected surface $M$.}
\label{RepStabilityGraph}
\end{figure}  
\begin{figure}[!ht]    \centering
\labellist
\Large \hair 0pt
\pinlabel {\scriptsize  homological } [c] at  -20 60
\pinlabel {\scriptsize   degree $i$} [c] at  -20 50
\pinlabel {\scriptsize  FI degree $k$} [c] at  90 0
\pinlabel {\fontsize{10}{10}{ \color{red} $i=k$}} [c] at  110 120
\pinlabel {\fontsize{10}{10} { \color{blue}  $i=\frac12 k$ }} [l] at  205 105
\pinlabel {\scriptsize homology vanishes} [c] at  47 97
\pinlabel {\scriptsize (above homological} [c] at  47 87
\pinlabel {\scriptsize dimension)} [c] at  47 77
\pinlabel {\scriptsize minimal generators vanish} [c] at  115 35
\pinlabel {\scriptsize (rep stability range)} [c] at  120 25
\pinlabel {\scriptsize module $H_0^{\FI}\Big(H_{\frac{k+j}{2}} \left( F(M) ;R\right) \Big)_k$ } [l] at  190 65
\pinlabel {\scriptsize secondary rep stability range } [l] at  190 45
\pinlabel {\scriptsize (precise bounds not known)} [l] at  190 35
\endlabellist
\begin{center}\scalebox{1}{\includegraphics{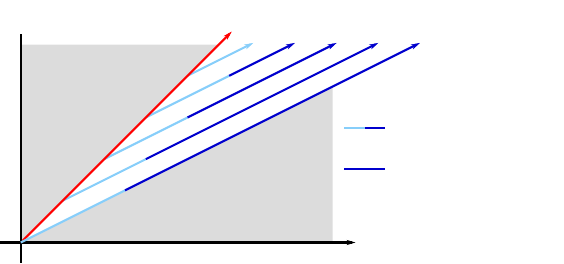}}\end{center}
\caption{The minimal generators $H_0^{\FI}\left(H_i(F(M))\right)_k$ for a noncompact connected surface $M$.}
\label{SecondaryStabilityGraph}
\end{figure}  

Viewing these homology groups as a $\bigwedge \,(\Sym^2 R)$-module and drawing on the theory of twisted skew-commutative algebras, we can prove a version of the main theorem that establishes isomorphisms instead of just surjections. 

\begin{corollary} \label{corSecondaryCentral} Let $R$ be a field of characteristic zero. For $k$ sufficiently large compared to $i$, $\W_i^M(k)$ is isomorphic to the quotient of $\Ind_{\fS_{k-2} \times \fS_2 }^{\fS_k} \W_i^M(k-2) \boxtimes R$ by the image of the sum of the two natural maps: \[ \Ind_{\fS_{k-4} \times \fS_2 \times \fS_2 }^{\fS_k}  \Big ( \W_i^M(k-4)  \Big ) \boxtimes R \boxtimes R \rightrightarrows \Ind_{\fS_{k-2} \times \fS_2 }^{\fS_k} \Big ( \W_i^M(k-2)  \Big ) \boxtimes R. \] Here $R$ represents the trivial $\fS_2$-representation. \end{corollary}

Concretely, this says that, in the stable range, $ H_0^{\FI}\left(H_{\frac{k+i}{2}} \left( F(M); R \right) \right)_k$ is the coequalizer of the (appropriately signed) maps

 \[ \Ind_{\fS_{k-4} \times \fS_2 \times \fS_2 }^{\fS_k}  H_0^{\FI}\left(H_{\frac{k+i}{2}-2} \left( F(M); R\right) \right)_{k-4} \boxtimes R \boxtimes R\rightrightarrows \Ind_{\fS_{k-2} \times \fS_2 }^{\fS_k}  H_0^{\FI}\left(H_{\frac{k+i}{2}-1} \left( F(M); R\right) \right)_{k-2}\boxtimes R.\] 
 
\noindent In particular, the representations $\W_i^M(k-4)$ and $\W_i^M(k-2)$ together with the maps $\W_i^M(k-2) \m \W_i^M(k-4)$ completely determine the representations $\W_i^M(k)$ in the stable range.   This corollary can be viewed as a secondary version of \emph{central stability} in the sense of Putman \cite{Pu}. The stability range where the isomorphisms of Corollary \ref{corSecondaryCentral} hold is typically smaller than the surjectivity range of Theorem \ref{maintheorem}.

If $M$ is at least three-dimensional, then the maps $\Ind^{\fS_k}_{\fS_{k-2} \times \fS_2 } \W^M_i(k-2) \boxtimes R \m \W^M_i(k)$ are both zero and surjective in a range. Hence, $\W^M_i(k)$ vanishes for $k$ sufficiently large, and secondary representation stability is the statement that $H_i(F_k(M))$ is representation stable in an improved range. In Theorem \ref{ImprovedRangeTheorem} we prove explicit stability bounds for these homology groups with integral coefficients. 

 For surfaces, however, the groups $\W^M_i(k)$ are generally nonzero as $k$ tends to infinity. For example, $\W^{\R^2}_i(2k+i)$ is a sequence of free abelian groups whose ranks grows super-exponentially in $k$; see Proposition \ref{PropWR20}. In Section \ref{secConj}, we formulate some conjectures for tertiary and higher order stability.
 
Since it was first observed that FI-modules could be interpreted in the language of tca's (Definition \ref{defTCA}), it has been an open question (see Part 4 of Motivation 1.2 of \cite{SSN1}) if algebraic properties of more general (skew-)tca's would have applications to topology in a similar fashion to the theory of FI-modules. Our paper represents one of the  first examples of such an application.

\subsubsection*{The proof of secondary representation stability}
The proof of Theorem \ref{maintheorem} involves the analysis of a semi-simplicial space, the \emph{arc resolution} of $F_k(M)$, described in Section \ref{SectionArcRes}. In Section \ref{SectionDifferentials}, we compute certain differentials in spectral sequences associated to the arc resolutions, which we use to prove the desired finiteness properties of the sequences $\W^M_i$ in Section \ref{SectionMainProof}.
The algebraic underpinnings of our proof of secondary representation stability is developed in Section \ref{SectionAlgebraicFoundations}, and draws on the theory of FI--modules introduced by Church--Ellenberg--Farb \cite{CEF}, the central stability complex introduced by Putman in \cite{Pu}, and the theory of twisted skew-commutative algebras. In particular, our proof relies on the Noetherian property for $\bigwedge \,(\Sym^2 R)$-modules established by Nagpal--Sam--Snowden \cite[Theorem 1.1]{NSSskew}.

This Noetherian property for $\bigwedge \,(\Sym^2 R)$-modules is currently only known when $R$ is a field of characteristic zero.  If it were possible to prove this result over more general commutative unital rings $R$, then (with a modification of our Proposition \ref{E2filtration}) our proof would establish our main results, Theorem \ref{maindecat}, Theorem \ref{maintheorem}, and Corollary \ref{corSecondaryCentral}, over these rings. Some conjectural generalizations and strengthenings of Theorem \ref{maintheorem} are discussed in Section \ref{secConj}.

\subsection{Other results}
 
In the process of establishing secondary representation stability for configuration spaces, we prove some other results which may be of independent interest. In particular, we prove new representation stability results for the homology of configuration spaces, and we give a new Lie-theoretic description of the top homology group of the complex of injective words.

\subsubsection*{The homology of the complex of injective words}

The \emph{complex of injective words} $\Inj_\bullet(k)$ on the set $[k]$ is a semi-simplicial set which was used by Kerz \cite{Ker} to give a new proof of homological stability for the symmetric groups (see Definition \ref{definj}). It has found application in algebraic topology, representation theory, and algebraic combinatorics.  The complex of injective words has only one nonvanishing reduced homology group, a subgroup of the free abelian group on the set of $k$-letter words on the set $[k]$.  In Section \ref{SectionInjHomology}, we describe an explicit basis for this group that resembles the Poincar\'e--Birkhoff--Witt basis for the free Lie superalgebra on $[k]$. The following result may be viewed as an analogue of the Solomon--Tits Theorem \cite{SolomonSteinberg} for the complex of injective words. \\[-3pt]

{\noindent \bf  Theorem \ref{TopHomology}.} {\it The reduced integral homology group $\tilde H_{k-1}(||\Inj_\bullet(k)||)$ is the submodule of the free associative algebra on the set $[k]$ generated by products of iterated graded commutators where every element of $[k]$ appears exactly once. An explicit $\Z$--module basis for this group is given in Lemma \ref{LProductBasis}.} \\[-3pt]

\subsubsection*{Primary representation stability for configuration spaces}
 
The work of Church, Ellenberg, Farb, and Nagpal \cite{Ch, CEF, CEFN} on representation stability for configuration spaces uses Totaro's spectral sequence \cite{Totaro}, which assumes that the manifold is orientable. We remove this assumption by giving an entirely different proof of representation stability for configuration spaces (see also Palmer \cite[Remark 1.8]{palmertwisted} and Casto \cite[Corollary 3.3]{Casto}). Following methods of Putman \cite{Pu} on congruence subgroups, we adapt Quillen's approach to homological stability to prove representation stability.

\vspace{.12in}
{\noindent \bf  Theorems \ref{ConfigSpaceRepStable} and \ref{ImprovedRangeTheorem}.} {\it Let $M$ be a connected noncompact manifold of dimension $n \geq 2$.
\begin{itemize}
\item[(a)] Then $ H_0^{\FI}(H_i(F(M);\Z))_k \cong 0$ for $k>2i.$
\item[(b)] Suppose $M$ has dimension at least $3$. Then  $ H_0^{\FI}(H_i(F(M);\Z))_k \cong 0$ for $k > i$.
\end{itemize} }

\subsection{Acknowledgments}
 
Our project was inspired by results on \emph{secondary homological stability} by S\o ren Galatius, Alexander Kupers, and Oscar Randal-Williams. Using the theory of $E_n$-cells, these authors have established secondary homological stability in many examples including classifying spaces of general linear groups and mapping class groups of surfaces \cite{GKRW1,GKRW2}. Our project benefited greatly from our interaction with them.
 
We would like to thank Martin Bendersky, Thomas Church, Jordan Ellenberg, Benson Farb,  S\o ren Galatius, Patricia Hersh, Ben Knudsen, Alexander Kupers, Rita Jim\'enez Rolland, Rohit Nagpal, Andrew Putman, Oscar Randal-Williams, Vic Reiner, Steven Sam, Andrew Snowden, Bena Tshishiku, John Wiltshire-Gordon, Jesse Wolfson, and Arnold Yim for helpful conversations. We thank our anonymous referee for an exceptionally close reading and detailed comments on the paper.

\section{Algebraic foundations} \label{SectionAlgebraicFoundations}

The goal of this section is to lay the algebraic groundwork necessary to state and prove the main theorem.  We begin, in Section \ref{FIreview}, with a review of FI-modules and their generalizations, modules over a twisted (skew-)commutative algebra. This provides a very general context for formulating representation stability for sequences of $\fS_k$-representations. We then discuss the relationship between Putman's central stability chain complex \cite{Pu}  and Farmer's complex of injective words \cite{Fa} in Section \ref{SectionPutmanTwisted}. In Section \ref{SectionInjHomology}, we give a new description of the homology of the complex of injective words. In Section \ref{subsecsecworcomplex},  we conclude with an analysis of a generalization of the central stability chain complex for FIM$^+$-modules. These chain complexes will appear in Section \ref{SectionConfigurationSpaces} on the pages of the arc resolution spectral sequence, a spectral sequence we use to prove secondary representation stability for configurations spaces.
\subsection{Review of twisted (skew-)commutative algebras} \label{FIreview}

Throughout this paper, we fix a commutative unital ring $R$. All homology groups will be assumed to have coefficients in $R$, all tensor products will be taken over $R$, and so forth, unless otherwise specified. 

\begin{definition} \label{DefnFI,FB}
Let FI be the category whose objects are finite (possibly empty) sets and whose morphisms are injective maps. Let FB be the category of finite sets and bijective maps.  
\end{definition} 

\begin{definition} \label{DefnC-module} \label{DefnC-module}
Let $\mathcal C$ be a category. A \emph{$\mathcal C$-module (over the ring $R$)} is a covariant functor from $\mathcal C$ to $R$-$\Mod$, the category of $R$-modules. A \emph{$\mathcal C$-space} is a covariant functor from $\mathcal C$ to the category of topological spaces and a \emph{homotopy} $\mathcal C$-space is a covariant functor from $\mathcal C$ to the homotopy category of topological spaces. Co-$\mathcal C$-modules and (homotopy) co-$\mathcal C$-spaces are the corresponding contravariant functors.
 \end{definition}

Recall that we denote the value of an FB or FI-module $\cV$ on a set $S$ by $\cV_S$ (or possibly $\cV(S)$ in instances where $\cV$ has other subscripts). When $S$ is the set $[k]= \{1, 2, \ldots, k\}$, we write $\cV_k$ or $\cV(k)$.

The category of FI-modules studied by Church--Ellenberg--Farb \cite{CEF} was later understood to be an example of a category of modules over a twisted commutative algebra (tca). We will use the theory of (skew-)tca's to define secondary representation stability, and we summarize the relevant aspects of this theory here. 

\begin{definition}
Let $\cV$ and $\cW$ be $\FB$-modules. The \emph{Day convolution} of two $\FB$-modules $\cV$ and $\cW$ is the $\FB$-module defined by the formula $$(\cV \otimes_\FB \cW)_S := \bigoplus_{A \sqcup B=S} \cV_A \otimes \cV_B. $$ 
\end{definition}

This product is symmetric monoidal with symmetry $\tau: \cV \otimes_{\FB} \cW \m \cW \otimes_{\FB} \cV$ induced by the canonical bijection $A \sqcup B \m B \sqcup A$.

\begin{definition} \label{defTCA}
A \emph{(skew-)twisted commutative algebra} is a (skew-)commutative unital monoid object in category of $\FB$-modules with Day convolution. A \emph{module over a twisted (skew-)commutative algebra} is a module object over the associated monoid object.

 \end{definition}

See Sam--Snowden \cite[Section 8]{SStcas} for more details.

\begin{definition}
Let $$T: \text{FB-}\Mod  \longrightarrow \text{FB-}\Mod$$ be given by the formula $$(T \cV ): = \bigoplus_k \left(  \cV^{\otimes_\FB \, k}   \right).$$  
\end{definition}

For the same reasons that tensor algebras are unital rings, this is a unital monoid object with respect to the Day convolution. In particular, there is a natural multiplication map $ \mu : T\cV \otimes_\FB T \cV  \m T \cV. $ 

%The map $\tau$ of Definition \ref{defTCA} gives an automorphism of $\tau:\cW \otimes_\FB \cW$ for any $\FB$-module $\cW$.

\begin{definition} 

Let TCA denote the category of twisted commutative algebras over $R$, and let STCA denote the category twisted skew-commutative algebras over $R$.   Let $\Sym: \text{FB-}\Mod  \to \text{TCA} $ be given by the formula $$\Sym \cV:=\mathrm{cokernel} \Big( \mu -\tau \circ \mu: T \cV \otimes_\FB T \cV \m T \cV \Big)$$ and let $\bigwedge: \text{FB-}\Mod \to \text{STCA} $ be given by the formula $$\bigwedge \cV:=\mathrm{cokernel} \Big( \mu +\tau \circ \mu: T \cV \otimes_\FB T \cV \m T \cV \Big)$$ where the multiplicaiton is induced by monoid structure on $T \cV$. Let $\Sym^k \, \cV$ or $\bigwedge^k \cV$ denote the image of $\cV^{ \otimes_{\FB} k}$ in $\Sym \, \cV$ or $\bigwedge \cV$ respectively.

\end{definition}

The tca $\Sym \, (\Sym^1 R)$ is the FB-module with a rank-1 trivial $\fS_k$-representation $R$ in every degree, and all multiplication maps given by the canonical isomorphisms $R \otimes R \cong R$. The data of a module over $\Sym \, (\Sym^1 R)$ is equivalent to an FI-module $\cV$ over $R$. See Sam--Snowden \cite[Section 10.2]{SStcas}. 

The tca  $\Sym  \, (\Sym^2 R)$ is generated by $$\Sym \, (\Sym^2 R)_{\{a,b\}} \cong R\langle x_{a,b} \; | x_{a,b} = x_{b,a} \rangle.$$ The multiplication map is given by multiplication of (commutative) monomials in the variables $x_{a,b}$, with the caveat that by definition we must take the disjoint union of the indices of each factor. For this reason $\Sym \, (\Sym^2 R)_S$ is not simply a polynomial algebra on variables of the form $x_{a,b}$; the indices of any monomial are all distinct by construction. Modules over $\Sym \, (\Sym^2 R)$ are equivalent to modules over the combinatorial category FIM we now define (see also Sam--Snowden \cite[Section 4.3]{SS1}).  

\begin{definition} \label{DefnFIM}
A \emph{matching} of a set $B$ is a set of disjoint $2$-element subsets of $B$, and a matching is a \emph{perfect matching} if the union of these subsets is $B$. Let FIM be the category whose objects are finite sets and whose morphisms are injective maps $f: S \hookrightarrow T$ together with the data of a perfect matching of the complement $T\setminus f(S)$ of the image. Composition of morphisms is defined by composing injective maps and taking the union of one matching with the image of the other.
 
\end{definition} 

The skew-tca  $\bigwedge \, (\Sym^2 R)$ is generated by $\bigwedge \, (\Sym^2 R)_{\{a,b\}}.$ In general for sets $S$ of even parity the group $\bigwedge \, (\Sym^2 R)_S$ is spanned by anticommutative monomials with distinct indices $$x_{a_1, b_1} \cdots x_{a_d, b_d} \qquad \text{such that} \qquad S=\{a_1, b_1, \ldots, a_d, b_d\}.$$ The category of modules over $\bigwedge \, (\Sym^2 R)$ cannot be encoded as a functor category to $R$-$\Mod$, however, $\bigwedge \, (\Sym^2 R)$-modules are equivalent to modules over an enriched category which we denote by $\FIM^+$. 

\begin{definition} \label{DefFIM+} Let FIM$^+$ be the following category enriched over $R$-$\Mod$. The objects are finite sets. The module of morphisms between sets of different parity is the $R$-module $0$. Between sets $[a-2b]$ and $[a]$, the module of morphisms is the following quotient: 
$$\frac{ R \left\langle( f: [a-2b] \to [a], \; A_1, A_2,  \ldots, A_b) \; \middle|  \; \begin{array}{l} f \text{ is injective}, \quad  |A_i|=2, \quad [a] = \im(f) \sqcup A_1 \sqcup \cdots \sqcup A_b \\  \text{so } \{ A_i \} \text{ is an ordered perfect matching on }[a]\setminus \im(f) \end{array} \right \rangle  }{\left\langle \; \left(f, \; A_1, A_2, \ldots A_b \right)  = \mr{sign}(\sigma) \left( f, \;  A_{\sigma(1)}, A_{\sigma(2)}, \ldots, A_{\sigma(b)}   \right)  \quad \text{for all }\sigma \in \fS_{b}  \; \right\rangle } $$ 
In other words, when $k \equiv m \pmod 2$, the morphisms from $[k]$ to $[m]$ are the free $R$-module on the set of all  injective maps $[k]  \hookrightarrow [m]$ along with a perfect matching on the complement of the image. These perfect matchings are oriented and reversing the orientation gives a sign. We denote a free generator of the morphisms by  $$F=(f , A_1  \wedge A_2 \wedge \cdots  \wedge A_b ).$$  
\noindent The composition of the maps 
$$F=(f , \; A_1  \wedge A_2 \wedge \cdots  \wedge A_b ) \quad \text{and} \quad G=(g , \; C_1  \wedge C_2 \wedge \cdots  \wedge C_d )$$ 
is given by the map 
$$ G \circ F := (g\circ f , \; C_1  \wedge C_2 \wedge \cdots  \wedge C_d \wedge g(A_1)  \wedge g(A_2) \wedge \cdots  \wedge  g(A_b)  ). $$
\end{definition}

\begin{definition} Let $\mathcal C$ be a category enriched over $R$-$\Mod$. We define a \emph{$\mathcal C$-module} to be an enriched functor from $\mathcal C$ to $R$-$\Mod$.
\end{definition}

We now extend the definition of $H_0^\FI$ to modules over a general (skew-)tca.

\begin{definition}
Let $\mathcal V$ be a module over a (skew-)tca $\mathcal{A}$.  Let $H_0^\mathcal{A}(\mathcal V)_S$ be the quotient 
$$H_0^\mathcal{A}(\mathcal V)_S := \mathrm{cokernel}\left( \bigoplus_{S=P\sqcup Q, \; P \neq \varnothing} \cA_P  \otimes \cV_Q \longrightarrow \cV_S \right).$$
 The $R$-modules $H_0^\mathcal{A}(\mathcal V)_S$ assemble to form an $\FB$-module. We say that $\mathcal V$ is \emph{finitely generated} if $\bigoplus_{k=0}^\infty H^\mathcal{A}_0(\mathcal V)_k$ is finitely generated as an $R$-module.
\end{definition}
We often replace the superscript $\mathcal{A}$ in the notation $H_0^\mathcal{A}(\mathcal V)$ with the corresponding category. Following Church--Ellenberg \cite{CE}, we use the following terminology. 
\begin{definition} \label{defdeg}
Let $\mathcal V$ be an $\mathcal{A}$-module with $\mathcal{A}$ a (skew-)tca. We say that $\deg \mathcal V \leq d$ if $\mathcal V_k=0$ for all $k>d$. We say $\mathcal V$ is \emph{generated in degrees $\leq d$} if $\deg H_0^\mathcal{A}(\mathcal V) \leq d$.
\end{definition}

The following Noetherianity result of Nagpal--Sam--Snowden \cite[Theorem 1.1]{NSSskew} is key to our understanding of $\FIM^+$-modules, equivalently, of $\bigwedge (\Sym^2 R)$-modules. 

\begin{theorem}[Nagpal--Sam--Snowden {\cite[Theorem 1.1]{NSSskew}}] \label{NSSnoth}
Let $R$ be a field of characteristic zero. Any submodule of a finitely generated module over $\bigwedge \,( \Sym^2 R)$ is finitely generated.
\end{theorem}

A similar Noetherian property also holds for FI-modules; see Snowden \cite[Theorem 2.3]{SnowdenSyzygies}, Church--Ellenberg--Farb \cite[Theorem 1.3]{CEF}, and Church--Ellenberg--Farb--Nagpal \cite[Theorem A]{CEFN}. 

\begin{definition} \label{DefnFreeModule} Let $\cA$ be a (skew-)tca. We define $$M^\mathcal{A} : \FB\text{-}\Mod \longrightarrow \mathcal{A}\text{-}\Mod$$ via the formula $$ M^\mathcal{A} (\cV):=\cA \otimes_\FB \cV. $$ The $\cA$-module structure on $M^\mathcal{A} (\cV)$ is induced by the map $\cA \otimes_\FB \cA \m \cA$. We call modules of the form $M^\mathcal{A}(\cV)$  \emph{free} $\mathcal A$-modules. Given an $R[\fS_d]$-module $W$, we define $M^{\mathcal A}(W)$ by viewing $W$ as the FB-module with module $W$ in degree $d$ and $0$ in all other degrees. We let $M^{\mathcal A}(d):=M^{\mathcal A}(R[\fS_d])$.
  We will often replace the superscript $\mathcal A$ with its corresponding category, and (following Church--Ellenberg--Farb \cite[Definition 2.2.2]{CEF}) simply write $M$ for $M^{\FI}$. \end{definition}

We now give another description of $M^\mathcal{A} (d)$ for the (skew)-tca's of interest. 

\begin{proposition}\label{altDescriptionOfFree}
There is a natural isomorphism of functors $M^{\FIM^+}(d)$ and $R\left[\Hom_{\FIM^+}([d], -  )\right]$. Similarly, there is a natural isomorphism of functors $M^{\FI}(d)$ and $R\left[\Hom_{\FIM^+}([d], -  )\right]$.
\end{proposition}
 \begin{proof}
Both pairs of functors are left adjoint to the forgetful functor $\mathcal{A}$-Mod $\to$ FB-$\Mod$ for $\cA=\bigwedge \Sym^2 R$ or $\cA=\Sym \, \Sym^1 R$. 
\end{proof}

See Proposition \ref{PropDecomposingM(d)} for an explicit description of $M^{\FIM^+}(W)$ as induced representations. Using the fact that $M(W)$ is a left Kan extension, Church, Ellenberg, and Farb observed that, given an $\fS_d$-representation $W$, the free FI-module $M(W)$ satisfies
 $$M(W)_k \; \cong \; \bigoplus_{\substack{A \subseteq [k] \\ |A|=d}} W \; \cong \; \Ind^{\fS_k}_{\fS_d \times \fS_{k-d}} W \boxtimes R$$ where $R$ denotes the trivial $\fS_{k-d}$--representation. These authors prove that the free FI-modules $M(W)$ can be promoted to modules over the larger category $\FI\sharp$, which we define as follows.

\begin{definition} \label{FISharp}
Define a \emph{based injection} $f: S_0 \to T_0$ between two based sets $S_0, T_0$ to be a based map such that $|f^{-1}(\{a\})| \leq 1$ for all elements $a \in T_0$ except possibly the basepoint. Let FI$\sharp$ be the category whose objects are finite based sets and whose morphisms are based injections. 
\end{definition}

The category defined in Definition \ref{FISharp} is isomorphic to the category called FI$\sharp$ by Church--Ellenberg--Farb \cite[Definition 4.1.1]{CEF}. The operation of adding a basepoint gives an embedding of categories $\FI \subseteq \FI\sharp$. Hence an $\FI\sharp$-module is an FI-module with additional structure and constraints, notably, the FI morphisms have one-sided inverses and so must act by injective maps. These backwards maps give $\FI\sharp$-modules the structure of co-FI-modules, and  we may view $\FI\sharp$-modules as co-FI-modules with a compatible FI-module structure. The following result of Church, Ellenberg, and Farb gives a classification of FI$\sharp$-modules: they are precisely the free $\FI$-modules. They show moreover that the functors $M:$ FB-$\Mod \to \FI\sharp$-$\Mod$ and $H^{\FI}_0: \FI\sharp$-$\Mod \to$ FB-Mod are inverses, and define an equivalence of categories. 

\begin{theorem}[Church--Ellenberg--Farb {\cite[Theorem 4.1.5]{CEF}}] \label{4.1.5}
An FI-module $\cV$ is the restriction of an FI$\sharp$-module if and only if it is free, in which case it is the restriction of a unique FI$\sharp$-module. In particular, for an FI$\sharp$-module $\mathcal V$, there is a natural isomorphism
$$\mathcal V \cong \bigoplus_{k=0}^\infty M\Big(H_0^{\FI}(\mathcal V)_k\Big).$$
\end{theorem}

 Theorem \ref{4.1.5} implies that  an FI$\sharp$-module $\cV$ is completely determined by its minimal generators.

\subsection{Twisted injective word complexes} \label{SectionPutmanTwisted}

 Putman \cite{Pu} defined a chain complex associated to a sequence of $\fS_k$-representations called the \emph{central stability chain complex}. This chain complex arises as the $E^1$-page of a certain spectral sequence, and its homology is the $E^2$-page. Natural analogues of the chain complex exist when the symmetric groups are replaced by other families of groups such as general linear groups. See for example Putman--Sam \cite[Section 5.3]{PutmanSamLinearGroups}. In the context of FI-modules, we show that this chain complex is closely related to the complex of injective words and accordingly we will denote the complex using the notation $\Inj$.  We first recall the definition of the \emph{complex of injective words}. 
 
 \begin{definition} \label{definj}
For a set $S$ and an integer $i \geq -1$, let $\Inj_i(S):=\Hom_{\FI}(\{0,\ldots,i\},S)$.
\end{definition}

For a fixed set $S$, $\Inj_\bullet(S)$ has the structure of an augmented semi-simplicial set. The face map $d_j$ acts by precomposition with the order-preserving injective map $\{0,\ldots,i-1\} \to \{0,\ldots,i\}$ that misses the element $j$. Farmer \cite{Fa} proved the following result on the connectivity of $|| \Inj_\bullet(S)||$.

\begin{theorem}[Farmer \cite{Fa}]  \label{Farmer}
The geometric realization $|| \Inj_\bullet(S)||$ is $|S|-2$ connected.
\end{theorem}

Since $|| \Inj_\bullet(S)||$ has dimension $|S|-1$, the reduced homology of $|| \Inj_\bullet(S)||$ is concentrated in dimension $|S|-1$. We now recall Putman's central stability chain complex, which we view as a twisted version of the complex of injective words.

\begin{definition} \label{DefnInj.}
For a set $S$, an FI-module $\mathcal V$, and integer $i \geq -1$, let 
$$\Inj_i(\mathcal V)_S:=\bigoplus_{f :\{0,\ldots,i\} \hookrightarrow S} \mathcal V_{S\setminus \im(f)}.$$ 
These groups assemble into an augmented semi-simplicial FI-module $\Inj_\bullet(\mathcal V)$. Let $\Inj_*(\mathcal V)$ denote the associated FI-chain complex. When $\cV$ is the FI-module $M(0)$, for a set $S$ the complex $\Inj_*(\mathcal V)_S$ is precisely the chain complex associated to the augmented semi-simplicial set $\Inj_\bullet(S)$. 
\end{definition}

\begin{remark} Given an FI-module $\mathcal V$, the chain complex $\Inj_*(\mathcal V)$ has appeared in the literature under a variety of different notations, and frequently with a shift in indexing. It is closely related to Putman's chain complex $\mathrm{IA}_{*+1}(\mathcal{V}_{n-*-1})$ \cite[Section 4]{Pu}, and the complexes computing \emph{$\FI$--homology} in work of Church, Ellenberg, Farb, Nagpal \cite{CEF, CEFN, CE} and Gan and Li \cite{GanLES, GanLi.OnCentralStability}. The complex $\Inj_*(\mathcal V)$ itself is denoted by $B_{*+1}(\mathcal V)$ in Church--Ellenberg--Farb--Nagpal \cite[Definition 2.16]{CEFN},  by $\Sigma_{*+1} (\mathcal V) $ in Putman--Sam \cite[Section 3]{PutmanSamLinearGroups}, by $\tilde{C}^{F_{\mathcal V}}_{\bullet+1}$ in Church--Ellenberg \cite[Section 5.1]{CE}, and by $\widetilde{C}^1_*\mathcal{V}$ in Patzt \cite[Definition 2.5]{Patzt.CentralStability}. We apologize for adding yet another name for this chain complex. 
\end{remark}

The goal of this subsection is to compute the homology of this chain complex on FI$\sharp$-modules.  

\begin{remark}\label{RemInj.Vanish} Suppose that $\mathcal V$ is an FI-module such that $\mathcal V_k = 0$ for all $k< d$. Observe that by Definition \ref{DefnInj.}, $\Inj_i(\mathcal V)_S=0$ whenever $|S|-i-1 < d$. 
\end{remark}

\begin{remark} \label{Inj.Formulas}
It follows from the definition of  $\Inj_i(\mathcal V)_k$ that there is an isomorphism of $\mathfrak S_k$-representations 
$$\Inj_i(\mathcal V)_k \cong \mathrm{Ind}_{\mathfrak  S_{k-i-1}}^{\mathfrak S_k} \mathcal V_{k-i-1} . $$ In particular, for the FI-module $M(d)$ there is an isomorphism of $\fS_k$-representations 
$$\Inj_i(M(d))_k \cong M(d+i+1)_k . $$
Given an $\mathfrak S_d$-representation $W$,  there is an isomorphism of $\fS_k$-representations 
$$\Inj_i(M(W))_k \cong M\Big(\mathrm{Ind}_{\mathfrak S_{d}}^{\mathfrak S_{d+i+1}} W\Big)_k . $$
\end{remark}

\begin{lemma} \label{H.Inj.d}

There is an isomorphism of $\fS_n$-representations: $$H_*(\Inj_*(M(d)))_n \cong \tilde H_{*} \left( \bigvee_{g \in \Hom_{\FI}([d],[n])} ||\Inj_\bullet([n]-\im(g))|| \right). $$
\end{lemma}

\begin{proof} This follows from the existence of a natural isomorphism of chain complexes between $\Inj_*(M(d))_n$ and the direct sum over $g \in \Hom_{\FI}([d],[n])$ of the reduced cellular chains of $||\Inj_\bullet([n]-\im(g))||$. 
\end{proof}

\begin{theorem} \label{connectedM(W)}
Let $W$ be an integral representation of $\mathfrak S_d$. There is an isomorphism: $$H_i(\Inj_*(M(W)))_S \cong \Big(H_i \big(\Inj_*(M(d))\big) \otimes_{\Z[\sS_d]} W \Big)_S.$$ In general, given an $\FI\sharp$-module $\mathcal V$, $$H_p (\mathrm{Inj}_*(\mathcal V
))_k= \mathrm{Ind}_{\sS_{p+1}\times \sS_{k-p-1} }^{\sS_{k}} H_p (\mathrm{Inj}_*(p+1)) \boxtimes (H_0^{\FI}(\mathcal V))_{k-p-1}.$$
\end{theorem}

\begin{proof}
Recall that $M(W) \cong M(d)\otimes_{\Z[\mathfrak S_d]} W$.  Then $$\Inj_*(M(W))_S \cong \Inj_*(M(d))_S \otimes_{\Z[\mathfrak S_d]} W.$$ 
The homological K\"unneth spectral sequence (see for example Theorem 10.90 of Rotman \cite{RotmanHomological}),  is a first quadrant spectral sequence:  $$ E^2_{p,q}  =\Tor_p^{\Z[\mathfrak S_d]}\Big(H_q \Big( \Inj_*(M(d))\Big)_S, W\Big).$$ Since the $\Z[\fS_d]$-modules $\Inj_q(M(d))_S$ are flat, the spectral sequence converges to $H_{p+q}(\Inj_*(M(W)))_S$. Theorem \ref{Farmer} and Lemma \ref{H.Inj.d} imply that $E^2_{p,q}=0$ except for $q = |S|-1-d$.  Since the $E^2_{p,q} $ page has only a single nonzero column, the spectral sequence collapses on this page. The limit is nonzero  only when  $i \geq (|S|-1-d)$, and in this case we see that: $$ H_{i}(\Inj_*(M(W)))_S \cong \Tor_{i-(|S|-1-d)}^{\Z[\mathfrak S_d]}\Big(H_{|S|-1-d} \Big( \Inj_*(M(d))\Big)_S, W\Big).$$ On the other hand, $M(W)_k = 0$ for $k<d$, and so by Remark \ref{RemInj.Vanish},  $H_i(\Inj_*(M(W)))_S=0$ whenever $i>|S|-d-1$. 

Thus this spectral sequence has a single nonzero entry. The homology groups $H_{i}(\Inj_*(M(W)))_S $ are nonzero only in degree $i=|S|-1-d$, in which case we have 
\begin{align*} H_{|S|-1-d}(\Inj_*(M(W)))_S &\cong \Tor_{0}^{\Z[\mathfrak S_d]}\Big(H_{|S|-1-d} \Big( \Inj_*(M(d))\Big)_S, W\Big) \\ 
&\cong \Big(H_{|S|-1-d} \Big( \Inj_*(M(d))\Big) \otimes_{\Z[\sS_d]} W \Big)_S. \end{align*} 
\noindent Theorem 4.1.5 of \cite{CEF} (here Theorem \ref{4.1.5}) implies that every $\FI\sharp$-module is a direct sum of modules of the form $M(W)$. Additionally, for an $\fS_d$-representation $W$, $H_0^{\FI}(M(W))_d \cong W$ and $H_0^{\FI}(M(W))_i \cong 0$ for $i \neq d$. These two facts imply the general result.
 \end{proof}

We obtain the following corollary.  

\begin{corollary} \label{twistedInj}
Let $\mathcal V$ be an FI$\sharp$-module with generation degree $\leq d$. Then $H_i(\Inj_*(\mathcal V))_S =0$ for $i \leq |S|-2-d$.
\end{corollary}

Unwinding definitions gives the following.

\begin{proposition} For any FI-module $\mathcal V$, $H_{-1}(\Inj_*(\mathcal V))_S \cong H^{\FI}_0(\mathcal V)_S$. \label{minusonehomology}
\end{proposition}

\subsection{Homology of the complex of injective words} \label{SectionInjHomology}

In the previous subsection, we computed the homology of the injective words chain complex of an FI$\sharp$-module in terms of the top homology group of the complex of injective words. We now will show this top homology group is a certain space of products of graded Lie polynomials, and compute a basis. 

Throughout this section we let $C^{(k)}_*$ denote the reduced cellular chains on the semi-simplicial space $\Inj_\bullet(k)$. In the language of the previous subsection, $C^{(k)}_*:=\Inj_*(M(0))_k$. For $q\geq -1$, the group $C^{(k)}_q$ is the free abelian group on words of $q+1$ distinct letters in $[k]$. By Theorem \ref{Farmer}, this chain complex has only one nonvanishing homology group, in homological degree $k-1$. 
\begin{definition}
Let $\Top_k := H_{k-1}(C^{(k)}_*)  \cong \tilde H_{k-1}(||\Inj_\bullet(k)||)$. 
\end{definition}
\noindent The symbol $\Top$ stands for ``top homology group."  Since $C^{(k)}_{k}=0$, the homology group $\Top_k$ is a submodule of $C^{(k)}_{k-1}$, the kernel of the differential: $$\cd := \sum_{j=0}^{k-1} (-1)^jd_j : C^{(k)}_{k-1} \to C^{(k)}_{k-2}$$ where $d_j$ is the face map that forgets the $j$th letter of each word. The top chain group  $C^{(k)}_{k-1}$ is naturally isomorphic to the regular representation $\Z[S_k]$, with a $\Z$--basis given by all injective words on $k$ letters in $[k]$.  The main objective of this section is to compute an alternate  $\Z$--basis for $C^{(k)}_{k-1}$ in the style of the Poincar\'e--Birkhoff--Witt theorem (Theorem \ref{InjectiveWordBasis}), and identify a sub-basis that spans the kernel of $\cd$ (Lemma \ref{LProductBasis} and Theorem \ref{TopHomology}). The result of this calculation is shown explicitly for $k=2,3,4$ in the Example \ref{ExampleTopHomology234}.

We adopt the following notational conventions. If $a$ is a word in the alphabet $[k]$, then in this section we write $|a|$ to mean the word-length of $a$. If $p$ is an integer linear combination of words, we call $p$ a \emph{(noncommutative) polynomial} in $[k]$, and define its degree $|p|$ to be  the length of the longest word occuring in $p$. Polynomials are assumed to be homogeneous unless otherwise stated. For words $a$ and $b$, we write $ab$ to denote their concatenation; this operation extends linearly to a multiplication on the additive group of polynomials in $[k]$.   A word is \emph{injective} if each letter appears at most once. We introduce a graded Lie bracket on polynomials in $[k]$.

\begin{definition} \label{DefnLieBracket} Define a graded Lie bracket on words in $[k]$  by $$[a,b]:= ab -(-1)^{|a||b|}ba$$ and extend bilinearly to a bracket on the free abelian group on words in $[k]$.
\end{definition}
On homogeneous polynomials $a,b,c$, the Lie bracket satisfies the graded antisymmetry rule $$[a,b] =- (-1)^{|a||b|}[b,a] $$ and the graded Jacobi identity $$(-1)^{|a||c|}[a,[b,c]]+(-1)^{|a||b|}[b,[c,a]]+(-1)^{|b||c|}[c,[a,b]]=0.$$ 

\begin{definition}  A \emph{Lie polynomial} is any element of the smallest submodule of the free abelian group on words in $[k]$ that contains the elements of $[k]$ and is closed under the Lie bracket. 
\end{definition}

The space of Lie polynomials is isomorphic to the \emph{free Lie superalgebra on $[k]$}. This space naturally embeds into the free abelian group of words on $[k]$, which, by a graded-commutative version of the Poincar\'e--Birkhoff--Witt Theorem,  we can identify with its universal enveloping algebra. The following result appears in Ross \cite[Theorem 2.1]{RossLieSuperalgebras}; see also Musson \cite[Theorem 6.1.1]{MussonLieSuperalgebras}.

\begin{theorem}[{See, eg, Ross \cite[Theorem 2.1]{RossLieSuperalgebras}}]  \label{GradedPBW}
Let $R$ be a commutative ring with unit such that $2$ is invertible. Let $L$ be a homogeneously free Lie superalgebra over $R$ with homogeneous bases $X_0$ for its even-graded part and  $X_1$ for its odd-graded part. If $\leq$ is a total order on $X=X_0\cup X_1$, then the set of monomials of the form $$ b_1 b_2 \cdots b_m \qquad \text{with $b_i \in X$, $b_i \leq b_{i+1}$, and $b_i \neq b_{i+1}$ if $b_i \in X_1$} $$ and $1$ form a free $R$--basis for the universal enveloping algebra $U(L)$. 
\end{theorem}

We remark that this set of monomials is not a basis when $R$ is $\Z$. In the case of the free Lie superalgebra on $[k]$, this failure is in some sense due to factors of two that appear with (nested) brackets involving repeated letters, for example, $[1,1]=11+11$. Fortunately for our purposes, we will show in Theorem \ref{InjectiveWordBasis} that those basis elements for which every letter is distinct \emph{do} form an integer basis for $C_{k-1}^{(k)}$. The following example illustrates the main result of this subsection, the bases for $C_{k-1}^{(k)}$ and the top homology group, for small $k$. 

\begin{example} \label{ExampleTopHomology234} When $k=2, 3,$ or $4$, Theorems \ref{InjectiveWordBasis} and \ref{TopHomology} give the following $\Z$-bases for the chain group $C^{(k)}_{k-1}$, and the top homology group $H_{k-1}(C^{(k)}_*)$. (Here we have taken the graded lexicographical ordering on the set $B$ of Theorem \ref{InjectiveWordBasis}). 

The $\Z$-basis for the rank-$2$ group $C^{(2)}_{1}$  is $ \{ [1,2], 12 \} $ and $H_{1}(C^{(2)}_*)$ is the rank-one subgroup spanned by $[1,2]=12+21$. This is the trivial $S_2$-representation. 

The basis for $C^{(3)}_{2}$  is $$ [[1,2],3], \;  [[1,3],2], \quad 1[2,3], \;  2[1,3], \; 3[1,2], \quad 123,$$ 
and $H_{2}(C^{(3)}_*)$ is the rank--two subgroup spanned by $$[[1,2],3] =  123+213- 312-321, \qquad [[1,3],2] = 132+312-213-231$$ isomorphic to the standard $S_3$-representation. 

The basis for $C^{(4)}_{3}$  is 
\begin{align*}
& [[[1,2],3],4], \; [[[1,2],4],3], \; [[[1,3],2],4],\; [[[1,3],4],2], \; [[[1,4],2],3],\; [[[1,4],3],2], \\
 & [1,2][3,4], \;  [1,3][2,4], \;   [1,4][2,3], \\
 &   1[[2,3],4], \; 1[[2,4], 3]], \; 2[[1,3], 4], \; 2[[1,4], 3], \; 3[[1,2], 4],\; 3[[1,4], 2],\;  4[[1,2], 3],\; 4[[1,3], 2],\; \\ 
& 12[3,4], \; 13[2,4], \; 14[2,3], \; 23[1,4], \; 24[1,3], \; 34[1,2], \; \\ & 1234. 
\end{align*} 
The top homology group $H_{3}(C^{(4)}_*)$ is the rank-nine free abelian group on the elements 
\begin{align*}
& [[[1,2],3],4], \; [[[1,2],4],3], \; [[[1,3],2],4],\; [[[1,3],4],2], \; [[[1,4],2],3],\; [[[1,4],3],2], \\ 
 & [1,2][3,4], \;  [1,3][2,4], \;   [1,4][2,3]. \;  
\end{align*} 

In general, the homology group will consist of all the basis elements that consist of a product of brackets, that is, the basis elements that contain no singleton factors. 
\end{example}

We now introduce notation for the free Lie superalgebra which we will  view as a submodule of $C^{(k)}_{k-1}$.

\begin{definition} For a finite set $S$ with $|S|\geq 2$, let $\cL_S$ denote the subset of homogeneous degree-$|S|$ Lie polynomials whose terms are all injective words in $S$. We write $\cL_k$ when $S=[k]$. It is spanned by $(k-1)$-fold iterated brackets such that each letter in $[k]$ appears exactly once.  We define $\cL_S=0$ if $S$ has one or zero elements. 
\end{definition}

For example, $\cL_2 \cong \Top_2$ is the rank-$1$ abelian group with basis $[1,2] = 12+21$, $\cL_3 \cong \Top_3$ is the rank-$2$ abelian group spanned by the elements  $[1,[2,3]],  [2,[1,3]]$ and $[3,[1,2]]$, which (by the Jacobi identity) sum to zero. The group $\cL_4 \subsetneq \Top_4$ is the rank-$6$ abelian group spanned by the Lie polynomials 
$$[[[1,2],3],4], \; [[[1,2],4],3], \; [[[1,3],2],4],\; [[[1,3],4],2], \; [[[1,4],2],3],\; [[[1,4],3],2].$$ We give a basis for $\cL_k$ using a graded-commutative variation on an argument appearing in Reutenauer \cite[Section 5.6.2]{ReutenauerFreeLieAlgebras}.

\begin{theorem}[{Compare to Reutenauer \cite[Section 5.6.2]{ReutenauerFreeLieAlgebras}}]\label{ReutenauersBasis} The abelian group $\cL_k$ is free of rank $(k-1)!$ with a $\Z$-basis all elements of the form $$ [[[ \cdots[1, a_2], a_3], \ldots], a_{k-1}], a_k ] \qquad \text{ for any ordering $(a_2, a_3, \ldots, a_k)$ of the set $\{2, 3, \ldots k\}$. }$$ 
\end{theorem}

More generally, for $S \subseteq [k]$, we define the Reutenauer basis for $\cL_S$ to be the $(|S|-1)!$ elements as above with the letter $1$ replaced by the smallest element of $S$ under the natural ordering on $[k]$. 

\begin{proof} As in Reutenauer's proof, we may inductively apply the antisymmetry and Jacobi relations 
$$[1,[P,Q]]=(-1)^{|Q||P|+1}[[1,Q], P]+[[1,P],Q] $$ 
 to write any element in $\cL_k$ as a linear combination of these generators. The generators must be linearly independent over $\Z$, since  $[[[ \cdots[1, a_2], a_3], \ldots], a_{k-1}], a_k ]$ is the only Lie polynomial in the list whose expansion includes the word $1a_2a_3 \ldots a_k$. We note that this last observation also implies that these elements span a direct summand of $C^{(k)}_{k-1}$, and not a higher-index subgroup of a direct summand. 
\end{proof}

\begin{corollary} \label{LieNGeneratingFunction} The exponential generating function for the sequence $\ell_k :=$ rank$(\cL_k)$ is 
\begin{align*} L(x) &= -\log(1-x)-x \\
&= (1!)\frac{x^2}{2!} + (2!)\frac{x^3}{3!} + (3!)\frac{x^4}{4!}+ (4!)\frac{x^5}{5!} + \cdots. \end{align*}
\end{corollary}

In the spirit of the PBW theorem, we will now construct a new basis for the free $\Z$-module $C^{(k)}_{k-1}$ using the bases defined in Theorem \ref{ReutenauersBasis}. Our eventual goal is to prove that a certain subset of this basis spans the top homology group of the complex of injective words. 

\begin{theorem} \label{InjectiveWordBasis} 
Fix $k \geq 2$. For each subset $S\subseteq [k]$ with $|S| \geq 2$, let $B_S$ be the basis of $\cL_S$ of Theorem \ref{ReutenauersBasis}. For each singleton subset $S=\{a \} \subset [k]$, let $B_S = \{a\}$.  Put a total order $\leq$ on $\displaystyle B = \sqcup_{{S\subseteq [k]}} B_S.$ Then the set $\Pi$  of polynomials of the form 
$$P_1 P_2 \cdots P_m \qquad \text{ such that $[k]=S_1 \sqcup S_2 \sqcup \cdots \sqcup S_m$, $P_i \in B_{S_i}$, and  $P_1 < P_2 < \ldots < P_m \in B$}$$
is a $\Z$--basis for  $C^{(k)}_{k-1}$. 
\end{theorem}

\begin{proof}  The set $\Pi$ is a subset of the basis given in Ross \cite[Theorem 2.311]{RossLieSuperalgebras}; the elements of $\Pi$  are linearly independent over $\Z[\frac12]$ and therefore over $\Z$. We must show that they span $C^{(k)}_{k-1}$. Assume without loss of generality that $1 < 2 < \ldots < k$ in our total order on $B$. Observe that the one element of $\Pi$ associated with the decomposition $[k] = \{1\} \sqcup \{2\} \sqcup \cdots \sqcup \{k\}$ is the single word $P=P_1 P_2 \cdots P_k = 1 2 3 \cdots k$. We wish to show all permutations of this word are also contained in the span of  $\Pi$. We proceed by induction. 

 Let $\Pi_m \subseteq \Pi$ be the subset of polynomials in $\Pi$ associated to a decomposition $[k]=S_1 \sqcup S_2 \sqcup \cdots \sqcup S_q$ with $q \leq m$. We prove by induction on $m$ that elements in the subset $\Pi_m$ span the space of all products of elements of $B$  (in any order) with $m$ or fewer factors. This is trivial when $m=1$; suppose $m>1$. Observe that, given a polynomial $P=P_1 P_2 \cdots P_m \in \Pi$ and a transposition $(i\; i+1) \in S_m$, we have: 
\begin{equation} \big(P_1 P_2 \cdots P_{i+1}P_i \cdots P_m \big)  = (-1)^{|P_i||P_{i+1}|} \left( \big(P_1 P_2 \cdots P_iP_{i+1} \cdots P_m \big) -\big( P_1 P_2 \cdots [P_i, P_{i+1}] \cdots P_m\big)\right). \label{EqnTransposition} \end{equation}
We may re-express $[P_i, P_{i+1}]$ as a linear combination of Reutenauer basis elements for  $\cL_{S_i \cup S_{i+1}}$, and by induction $\big(P_1 P_2 \cdots [P_i, P_{i+1}] \cdots P_m\big)$ is in the span of polynomials in $\Pi_{m-1}$. Since transpositions of the form $(i\; i+1)$ generate $S_m$, this implies that all $S_m$--permutations of the factors of $P=P_1 P_2 \cdots P_m$ are in the span of $\Pi_m$, which concludes our induction. In particular,  when $m=k$ all permutations of our word $P= 123\cdots k$ of length $k$ are contained in the span of $\Pi_k = \Pi$, so $C^{(k)}_{k-1}$ is contained in the span of $\Pi$ as claimed. 
\end{proof}

 Our next goal is to identify $H_{k-1}(C^{(k)}_*) \subseteq C^{(k)}_{k-1}$. We will show that the top homology group is spanned by certain polynomials we call $\cL$--products. 

        \begin{definition}  We call an element  $P$ of $C^{(k)}_{k-1}$ an \emph{$\cL$--product} if it has the following form. For some partition of $[k]=S_1 \sqcup S_2 \sqcup \cdots \sqcup S_m$, we can decompose $P$ as a product: $$P = P_1 P_2 \cdots P_m \qquad \text{with $P_i \in \cL_{S_i}$.}$$
\end{definition}

Note that, in contrast to  the elements of the basis $\Pi$ in Lemma \ref{InjectiveWordBasis}, $\cL$--products exclude factors $P_i$ that are a single letter. For example, the polynomial $$[1,2][3,4] = (12+21)(34+43) = (1234+1243+2134+2143)$$ is an $\cL$--product in $C^{(4)}_3$, but 
$$[1,[2,3]]4 = (1(23+32)-(23+32)1)4 = (1234+1324-2314-3214)$$ is  \emph{not} an $\cL$--product. The following proposition shows that all $\cL$--products are in the kernel of the differential $\cd$. 

\begin{proposition} Any  $\cL$--product in $C^{(k)}_{k-1}$ is a cycle. \label{PropLProductCycles}
\end{proposition}

Since the homology group $H_{k-1}(C^{(k)}_*)$ is the subgroup of cycles in $C^{(k)}_{k-1}$, we may view elements in the span of the $\cL$--products as homology classes. 

\begin{proof}[Proof of Proposition \ref{PropLProductCycles}] We will verify that elements of $\cL_k$ are contained in $\ker(\cd)$. 
Since the differential $\cd$ satisfies the Leibniz rule  on elements of $C^{(k)}_{k-1}$ $$\cd(ac) = \cd(a)c+(-1)^{|a|}a\cd(c),$$
it follows that products of these Lie polynomials are in the kernel of $\cd$. We will proceed by induction on $k$. When $k=2$  we have $\cL_2 = \Z[1,2]$ and
$$ \cd([1,2]) = \cd(12+21) = 2-1+1-2 = 0.$$ Now fix $k$ and suppose that any Lie polynomial of degree less than $k$ is mapped to zero by $\cd$. To show that $\cL_k \subseteq \ker(\cd)$, it suffices to check Lie polynomials of the form $[P, a_k]$ of Reutenauer's basis (Theorem \ref{ReutenauersBasis}). We have: 
\begin{align*}
&\cd([ P, a_k]) = \cd\Big( Pa_k - (-1)^{|P|} a_kP \Big) \\
& = \Big(\cd(P)a_k +(-1)^{|P|}P\cd(a_k)\Big)-(-1)^{|P|}\Big(\cd(a_k)P - a_k\cd(P) \Big) \\ 
&=  0+(-1)^{|P|}P\cd(a_k) -(-1)^{|P|}\cd(a_k)P + 0\qquad \qquad \qquad  \text{since $\cd P=0$ by the inductive hypothesis,} \\ 
&= (-1)^{|P|}( P  - P) \\
&=0.
\end{align*}
Thus the Lie polynomials in $\cL_k$ and their products are cycles, as claimed. 
\end{proof}

The next result gives a basis for the subgroup of $C^{(k)}_{k-1}$ spanned by $\cL$--products. Theorem \ref{TopHomology} will then show us that this subgroup is, in fact, precisely the top homology group $\T_k$. 

\begin{lemma} \label{LProductBasis} 
Fix a finite set $[k]$ with $k \geq 2$. As in Theorem \ref{InjectiveWordBasis}, for each subset $S\subseteq [k]$, let $B_S$ be the basis of $\cL_S$ of Theorem \ref{ReutenauersBasis}. Put a total order $\leq$ on $B := \cup_{S\subseteq [k], |S| \geq 2} B_S$. 
The set $\Pi^*$  of polynomials of the form 
$$P_1 P_2 \cdots P_m \qquad \text{ such that $[k]=S_1 \sqcup S_2 \sqcup \cdots \sqcup S_m$, $P_i \in B_{S_i}$, and  $P_1 < P_2 < \ldots < P_m \in B$}$$
form a basis for the subgroup of $C^{(k)}_{k-1}$ spanned by $\cL$--products. Moreover, this subgroup is a direct summand of $C^{(k)}_{k-1}$. 
\end{lemma}

Note that, in contrast to Theorem \ref{InjectiveWordBasis}, our generating set $B$ excludes all words of length $1$. 

\begin{proof} Because $\Pi^*$ is a subset of the basis $\Pi$ for $C_{k-1}^{(k)}$ of Theorem \ref{InjectiveWordBasis}, the polynomials in $\Pi^*$  must be linearly independent, and their span must be a direct summand of $C_{k-1}^{(k)}$.  Each polynomial in $\Pi^*$ is an $\cL$--product, so it remains to show that they span. As in the proof of Theorem \ref{InjectiveWordBasis}, we need to show that any permutation of the factors of an element $P_1 P_2 \cdots P_m$ of $\Pi^*$ is in the span of $\Pi^*$, and we may use the same induction argument from Theorem \ref{InjectiveWordBasis}. 
Again let $\Pi^*_m \subseteq \Pi^*$ be the subset of polynomials of $\Pi$ with at most $m$ factors; we prove by induction that $\Pi^*_m$ spans the space of $\cL$--products with $m$ or fewer factors. When $m=1$, the polynomials $P_1$ are precisely the elements in Reutenauer's basis for $\cL_k$ (Theorem \ref{ReutenauersBasis}). For $m>1$,  Equation (\ref{EqnTransposition}) in the proof of Theorem \ref{InjectiveWordBasis} again completes the inductive step, which concludes our proof. 
\end{proof}

To prove that the subgroup of the chains $C^{(k)}_{k-1}$ given in Lemma \ref{LProductBasis} is in fact the entire top homology group, we will compare their ranks. We now use an Euler characteristic argument to compute the rank of $\T_k$. 

\begin{proposition} \label{RankTopHomology} 
The top homology group of the complex of injective words is a free abelian group with rank
$$ \mathrm{rank} \T_k= \frac{k!}{0!} - \frac{k!}{1!} + \frac{k!}{2!} - \frac{k!}{3!} + \cdots + (-1)^{k-2} \frac{k!}{(k-2)!} + (-1)^{k-1} \frac{k!}{(k-1)!} + (-1)^{k} \frac{k!}{k!}.$$ 
The exponential generating function for the ranks of these groups is: $$H(x)=\frac{e^{-x}}{1-x}.$$ 
\end{proposition}

\begin{proof}
Since the group $C^{(k)}_q$ has rank $[\fS_k:\fS_{k-q-1}] = \frac{k!}{(k-q-1)!}$, the Euler characteristic of the chain complex  $C^{(k)}_*$ is:  
$$ \chi =  -\frac{k!}{k!} + \frac{k!}{(k-1)!} - \frac{k!}{(k-2)!} + \ldots + (-1)^{k-3} \frac{k!}{2!} + (-1)^{k-2} \frac{k!}{1!} + (-1)^{k-1} \frac{k!}{0!} . $$ 
Farmer's results imply that  the homology of the complex $C^{(k)}_*$ is a free abelian group concentrated in degree $(k-1)$; see Theorem \ref{Farmer}.  It follows that its Euler characteristic is $(-1)^{k-1} h_k$. Thus,

$$ h_k =  \frac{k!}{0!} - \frac{k!}{1!} + \frac{k!}{2!} - \cdots + (-1)^{k-2} \frac{k!}{(k-2)!} +  (-1)^{k-1} \frac{k!}{(k-1)!} + (-1)^{k} \frac{k!}{k!}.$$ 

\noindent By inspection, the sequence $\{h_k\}$ satisfies the relation $h_k = k h_{k-1} + (-1)^k$ for $k\geq 1$. 
  
Since $h_0=1$, we infer that its exponential generating function $H(x)$ satisfies the relation: \begin{align*}
H(x) -1 & = \sum_{k \geq 1} h_k \frac{x^k}{k!} \\
&= \sum_{k \geq 1}  k h_{k-1} \frac{x^k}{k!} + \sum_{k \geq 1}  (-1)^k \frac{x^k}{k!}  \\ 
&=  x H(x)+e^{-x}- 1.
\end{align*} Solving for $H(x)$ gives:
\begin{align*} H(x) &=\frac{e^{-x}}{1-x} \\
&= 1 + (0) \frac{x}{1!}  + (1)\frac{x^2}{2!}  + (2)\frac{ x^3}{3!} + (9)\frac{x^4}{4!}  + (44)\frac{x^5}{5!}  + (265)\frac{x^6}{6!} + (1854)\frac{x^7}{7!}  + \cdots.  && \qedhere \end{align*}
\end{proof}

\begin{theorem} \label{TopHomology} 
$\T_k$ is equal to the subgroup of  $C^{(k)}_{k-1}$ spanned by $\cL$--products. It  has a $\Z$--module basis given by Lemma \ref{LProductBasis}.
\end{theorem}

\begin{proof}
Because the subgroup spanned by the $\cL$--products is a direct summand of $C^{(k)}_{k-1}$ by Lemma \ref{LProductBasis}, to prove the theorem it is enough to prove that its rank is equal to the rank of $H_{k-1}(C^{(k)}_*)$. Recall for $k \geq 2$ the basis given in Lemma \ref{LProductBasis},
$$P_1 P_2 \cdots P_m \qquad \text{ such that $[k]=S_1 \sqcup S_2 \sqcup \cdots \sqcup S_m$, $P_i \in B_{S_i}$, and  $P_1 < P_2 < \ldots < P_m \in B$}.$$
In Theorem \ref{ReutenauersBasis} we saw that $|B_{a}|$ has order $\ell_a=(a-1)!$. The number of ways to decompose $[k]$ into subsets of orders $a_1, a_2, \ldots, a_m$ is ${ k \choose a_1, a_2, \ldots, a_m}$, and the number of products of Reutenauer basis elements for these subsets (where factors can appear in any order) is ${ k \choose a_1, a_2, \ldots, a_m} \ell_{a_1} \ell_{a_2} \cdots \ell_{a_m}$. The number of products with factors in ascending order is $\frac{1}{m!} { k \choose a_1, a_2, \ldots, a_m} \ell_{a_1} \ell_{a_2} \cdots \ell_{a_m}$. Hence the basis $\Pi^*$ for the space of $\cL$--products given in Lemma \ref{LProductBasis} has  cardinality: 
$$ \ell_k + \frac1{2!} \sum_{a+b=k} {k \choose a,b} \ell_a \ell_b + \frac1{3!} \sum_{a+b+c=k} {k \choose a,b,c} \ell_a \ell_b  \ell_c + \frac1{4!} \sum_{a+b+c+d=k} {k \choose a,b,c, d} \ell_a \ell_b  \ell_c \ell_d + \cdots $$ 
This implies that the exponential generating function for the rank of this space is given by exponentiating the generating function $L(x) = -\log(1-x)-x$ for $\ell_k$ found in Corollary \ref{LieNGeneratingFunction}. But 
$$ e^{L(x)} = \frac{e^{-x}}{1-x} = H(x),$$ where $H(x)$ is the exponential generating function found in Proposition \ref{RankTopHomology}, and so we conclude that for $k\geq 1$ the cardinality of the basis $\Pi^*$ is equal to the rank of $H_{k-1}(C^{(k)}_*)$. Hence $H_{k-1}(C^{(k)}_*)$ is equal to the subgroup of  $C^{(k)}_{k-1}$ spanned by $\cL$--products. 
\end{proof}

\begin{remark} \label{RemarkDerangements}
We remark that Theorem \ref{TopHomology} and the basis for $\Top_k$ given in Lemma \ref{LProductBasis} make it apparent that the rank of $\Top_k$ will be equal to the number of derangements of $\fS_k$. The Reutenauer basis for $\Lie_S$, $|S|=k$ of Theorem \ref{ReutenauersBasis} are the $(k-1)!$ elements $\{ [[[ \cdots[a, a_2], a_3], \ldots], a_{k-1}], a_k ] \}$ where $a$ denotes the lexicographically first element of $S$ and  all permutations of the remaining elements $a_i$ of $S$ appear. Then the map $$[[[ \cdots[a, a_2], a_3], \ldots], a_{k-1}], a_k ] \longmapsto (a \, a_2 \, a_3 \, \cdots \, a_k) $$ identifies the Reuntenauer basis elements with the set of $k$-cycles on $S$. Extending this map to the basis in Lemma \ref{LProductBasis} identifies each basis element with a permutation without 1-cycles, written in cycle notation, with cycles ordered lexicographically.  We have a naturally defined bijection between our basis for $\Top_k$ and the set of derangements on $k$ letters. This bijection, however, is not $\fS_k$-equivariant. \end{remark}

\subsection{Secondary injective word complexes} \label{subsecsecworcomplex} 
 
 In this subsection, we define a chain complex called the secondary injective words chain complex of a $\bigwedge\,( \Sym^2 R)$-module. This chain complex should be thought of as a central stability complex for $\bigwedge \,(\Sym^2 R)$-modules and this complex will appear on the $E^2$-page of a certain spectral sequence. 
 
Recall that if $A=\{a,b\}$ is a 2-element set, then $\Lie_A \cong R$ is the free $R$-module on the graded Lie bracket $[a,b]=ab+ba$, that is, $\Lie_A$ is a rank-one trivial $\fS_2$-representation.

\begin{definition}
Let $\mathcal V$ be a $\FIM^{+}$-module, $S$ a set of cardinality $k$. Let 
\begin{align*} \Inj^2_p(\mathcal V)_S & : =\Ind^{\fS_k}_{(\fS_2)^{p+1}  \times \fS_{k-2p-2}} \Lie_2^{ \boxtimes (p+1)} \boxtimes \mathcal V_{k-2p-2} \\ 
& = \bigoplus_{ \substack{\text{ordered partitions} \\  S=A_0 \sqcup A_1 \sqcup \cdots \sqcup A_{p} \sqcup B \\  |A_i|=2, \; |B|=k-2p-2}}  \Lie_{A_0} \otimes \Lie_{A_1} \otimes \cdots \otimes \Lie_{A_{p}} \otimes \mathcal V_B.
\end{align*}

\noindent These groups assemble to form a chain complex as follows. Define maps 
\begin{align*} d_i :  \Inj^2_p(\mathcal V)_S &\longrightarrow \Inj^2_{p-1}(\mathcal V)_S  \qquad \qquad  (i=0, \ldots, p) \\ 
\Lie_{A_0} \otimes \Lie_{A_1} \otimes \cdots \otimes \Lie_{A_{p}} \otimes \mathcal V_B &\longrightarrow  \Lie_{A_0} \otimes \Lie_{A_1} \otimes \cdots \otimes \widehat{\Lie_{A_i}} \otimes \cdots \otimes \Lie_{A_{p-1}} \otimes \mathcal V_{B \sqcup A_i}
\end{align*}
as follows: let $d_i$ act by the identity on the tensor factors $\Lie_{A_0}$, $\Lie_{A_1}, \ldots, \widehat{\Lie_{A_i}}, \ldots  \Lie_{A_{p-1}} $, and act on the factor $\mathcal V$ by the map $\mathcal V_B \to \mathcal V_{B \sqcup A_i}$ induced by the FIM$^+$ morphism associated to the inclusion ($B \hookrightarrow B \sqcup A_i$, $A_i$).

Because the composition of maps $\mathcal V_B \to \mathcal V_{B \sqcup A_i} \to \mathcal V_{B \sqcup A_i \sqcup A_j} $ is the negative of the composition  
$\mathcal V_B \to \mathcal V_{B \sqcup A_j} \to \mathcal V_{B \sqcup A_j \sqcup A_i}$, we must take the sum of the maps $d_i$ (instead of the alternating sum) to obtain a chain complex. Let $\Inj^2_*(\mathcal V)_S$ denote the chain complex with differentials given by the sum of the maps $d_i$. 
\end{definition}

\begin{proposition} \label{StructureInj2M(d)}
There is an isomorphism of chain complexes (which is non-equivariant with respect to the permutation group of $S$): $$\Theta: \bigoplus_{\substack{ f \in \Hom_{\FI}([d],S) \\ Z  \text{ a perfect matching on }S \setminus \im(f) }} \Inj_*(M(0))_Z \m \Inj^2_* \left( M^{\FIM^+}(d) \right)_S.$$ 
\end{proposition}

Proposition \ref{StructureInj2M(d)} and Theorem \ref{Farmer} imply that $\Inj^2_*( M^{\FIM^+}(d))_S$ is highly acyclic.

\begin{corollary} \label{free2connected}
The homology groups $H_i(\Inj^2_ *(M^{\FIM^+}(d)))_S \cong 0$ if $i \leq \left(\frac{|S|-d}{2}-2\right)$. 
\end{corollary}

\begin{proof}[Proof Proposition \ref{StructureInj2M(d)}] Let $k=|S|$; we may assume $k \equiv d$ (mod 2) or both chain complexes are zero. By an \emph{order} on a matching $\{ \{x_1,y_1\},\ldots, \{x_l,y_l\} \}$, we mean a bijection to $[l]$. Choose an order on every perfect matching of every $(k-d)$-element subset of $S$.  The map $\Theta$ that we will construct will not be equivariant with respect to the $\fS_k$-action and will depend on the choice these choices of orderings, but only up to sign. 

 Fix a homological degree $p$ and let $f \in \Hom_{\FI}([d],S)$. Let $Z$ be  a perfect matching on $S\setminus \im(f) $ and let $z$ be an injective word in $Z$ of length $(p+1)$. The data $(f,Z,z)$ specifies a generator of the domain of $\Theta$.  Let $\{a_j,b_j\}$ denote the 2-element set that is the $j$th letter of $z$. Let 
 $$z'= [a_0,b_0] \otimes \ldots \otimes [a_{p},b_{p}] \in \Lie_{\{a_{0},b_{0}\}} \otimes \ldots \otimes \Lie_{\{a_{p},b_{p}\}}.$$ Let $$w=\{ a_{p+1},b_{p+1} \} \wedge \ldots \wedge \{ a_{\frac{k-d-2}{2}},b_{\frac{k-d-2}{2}}\} $$ be an oriented perfect matching of the set of elements of $Z$ not appearing in $z$, written in an arbitrary order (see Definition \ref{DefFIM+}). Up to sign, this depends on choice of ordering on the set   $$\big \{ \{ a_{p+1},b_{p+1} \} , \ldots , \{ a_{\frac{k-d-2}{2}},b_{\frac{k-d-2}{2}}\} \big \}.$$  Our construction will ultimately be independent of this choice of order, in contrast to the choices of orders made in the first paragraph. Let 
 $Q=S \setminus \{a_0,b_0,\ldots , a_{p},b_{p} \}$ and let $F=(f,w)$. Using the isomorphism $$M^{\FIM^+}(d)_{Q} \cong \Hom_{\FIM^+}([d],Q)$$ described in Proposition \ref{altDescriptionOfFree}, we may view $F$ as an element of $M^{\FIM^+}(d)_{Q}$. Let $\sigma: Z \m Z$ be the permutation from our pre-selected order on $Z$ to the order $\{a_0,b_0\}, \ldots, \{a_{\frac{k-d-2}{2}},b_{\frac{k-d-2}{2}}\}$ defined by $(z', w)$. Define $\Theta$ via the formula $$\Theta(f,Z,z):=(-1)^{\frac{(p+1)(p)}{2}} \mr{sign}(\sigma) \left( z' \boxtimes F\right)$$ and extend linearly.  We will check that this map is well defined, an isomorphism of abelian groups and a map of chain complexes. 

To see that $\Theta$ is well defined, it suffices to check that it does not depend on the choice of order on the set $$\big \{ \{ a_{p+1},b_{p+1} \} , \ldots , \{ a_{\frac{k-d-2}{2}},b_{\frac{k-d-2}{2}}\} \big \}.$$  This is the reason that we included a $\mr{sign}(\sigma)$ term in the definition of $\Theta$; permuting two terms of the oriented matching $ \{ a_{p+1},b_{p+1} \} \wedge \ldots \wedge \{ a_{\frac{k-d-2}{2}},b_{\frac{k-d-2}{2}}\} $ gives a minus sign which exactly cancels the sign change in the $\mr{sign}(\sigma)$ term.

 We now check that this map is an isomorphism of abelian groups. Both $$ \bigoplus_{\substack{ f \in \Hom_{\FI}([d],S) \\ Z  \text{ a perfect matching on }S \setminus \im(f) }} \Inj_p(M(0))_Z \hspace{.2in} \text{   and   } \hspace{.2in} \Inj^2_p \left( M^{\FIM^+}(d) \right)_S$$ are isomorphic as abelian groups to the free abelian group on the set \[\{ (f,Z,z) \; | \; f:[d] \hookrightarrow S, Z \text{ a perfect matching on } S \setminus \im(f) , z \text{ an injective word of length }(p+1) \text{ in }Z    \} .\] This gives bases of the domain and the codomain of $\Theta$. Up to sign, $\Theta$ maps one basis to the other basis and so it is an isomorphism of abelian groups.

All that remains is to check that $\Theta$ is a map of chain complexes. The differential on the domain of $\Theta$ is an alternating sum of maps which we called $d_i$ and the differential on the codomain is a (non-alternating) sum of maps which we also called $d_i$. Thus, it suffices to check that $d_i \circ \Theta =(-1)^i \Theta \circ d_i$.  We will continue to use the notation of the second paragraph. We have  $$d_i(f,Z,z) =\Big (f,Z, \{a_0,b_0 \} , \cdots, \widehat{ \{a_i,b_i \} }, \cdots,  \{a_p,b_p \}  \Big ) \in \Inj_{p-1}(M(0))_Z .$$ Let 
 $$z'_i= [a_0,b_0] \otimes \ldots \otimes \widehat{[a_i,b_i]} \otimes \ldots \otimes [a_{p+1},b_{p+1}] \in \Lie_{\{a_{0},b_{0}\}} \otimes \ldots \otimes \widehat{\Lie_{ \{ a_i,b_i \} }  } \otimes \ldots \otimes \Lie_{\{a_{p},b_{p}\}},$$ $$ w_i= \{a_i,b_i\}  \wedge \{ a_{p+1},b_{p+1} \} \wedge \ldots \wedge \{ a_{\frac{k-d-2}{2}},b_{\frac{k-d-2}{2}}\} $$ and $F_i=(f,w_i) \in M^{\FIM^+}_{Q \cup \{a_i,b_i \}}$. Let $\sigma_i: Z \m Z$ be the permutation from our pre-selected order on $Z$ to the order $$\{a_0,b_0\}, \ldots,\widehat{\{ a_i,b_i \}  } ,  \ldots, \{ a_p,b_p \}, \{a_i,b_i \} , \{a_{p+1}, b_{p+1} \} , \ldots ,    \{a_{\frac{k-d-2}{2}},b_{\frac{k-d-2}{2}}\}.$$ We have  $$\Theta(d_i(f,Z,z))=(-1)^{\frac{(p)(p-1)}{2}} \mr{sign}(\sigma_i) \left( z_i' \boxtimes F_i\right).$$ In contrast, $$ d_i(\Theta(f,Z,z)) =(-1)^{\frac{(p+1)(p)}{2}} \mr{sign}(\sigma) \left( z_i' \boxtimes F_i\right).$$  Since $ \sigma_i$ can be obtained from $\sigma$ by composing with $(p-i)$ simple transpositions, we see that $\mr{sign}(\sigma_i)=(-1)^{p-i} \mr{sign}(\sigma)$, and the claim follows. 
\end{proof}

Using Theorem \ref{NSSnoth}, will prove a vanishing result for $H_*(\Inj^2_*(\mathcal V))$ for $\mathcal V$ finitely generated and $R$ a field of characteristic zero.

\begin{proposition} \label{connected2}
Let $R$ be a field of characteristic zero and let $\mathcal V$ be a finitely generated $\bigwedge\,(\Sym^2 R)$-module. For each $p$, there is a number $N^{\mathcal V}_p$ such that if $|S|>N^{\mathcal V}_p$, the homology group $H_p(\Inj^2_*(\mathcal V))_S $ vanishes. 
\end{proposition}

\begin{proof}

First we will use Theorem \ref{NSSnoth} to construct a free resolution of $\mathcal V$. That is, we will show there are integers $d_i$,$e_i$, $m_{ij}$ and maps making the following an exact sequence of $\bigwedge \,(\Sym^2 R)$-modules: 
$$\ldots  \m \bigoplus_{j=d_1}^{e_1} (M^{\FIM^+}(j))^{m_{1j}} \m \bigoplus_{j=d_0}^{e_0} (M^{\FIM^+}(j))^{m_{0j}} \m \mathcal V   \m 0. $$ Suppose for the purposes of induction that we have constructed the first $k$ stages of a resolution by modules of the form $\bigoplus_{j=d_i}^{e_i} (M^{\FIM^+}(j))^{m_{ij}}$. The kernels of the last map is a submodule of a finitely generated FIM$^+$-module. Hence it is finitely generated, and so there exists a surjection onto it from an FIM$^+$-module of this form. Using this map, we construct the next term in the sequence.

Let $C_*$ be the chain complex obtained by replacing $\mathcal V$ in the above sequence with $0$. Note that the functor $\mathcal W \mapsto \Inj^2_p( \mathcal W)$ is exact for all $p$. Consider the double complex spectral sequences associated to the double complex $\Inj^2_*(C_*)$. One spectral sequence has: $$E^2_{p,q}=H_p(\Inj^2_*(H_q(C_*))).$$ Since $H_q(C_*)$ vanishes for $q>0$, this spectral sequence collapses on the second page. Since $H_0(C_*)=\mathcal V$, this spectral sequence converges to $H_p(\Inj^2_*(\mathcal V))$. The other spectral sequence has: $$'E^1_{p,q}=H_p(\Inj^2_*(C_q)).$$ Since $$C_q=\bigoplus_{j=d_q}^{e_q} (M^{\FIM^+}(j))^{m_{qj}},$$ Corollary \ref{free2connected} implies that $'E^1_{p,q}(S)$ vanishes in range increasing with the size of $S$. Thus, this spectral sequence converges to zero in a range increasing with the size of $S$. This implies that $H_p(\Inj^2_*(\mathcal V))_S \cong 0$ for $S$ sufficiently large compared with $p$. 
\end{proof}

The following corollary shows that Theorem \ref{maintheorem} implies Corollary \ref{corSecondaryCentral}.

\begin{corollary}
Let $R$ be a field of characteristic zero and $\mathcal V$ a finitely generated $\bigwedge \,(\Sym^2 R)$-module. For $k$ sufficiently large,  $\mathcal V_k$ is isomorphic to the quotient of  $ \Ind_{\fS_{k-2} \times \fS_2 }^{\fS_k} \mathcal V_{k-2}$  by the image of the sum of the two natural maps \[ \Ind_{\fS_{k-4} \times \fS_2 \times \fS_2 }^{\fS_k} \mathcal V_{k-4} \rightrightarrows \Ind_{\fS_{k-2} \times \fS_2 }^{\fS_k} \mathcal V_{k-2}\]
\end{corollary}

\begin{proof}
This statement is exactly the condition that $$H_{0}(\Inj^2_*(\mathcal V))_k \cong H_{-1}(\Inj^2_*(\mathcal V))_k \cong 0.$$ This is true for large $k$ by Proposition \ref{connected2}.
\end{proof}

\section{Configuration spaces} \label{SectionConfigurationSpaces}

In this section, we apply the tools of the previous section to prove secondary representation stability. We begin by recalling the definition of configuration spaces and their stabilization maps in Section \ref{secStabmap}. Then, in Section \ref{SectionArcRes}, we define the \emph{arc resolution} and an associated spectral sequence, which we use to prove representation stability for configuration spaces of (possibly nonorientable) manifolds. In Section \ref{SectionDifferentials}, we compute some differentials in this spectral sequence, and use this calculation to prove secondary representation stability for configuration spaces of surfaces in Section \ref{SectionMainProof}, as well as an improved range for representation stability for configuration spaces of high-dimensional manifolds in Section \ref{SectionHigherDimensions}. In Section \ref{secConj}, we give some computations for specific manifolds, and some conjectures. For simplicity, we will assume that $M$ is a smooth manifold, although all results are true for general topological manifolds. See Remark \ref{topologicalmanifolds} for a discussion of the necessary modifications needed to address non-smoothable manifolds.

\subsection{Stabilization maps and homology operations}
\label{secStabmap}

In this subsection, we define stabilization maps and an FI$\sharp$-module structure on the homology of the configuration spaces of a noncompact manifold.  Throughout the section $M$ will always denote a connected manifold of dimension $n \geq 2$.  Manifolds in this paper are assumed to be without boundary, unless otherwise stated.

\begin{definition}
For $M$ and $N$ smooth manifolds possibly with boundary, let $\Emb(N,M)$ denote the space of smooth embeddings topologized with the $C^\infty$ topology. 
\end{definition}

\begin{definition} Given a finite set $S$, let $F_S(M):=\Emb(S,M)$. We write $F_k(M)$ for $F_{[k]}(M)$. Let $C_k(M)$ denote the quotient of $F_k(M)$ by the action of $\mathfrak S_k = \mathrm{Aut}([k])$. The space $$F_k(M) \cong \{(m_1,\ldots,m_k) \in M^k\, |\, m_i \neq m_j \text{ for } i\neq j \}$$ is the configuration space of $k$ ordered points in $M$, and the space $C_k(M)$ is the configuration space of $k$ unordered points in $M$.

\end{definition}

Given an embedding of smooth manifolds $N \sqcup L \m M$ and sets $S$ and $T$, we get a map of spaces $$F_{S}(N) \times F_{T}(L) \m F_{S \sqcup T}(M).$$ Recall $n:=\dim(M)$. If $M$ is not compact, there exists a smooth embedding $e: \R^n \sqcup M \hookrightarrow M$ with $e|_{M}$ isotopic to the identity, as described in Section \ref{intro} (see Figure \ref{MEmbedding}). We fix such an embedding for the duration of this paper.  With this embedding we define the following maps on the homology of configuration spaces. 

\begin{definition} Let $M$ be a noncompact smooth manifold.
Given a class $\alpha \in H_{i}(F_S(\R^n))$, let $$t_\alpha:H_*(F_T(M)) \longrightarrow H_{*+i}(F_{T \sqcup S}(M))$$ be the map on homology induced by the embedding $e: \R^n \sqcup M \hookrightarrow M$. 

\end{definition}

The sequence of $\fS_k$-representations $H_i(F_k(M))$ assemble to form an FI-module as follows.  For a set $P$, let $[P]$ be the class of a point in $H_0(F_{P}(\R^n))$.  Let $f:S \m T$ be an injective map of finite sets.  The FI-module structure on $H_i(F(M))$ is defined so that the map $f$ is sent to the composition of the map induced by the diffeomorphism $F_S(M) \m F_{f(S)}(M)$ and $t_{[T\setminus f(S)]}$. See Figure \ref{ModuleStructure} for an illustration. This FI-module structure on homology arises from a homotopy-FI-space structure on the functor $S \mapsto F_S(M)$. 

The configuration spaces of $M$ also admit a co-FI-space structure defined as follows. View $F_S(M)$ as the spaces of embeddings $\Emb(S,M)$ and let injections act by precomposition, as in Figure \ref{CoModuleStructure}.
\begin{figure}[!ht]    \centering
\labellist
\Large \hair 0pt
\pinlabel {\fontsize{20}{40} $f: S \hookrightarrow T$} [c] at 25 95
\pinlabel {\fontsize{20}{40} $\tilde{f}$} [c] at 268 75
\endlabellist
\begin{center}\scalebox{.5}{\includegraphics{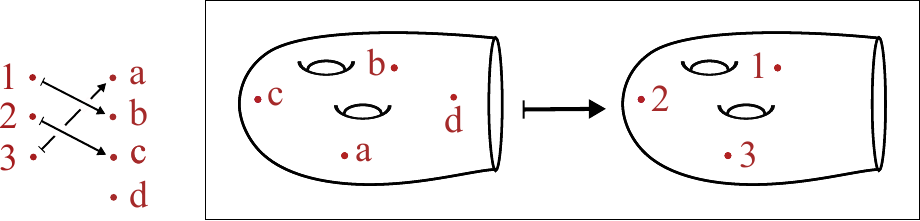}}\end{center}
\caption{The co-FI-space structure on $F(M)$.}
\label{CoModuleStructure}
\end{figure}  
The induced co-FI-module structure on $H_i(F(M))$ is compatible with the FI-module structure in such a way as to give $H_i(F(M))$ the structure of an FI$\sharp$-module. Church--Ellenberg--Farb describe this structure in detail \cite[Section 6]{CEF}.

Generalizing the construction of the stabilization map, for smooth manifolds $N, L, M$ there is a natural map: $$\Emb(N \sqcup L,M) \times F_{S}(N) \times F_{T}(L) \m F_{S \sqcup T}(M).$$ Define a map $\theta:S^{n-1} \m \Emb(\R^n \sqcup \R^n, \R^n)$ as follows: Let $r:\R^n \m \R^n$ be a map which induces an orientation preserving homeomorphism between $\R^n$ and the open unit ball around the origin. View $S^{n-1}$ as the unit vectors in $\R^n$ and let $\theta(\vec v): \R^n \sqcup \R^n \m \R^n$ be the function mapping $\vec x$ in the first copy of $\R^n$ to $r(\vec x)+\vec v$ and mapping $\vec x$ in the second copy of $\R^n$ to $r(\vec x)-\vec v$. By restricting to the class of a point in $H_0(S^{n-1})$, this induces a product on the homology of ordered configuration space of $\R^n$: 
$$\bullet: H_i(F_S(\R^n)) \otimes H_j(F_T(\R^n)) \m H_{i+j}(F_{S \sqcup T}(\R^n)).$$ \noindent
By restricting to a fundamental class of $S^{n-1}$, this induces a bracket: 
$$\psi^n: H_i(F_S(\R^n)) \otimes H_j(F_T(\R^n)) \m H_{i+j+n-1}(F_{S \sqcup T}(\R^n)).$$ 
The map $\psi^n$ can be thought of as a version of the Browder operation for $E_n$-algebras in the category of $\FB$-spaces. See May \cite[Definition 4.1]{M} for the definition of $E_n$-algebras and Browder \cite[Page 351]{Br} for the definition of Browder operations. The operations $\bullet$ and $\psi^n$ are illustrated in Figure \ref{HomologyOperations}. 
\begin{figure}[!ht]    \centering
\begin{subfigure}{.4\textwidth}
\labellist
\Large \hair 0pt
\pinlabel {\color{Fuchsia} $\alpha$} at  98 148
\pinlabel {\color{Fuchsia} $\beta$} at  245 110
\pinlabel {\color{black} $\R^2$} at  410 135
\endlabellist
  \centering
{\includegraphics[scale=.13]{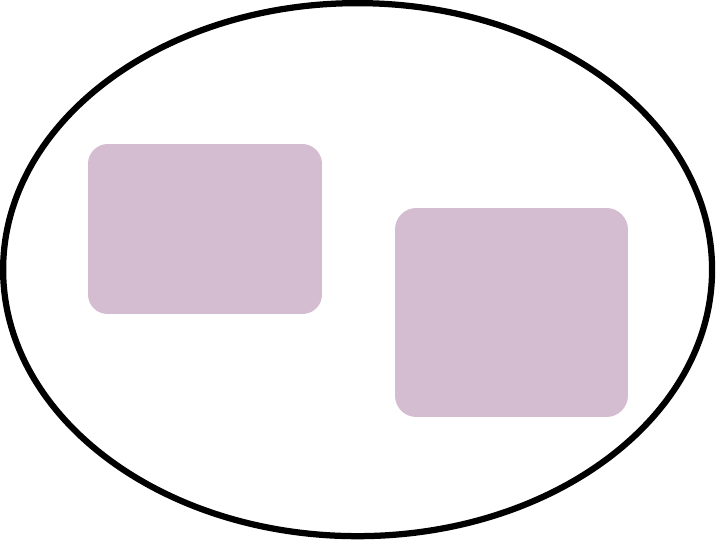}}
 \caption{The homology class $\alpha \bullet \beta$.}
\end{subfigure}%
\begin{subfigure}{.4\textwidth}
\labellist
\Large \hair 0pt
\pinlabel {\color{Fuchsia} $\beta$} at 248 126
\pinlabel {\color{Fuchsia} $\alpha$} at 148 230 
\pinlabel {\color{black} $\R^2$} at  420 180
\endlabellist
  \centering
{\includegraphics[scale=.13]{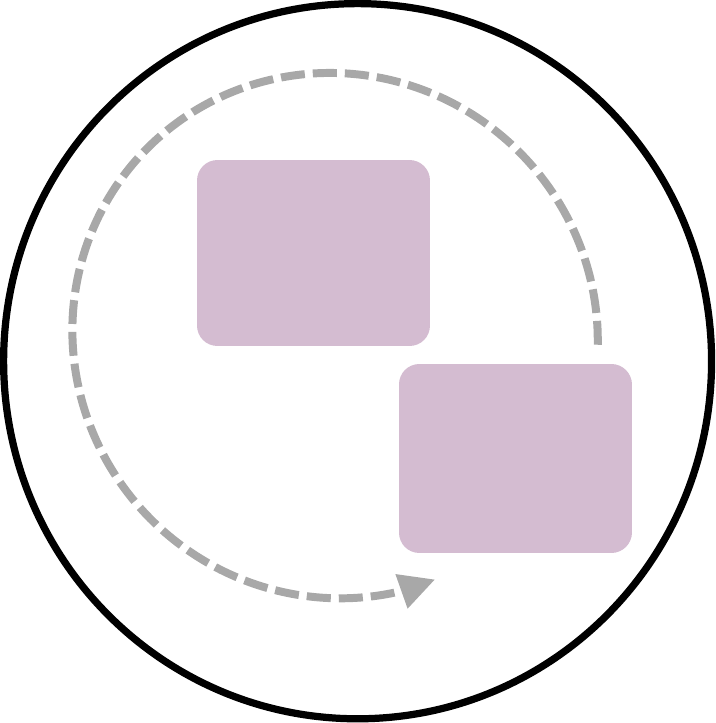}}
\caption{The homology class $\psi^2(\alpha, \beta)$.}
\end{subfigure}
\caption{Chains representing the homology operations on $H_*(F(\R^n))$.}
\label{HomologyOperations}
\end{figure}
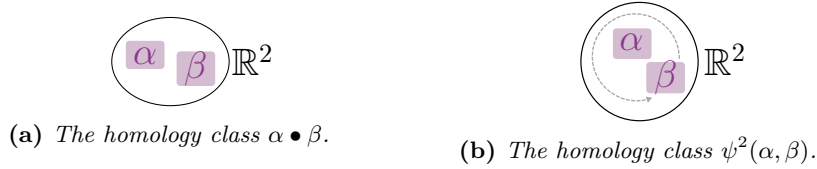  

In this paper, we are primarily interested in the operation $\psi^2$, which we simply call $\psi$. The maps $\psi^n$ come from maps at the chain level which we will also call $\psi^n$. We define $\psi$ in dimension $n >2$ at the chain level as follows. Let $\theta': S^1 \m \Emb(\R^n \sqcup \R^n, \R^n)$ be the restriction of $\theta$ to an equatorial circle. The (counterclockwise) fundamental chain of $S^1$ induces a map: $$\psi: C_i(F_S(\R^n)) \otimes C_j(F_T(\R^n)) \m C_{i+j+1}(F_{S \sqcup T}(\R^n)).$$
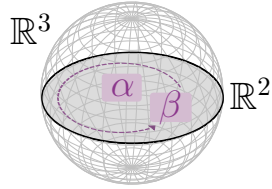
\begin{figure}[!ht]    \centering
\labellist
\Large \hair 0pt
\pinlabel {\color{Fuchsia}  $\alpha$} at  45 53
\pinlabel { \color{Fuchsia}  $\beta$} at  71 42
\pinlabel {\color{black} $\R^2$} at  115 50
\pinlabel {\color{black} $\R^3$} at  -5 85
\endlabellist
\scalebox{1}{\includegraphics[scale=.7]{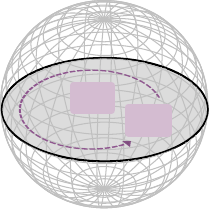}}
\caption{The homology class $\psi(\alpha, \beta) \in H_*(F(\R^3))$.}
\label{HigherOperation}
\end{figure} 

\noindent Figure \ref{HigherOperation} shows the map $\psi$ on $H_*(F(\R^3))$. Given a singleton set $S=\{s\}$, let $s$ denote the class of a point in $H_0(F_S(\R^n))$. Figure \ref{browder12isolated} shows $\psi(1,2) \in H_1(F_2(\R^n))$.

 \begin{figure}[!ht]    \centering
 \labellist \endlabellist
\includegraphics[scale=.23]{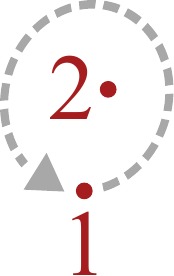}
\caption{A chain  representing $\psi(1,2)$.}
\label{browder12isolated}
\end{figure} 

 Cohen described the algebraic structure on $H_*(F_k(\R^n))$ imposed by the operations $\psi^n$ and $\bullet$: the groups $H_*(F_k(\R^n))$ assemble to form the $n$-Poisson operad. Cohen denoted the Browder operations  by $\lambda_{n-1}$ and described relations they satisfy \cite[Chapter III Theorem 1.2]{CLM} (also see Sinha \cite[Sinha Theorem 2.10]{Sinha}).

\begin{theorem}[Cohen {\cite[Chapter III]{CLM}}] \label{PropGerstenhaber} Fix $n \geq 2$.
The product $\bullet$ is an associative and graded commutative product, and the Browder operation $\psi^n$ is a graded Lie bracket of degree $(n-1)$, which together satisfy the Gerstenhaber relations. Specifically,  these operations satisfy the following identities.  Let $|\cdot|$ denote the degree of a homology class. 
\begin{itemize}
\item[] (Degrees of $\bullet$ and $\psi^n$) \qquad \qquad $ \displaystyle |\alpha \bullet \beta| = |\alpha | + |\beta|,  \qquad 
|\psi^n(\alpha ,\beta)| = |\alpha | + |\beta| + (n-1) $
\item[] (Graded commutativity for  $\bullet$)  \qquad \qquad \; $ \displaystyle  \alpha \bullet \beta = (-1)^{|\alpha||\beta|} \beta \bullet \alpha$
\item[] (Graded antisymmetry law for $\psi^n$) \qquad  $ \displaystyle  \psi^n(\alpha,\beta)=  -(-1)^{|\alpha||\beta|+(n-1)(|\alpha|+|\beta|+1)}\psi^n(\beta,\alpha)$
\item[] (Graded Jacobi identity for $\psi^n$) 
\begin{align*}
(-1)^{(|\alpha|+n-1)(|\gamma|+n-1)}\psi^n(\alpha,\psi^n(\beta, \gamma)) & + (-1)^{(|\beta|+n-1)(|\alpha|+n-1)}\psi^n(\beta, \psi^n(\gamma,\alpha))  \\ 
& + (-1)^{(|\gamma|+n-1)(|\beta|+n-1)}\psi^n(\gamma,\psi^n(\alpha,\beta))=0 
\end{align*} 
\item[](The Browder operation $\psi^n$ is a derivation of the product $\bullet$ in each variable)
$$ \psi^n(\alpha,\beta \bullet \gamma)=\psi^n(\alpha,\beta) \bullet \gamma + (-1)^{(|\alpha|+n-1)|\beta|}\beta \bullet \psi^n(\alpha,\gamma)$$
\end{itemize}
\end{theorem}

\begin{remark} When we say that $\bullet$ is a commutative product, we do not mean that $\bigoplus_{i,k} H_i(F_k(\R^n)) $ is a commutative ring. Instead, we mean that the associated $\FB$-module has the structure of graded tca. Similarly, $\psi^n$ gives an appropriate shift of the $\FB$-module associated to $\left\{\oplus_{i,k} H_i(F_k(\R^n)) \right \}_{k=0}^{\infty}$ the structure of an algebra over the Lie operad in $\FB$-modules with Day convolution.
\end{remark}

\subsection{The arc resolution and representation stability} \label{SectionArcRes}

We now recall two related semi-simplicial spaces. One was used by Kupers--Miller \cite[Appendix]{kupersmillercells} to give a new proof of homological stability for unordered configuration spaces. We will use the second to give a new proof of representation stability for ordered configuration spaces. If $M$ is a noncompact manifold, there exists a (not necessarily compact) manifold with non-empty boundary $\overline{M}$ such that $M$ is the interior of  $\overline{M}$ (see for example Miller--Palmer \cite[Section 3]{MiP1}).

\begin{definition} Let $M$ be the interior of a (not necessarily compact) smooth manifold $\overline{M}$ with nonempty boundary $\partial M$. Fix an embedding $\gamma: [0,1] \m  \partial M$. Let $$\Arc_j(F_k(M)) \subset F_k(M) \times \Emb(\sqcup_{j+1} [0,1], \overline{M})$$ be the subspace of points and arcs $(x_1,\ldots x_k;\alpha_0,\ldots \alpha_j)$ satisfying the following conditions:
\begin{itemize} \begin{multicols}{2}
  \item $\alpha_i(0) \in \gamma([0,1])$
  \item $\alpha_i(1) \in \{x_1,\ldots, x_k\}$
  \item $\alpha_i(t) \notin \partial M \cup \{x_1,\ldots, x_k\}$ for $t \in (0,1)$
  \item $\gamma^{-1}(\alpha_{j_1}(0))>\gamma^{-1}(\alpha_{j_2}(0)) $ whenever $j_1 >j_2$.
  \end{multicols}
\end{itemize}
\noindent Let $\Arc_j(C_k(M))$ denote the quotient of $\Arc_j(F_k(M))$ by the action of $\sS_k$, as in Figure \ref{ArcResolution}. 
\begin{figure}[!ht]
\begin{center}\scalebox{.23}{\includegraphics{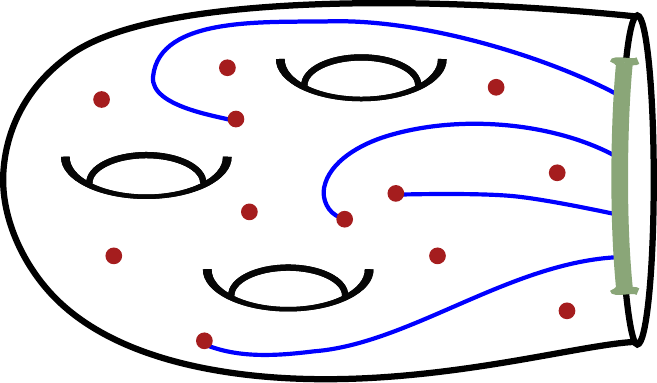}}\end{center}
\caption{An element of $\Arc_3(C_{12}(M))$.}
\label{ArcResolution}
\end{figure}  

As $j$ varies, the spaces $\Arc_j(F_k(M))$ assemble into an augmented semi-simplicial space. The $i$th face map $ d_i: \Arc_j(F_k(M)) \m \Arc_{j-1}(F_k(M))$ is given by forgetting the $i$th arc $\alpha_i$. The space $\Arc_{-1}(F_k(M))$ is homeomorpic to $F_k(M)$, and so the augmentation map induces a map $||\Arc_\bullet(F_k (M))|| \m F_k(M)$. Similarly $\Arc_j(C_k(M))$ assemble to form an augmented semi-simplicial space and $$\Arc_{-1}(C_k(M)) \cong C_k(M).$$ We call the two augmented semi-simplicial spaces $\Arc_\bullet(F_k (M))$ and $\Arc_\bullet(C_k (M))$ the \emph{ordered} and \emph{unordered arc resolutions}, respectively. 
\end{definition}
 Building on Hatcher--Wahl \cite{hatcherwahl} and a lecture of Randal-Williams, Kupers--Miller \cite[Appendix]{kupersmillercells} proved the following. 
 
\begin{theorem}[Kupers--Miller {\cite[Appendix]{kupersmillercells}}] Let $M$ be a smooth noncompact connected manifold of dimension at least two. The map $||\Arc_\bullet(C_k (M))|| \m C_k(M)$ is $(k-1)$-connected.  \label{arcresolutionconnected}
\end{theorem}

This implies the same connectivity for the arc resolution of ordered configuration spaces.

\begin{proposition} Let $M$ be a smooth noncompact connected manifold of dimension at least two. The map $||\Arc_\bullet(F_k (M))|| \m F_k(M)$ is $(k-1)$-connected. \label{ArcF}
\end{proposition}
\begin{proof}
Since the map $||\Arc_\bullet(C_k (M))|| \m C_k(M)$ is $(k-1)$-connected, its homotopy fibers (the standard path space construction) are $(k-2)$-connected.  The quotient map $F_k(M) \m C_k(M)$ induces homeomorphisms between the homotopy fibers of $||\Arc_\bullet(F_k (M))|| \m F_k(M)$ and the homotopy fibers of $||\Arc_\bullet(C_k (M))|| \m C_k(M)$. Thus the homotopy fibers of $||\Arc_\bullet(F_k (M))|| \m F_k(M)$ are $(k-2)$-connected and so the map $||\Arc_\bullet(F_k (M))|| \m F_k(M)$ is $(k-1)$-connected as well.
\end{proof}

If $M$ is connected and of dimension at least two, then the connected components of $\Arc_j(F_k(M))$ are determined by which arc connects to which point. For example, a connected component could be specified by saying that $\alpha_0$ connects to $x_3$, that $\alpha_1$ connects to $x_8$, and so forth.
Thus $\Arc_j(F_k(M))$ has $k!/(k-j-1)!$ connected components. Kupers--Miller \cite[Appendix]{kupersmillercells} showed that $\Arc_j(C_k(M))$ is homotopy equivalent to $C_{k-j-1}(M)$, and their result implies that each connected component of $\Arc_j(F_k(M))$ is homotopy equivalent to $F_{k-j-1}(M)$. The face maps of the unordered arc resolution are homotopic to the stabilization maps for unordered configuration spaces  \cite[Appendix]{kupersmillercells}. It follows that the face map on  $\Arc_j(F_k(M))$ that forgets the arc attached to the point labeled by $i$ has the effect of stabilizing by a point labeled by $i$.

An augmented semi-simplicial space $A_\bullet$ gives rise to a homology spectral sequence; see for example Randal-Williams \cite[Section 2.3]{RW}.  This spectral sequence satisfies $$E^1_{p,q}=H_q(A_p) \implies H_{p+q+1}(A_{-1},||A_\bullet ||),$$ and the differentials on the $E^1$-page are given by the alternating sum of the face maps.
 
\begin{definition} We call the spectral sequence associated to an (augmented) semi-simplicial space $A_\bullet$ the \emph{(augmented) geometric realization spectral sequence}. We call the augmented geometric realization spectral sequence for the ordered arc resolution the \emph{arc resolution spectral sequence}. We will denote the $(p,q)$th spot on the $r$th page by $E^r_{p,q}[M](S)$ and will often drop the $M$ or $S$ from the notation.
\end{definition}

\begin{proposition} \label{GeometricRealizationSS} Let $M$ be a noncompact connected smooth manifold of dimension at least two. The arc resolution spectral sequence satisfies:
$$E^1_{p,q}(S) \cong \Inj_p(H_q(F_{S}(M))) \qquad \text{for $q \geq 0$ and $p \geq -1$} .$$  It converges to: $$ H_{p+q+1}\big(F_{S}(M),||\Arc_\bullet(F_{S}(M))||\big).$$

%%%%%%%
\begin{figure}[h!]  \centering
 \begin{tikzpicture}  \footnotesize
  \matrix (m) [matrix of math nodes,
    nodes in empty cells,nodes={minimum width=5ex,
    minimum height=5ex,outer sep=2pt},
    column sep=5ex,row sep=3ex]{ 
3     &  H_3(F_{S}(M)) &  \displaystyle \bigoplus_{f:\{0\} \hookrightarrow S}H_3(F_{S-f(\{0\})}(M))   &  \displaystyle  \bigoplus_{f:\{0,1\} \hookrightarrow S}H_0(F_{S-f(\{0,1\})}(M))   &\cdots & \\       
2     &  H_2(F_{S}(M)) &  \displaystyle \bigoplus_{f:\{0\} \hookrightarrow S}H_2(F_{S-f(\{0\})}(M))   &  \displaystyle  \bigoplus_{f:\{0,1\} \hookrightarrow S}H_2(F_{S-f(\{0,1\})}(M))   &\cdots & \\              
 1     &  H_1(F_{S}(M)) & \displaystyle  \bigoplus_{f:\{0\} \hookrightarrow S}H_1(F_{S-f(\{0\})}(M))   &   \displaystyle \bigoplus_{f:\{0,1\} \hookrightarrow S}H_1(F_{S-f(\{0,1\})}(M))   &\cdots & \\                  
 0     &  H_0(F_{S}(M)) & \displaystyle  \bigoplus_{f:\{0\} \hookrightarrow S}H_0(F_{S-f(\{0\})}(M))   & \displaystyle   \bigoplus_{f:\{0,1\} \hookrightarrow S}H_0(F_{S-f(\{0,1\})}(M))   &\cdots & \\       
 \quad\strut &   -1  &  0  &  1  &2&\\};

 \draw[-stealth] (m-4-3.mid west)--(m-4-2.mid east); 
 \draw[-stealth] (m-3-3.mid west) -- (m-3-2.mid east);
 \draw[-stealth] (m-2-3.mid west) -- (m-2-2.mid east);
 \draw[-stealth] (m-1-3.mid west) -- (m-1-2.mid east);

 \draw[-stealth] (m-4-4.mid west) -- (m-4-3.mid east);
 \draw[-stealth] (m-3-4.mid west) -- (m-3-3.mid east);
 \draw[-stealth] (m-2-4.mid west) -- (m-2-3.mid east);
 \draw[-stealth] (m-1-4.mid west) -- (m-1-3.mid east);

 \draw[-stealth] (m-4-5.mid west) -- (m-4-4.mid east);
 \draw[-stealth] (m-3-5.mid west) -- (m-3-4.mid east);
 \draw[-stealth] (m-2-5.mid west) -- (m-2-4.mid east);
 \draw[-stealth] (m-1-5.mid west) -- (m-1-4.mid east);

\draw[thick] (m-1-1.east) -- (m-5-1.east) ;
\draw[thick] (m-5-1.north) -- (m-5-6.north) ;

\end{tikzpicture}
\caption{$E^1_{p,q}(S)=H_q(\Arc_p(F_S (M))) \cong   \Inj_p\Big(H_q(F(M))\Big)_{S}.$ } \label{FigureE1}
\end{figure}  
%%%%%%%

For $|S|=k$, the $E^2$-page satisfies
\begin{align*}
E^2_{p,q}(S) &\cong \bigoplus_{\substack{S = P \sqcup Q, \\ |P|=p+1}}  \tilde H_{p}(||\Inj_\bullet(P)||)\otimes H_0^{\FI} (H_q(F(M)))_Q \\
&\cong \Ind_{\sS_{p+1}\times \sS_{k-p-1}}^{\sS_k} \Top_{p+1} \boxtimes H_0^{\FI} (H_q( F(M)))_{k-p-1}.
\end{align*}
where  $\Top_{p+1} :=  \tilde H_{p}(||\Inj_\bullet(p+1)||)$.

\begin{figure}[h!]    \centering \begin{tikzpicture} \footnotesize
  \matrix (m) [matrix of math nodes,
    nodes in empty cells,nodes={minimum width=3ex,
    minimum height=5ex,outer sep=2pt},
 column sep=3ex,row sep=3ex]{ 
4    & H_0^{\FI}\Big(H_4(F(M))\Big)_{6}  & 0 &  \up_{\sS_{2}\times \sS_{4}}^{\sS_6} \Top_{2} \boxtimes H_0^{\FI} (H_4( F(M)))_{4}   &\up_{\sS_{3}\times \sS_{3}}^{\sS_6} \Top_{3} \boxtimes H_0^{\FI} (H_4( F(M)))_{3} & \\  
 3    &  H_0^{\FI}\Big(H_3(F(M))\Big)_{6}  &0 &  \up_{\sS_{2}\times \sS_{4}}^{\sS_6} \Top_{2} \boxtimes H_0^{\FI} (H_3 (F(M)))_{4}   &\up_{\sS_{3}\times \sS_{3}}^{\sS_6} \Top_{3} \boxtimes H_0^{\FI} (H_3 (F(M)))_{3} & \\  
 2    &  H_0^{\FI}\Big(H_2(F(M))\Big)_{6} & 0&  \up_{\sS_{2}\times \sS_{4}}^{\sS_6} \Top_{2} \boxtimes H_0^{\FI} (H_2( F(M)))_{4}    &\up_{\sS_{3}\times \sS_{3}}^{\sS_6} \Top_{3} \boxtimes H_0^{\FI} (H_2( F(M)))_{3} &  \\          
1     &  H_0^{\FI}\Big(H_1(F(M))\Big)_{6}   &0    & \up_{\sS_{2}\times \sS_{4}}^{\sS_6} \Top_{2} \boxtimes H_0^{\FI} (H_1( F(M)))_{4}  &\up_{\sS_{3}\times \sS_{3}}^{\sS_6} \Top_{3} \boxtimes H_0^{\FI} (H_1( F(M)))_{3} & \\             
 0     &  H_0^{\FI}\Big(H_0(F(M))\Big)_{6}  & 0 & \up_{\sS_{2}\times \sS_{4}}^{\sS_6} \Top_{2} \boxtimes H_0^{\FI} (H_0( F(M)))_{4}  &\up_{\sS_{3}\times \sS_{3}}^{\sS_6} \Top_{3} \boxtimes H_0^{\FI} (H_0( F(M)))_{3}  & \\       
 \quad\strut &   -1  &  0  &  1  & 2  &\\};

 \draw[thick] (m-1-1.east) -- (m-6-1.east) ;
 \draw[thick] (m-6-1.north) -- (m-6-6.north east) ;

\end{tikzpicture}
\caption{$E^2_{p,q}(6) \cong \Ind_{\sS_{p+1}\times \sS_{6-p-1}}^{\sS_6} \Top_{p+1} \boxtimes H_0^{\FI} (H_q( F(M)))_{6-p-1}$.  \\ The $0$th column is identically zero because $\Top_1=0$; see for example {Proposition \ref{RankTopHomology}}. } \label{E2Page}
\end{figure}

In particular, the leftmost $E^2$ column $p=-1$ are the FI--homology groups
$$E^2_{-1,q}(S) \cong H_0^{\FI} (H_q(F(M)))_{S} $$
and the bottom $E^2$ row $q=0$ are the reduced homology groups of the complex of injective words
$$E^2_{p,0}(S) \cong \tilde H_{p}(|| \Inj_\bullet(S)||) $$
which vanish except at $p = k - 1$.
\end{proposition}

The $E_1$-page and $E_2$-page of the arc resolution spectral sequence are shown in Figures \ref{FigureE1} and  \ref{E2Page}.

\begin{proof}[Proof of Proposition \ref{GeometricRealizationSS}]
By definition, the arc resolution spectral sequence satisfies $$E^1_{p,q}(S)=H_q(\Arc_p(F_{S}(M))) \qquad \text{for $q \geq 0$ and $p \geq -1$.} $$ Since $F_S(M)$ is the space of $(-1)$-simplices, the spectral sequence converges to $H_{p+q+1}(F_S(M),||\Arc_\bullet(M)||)$.
 
By the above remarks on the structure of the space $\Arc_p(F_{S}(M))$, the $E^1$-page satisfies $$E^1_{p,q}(S)=H_q(\Arc_p(F_{S}(M))) \cong \bigoplus_{ f: \{ 0 , 1, \ldots, p \} \hookrightarrow S} H_q ( F_{S \setminus \im(f) }(M)) $$ and has $d^1$ differentials induced by the alternating sum of face maps on $\Arc_\bullet(F_S (M))$ which are homotopic to stabilization maps. Hence each row of the $E^1$-page is precisely the twisted complex of injective words 
$$E^1_{p,q}(S) \cong \Inj_p\Big(H_q(F(M))\Big)_{S}$$ 
of Definition \ref{DefnInj.}.  It follows that $$ E^2_{p,q}(S) \cong H_p \Big(\Inj_*(H_q(F(M)))\Big)_{S}.$$ 
When $p=-1$,  by Proposition \ref{minusonehomology}, $$E^2_{-1,q}(S) \cong H_0^{\FI}\Big(H_q(F(M))\Big)_{S}.$$ 

Since $n \geq 2$ and $M$ is connected, the configuration space $F_k (M)$ is connected, and there is an isomorphism of FI-modules  $H_0 ( F(M) ) \cong M(0)$. Therefore when $q=0$, 
$$ E^2_{p,0}(S) \cong H_p \Big(\Inj_*(H_0(F(M)))\Big)_{S} \cong H_p \Big(\Inj_*(M(0))\Big)_S \cong \tilde H_p \Big(|| \Inj_\bullet(S) || \Big), $$ 
a group that is nonzero only when $p=|S|-1$ by Theorem \ref{Farmer}. When $|S|=k$, the $\FI\sharp$-module structure on  $H_q(F(M))$ and Theorem \ref{connectedM(W)} imply that the $E^2$-page has the form
$$E^2_{p,q}(k) \cong H_p \Big( \text{Inj}_* (H_q(F_{k}(M)))\Big) \cong \Ind_{\sS_{p+1}\times \sS_{k-p-1}}^{\sS_k} \tilde H_{p}(||\Inj_\bullet(p+1)||)\boxtimes H_0^{\FI} (H_q (F(M)))_{k-p-1} $$
as claimed. 
\end{proof}

Before we discuss applications of the arc resolution, we describe modifications necessary to deal with non-smoothable manifolds. 

\begin{remark}\label{topologicalmanifolds}
If $M$ or $N$ does not have a smooth structure, then the space of smooth embeddings of $N$ into $M$ is not defined. To modify the arguments of \cite{kupersmillercells} to prove a version of Theorem  \ref{arcresolutionconnected} which applies to topological manifolds, we need to consider a space of embeddings that satisfies a parameterized isotopy extension theorem (see Burghelea--Lashof \cite[Page 19]{topisotopy}). One space of embeddings of topological manifolds that is compatible with the proof in \cite{kupersmillercells} is the following: Let $\Emb^{lf}_\bullet(N,M)$ denote the simplicial set whose space of $k$-simplices is the set of locally flat embeddings of $\Delta^k \times N $ into $ \Delta^k \times M$ that commute with the projection to $\Delta^k$. Using $||\Emb^{lf}_\bullet(N,M)||$ in the definition of the arc resolution allows us to apply the arguments of \cite{kupersmillercells} to topological manifolds without significant modifications. For ease of exposition, we will only give proofs in the smooth case.  
\end{remark}

\begin{theorem} \label{ConfigSpaceRepStable} Let $M$ be a noncompact connected smooth manifold of dimension at least two. Then $$\deg  H_0^{\FI}(H_i(F(M);\Z)) \leq 2i.$$
\end{theorem}

When $M$ is orientable, Theorem \ref{ConfigSpaceRepStable} is a result of Church--Ellenberg--Farb \cite[Theorem of 6.4.3]{CEF} proved by different methods.

\begin{proof}[Proof of Theorem \ref{ConfigSpaceRepStable}]
Consider the arc resolution spectral sequence described in Proposition \ref{GeometricRealizationSS}. For $p+q \leq |S|-2$, Proposition \ref{ArcF} implies that the sequence converges to zero: $$E^\infty_{p,q}(S)  \cong H_{p+q+1}(F_{S}(M),||\Arc_\bullet(F_{S}(M))||) \cong 0 \qquad \text{ for $p+q \leq |S|-2$} .$$

We now prove Theorem \ref{ConfigSpaceRepStable} by induction on homological degree $i$.  Observe that $$\deg  H_0^{\FI}(H_0(F(M)))=0$$ since $H_0(F(M)) \cong M(0)$. Assume that $\deg  H_0^{\FI}(H_q(F(M))) \leq 2q$ for all $q<i$. Using Theorem \ref{twistedInj} and our inductive hypothesis,  we obtain $$E^2_{p,q}(S)=0 \qquad \text{ for $p \leq |S| -2 - 2q$ \; and \;  $q <i$}$$ (equivalently $|S| \geq p + 2(q+1)$). This shows that there are no possible differentials into (or out of) $E^r_{-1,i}(S)$ for $r>1$ and $|S|>2i$. See Figure \ref{FigureE2Inductive}.

%%%%%%%
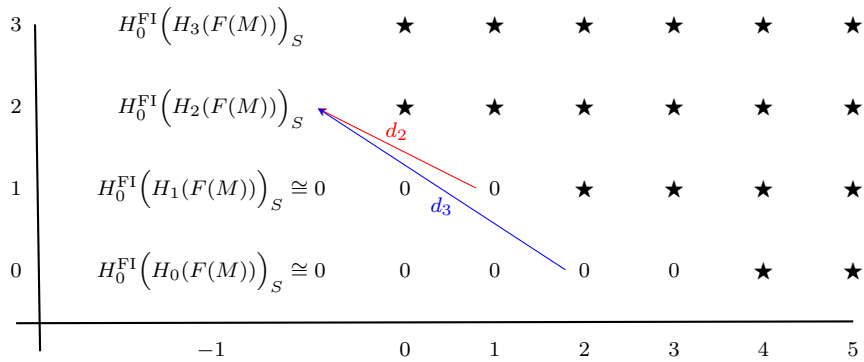
\begin{figure}[h!]    \centering \begin{tikzpicture} \footnotesize 
  \matrix (m) [matrix of math nodes,
    nodes in empty cells,nodes={minimum width=3ex,
    minimum height=5ex,outer sep=2pt},
    column sep=6ex,row sep=3ex]{  
 3    &  H_0^{\FI}\Big(H_3(F(M))\Big)_{S} &  \bigstar &  \bigstar &\bigstar & \bigstar &\bigstar &  \bigstar & \\  
 2    &  H_0^{\FI}\Big(H_2(F(M))\Big)_{S} &  \bigstar&  \bigstar &\bigstar &\bigstar &\bigstar &\bigstar &   \\          
1     &  H_0^{\FI}\Big(H_1(F(M))\Big)_{S} \cong 0  & 0   & 0&\bigstar&\bigstar &\bigstar &\bigstar &  \\             
 0     &  H_0^{\FI}\Big(H_0(F(M))\Big)_{S} \cong 0 & 0   & 0   &0 &0 &\bigstar & \bigstar &  \\       
 \quad\strut &   -1  &  0  &  1  & 2 &3 &4 & 5\\};

 \draw[-stealth, red] (m-3-4.west) -- (m-2-2.east) node [midway,above] {$d_2$};
 \draw[-stealth, blue] (m-4-5.west) -- (m-2-2.east) node [midway,below] {$d_3$};

\draw[thick] (m-1-1.east) -- (m-5-1.east) ;
\draw[thick] (m-5-1.north) -- (m-5-8.north) ;

\end{tikzpicture}
\caption{$E^2_{p,q}(S)$ in the inductive step, illustrated for $|S|=5$ and $i=2$. } \label{FigureE2Inductive}
\end{figure}  
%%%%%%
Thus for $|S|>2i$, 
 \[H_0^{\FI}(H_i(F(M)))_S \cong E^2_{-1,i}(S) \cong E^\infty_{-1,i}(S) \cong 0.\] This shows that $\deg  H_0^{\FI}(H_i(F(M))) \leq 2i$. The claim now follows by induction. 
\end{proof}

\subsection{Differentials in the arc resolution spectral sequence} \label{SectionDifferentials}

The goal of this subsection is to compute many of the differentials in the arc resolution spectral sequence. This calculation will be used in the subsequent two subsections to prove secondary representation stability for manifolds of dimension $2$ and an improved representation stability range for higher dimensional manifolds. We begin by comparing the geometric realization spectral sequence to a double complex spectral sequence.

In this subsection, the symbol $C_i(X)$ will denote the $i$-dimensional singular chains on a space $X$, as opposed to the configuration space of $i$ unordered points. Let $\partial:C_i(X) \m C_{i-1}(X)$ denote the usual boundary operator. Due to the abundance of the letter ``$d$'' in this subsection, we will denote maps on singular chains induced by face maps in the arc resolution by $f_i$. The differential $d^1$ of the arc resolution spectral sequence is given by the alternating sum of the maps in homology induced by the face maps. We will denote the map on singular chains given by the alternating sum of the face maps by $d^1$ as well. See Bendersky--Gitler \cite[Proof of Proposition 1.2]{BenderskyGitler} for a proof of the following.

\begin{proposition} 
Let $A_\bullet$ be a semi-simplicial space. Beginning on the $E^1$ pages, the geometric realization spectral sequence agrees with the spectral sequence for the double complex $C_*(A_\bullet)$ that has $d^0$ differential induced by $\partial$ and $d^1$ induced by the alternating sum of the face maps.  
\end{proposition}

In particular, we will redefine the $E^0$-page of the arc resolution spectral sequence to be the complex  $E^0_{p,q}(k) \cong C_q(\Arc_p(F_k(M)))$ shown in Figure \ref{E0Page}. Since we have reformulated the arc resolution spectral sequence as a double complex spectral sequence, we can use the standard formula for the differentials in a double complex spectral sequence (see for example Bott--Tu \cite[Formula 14.12, Page 164]{Bott&Tu}).
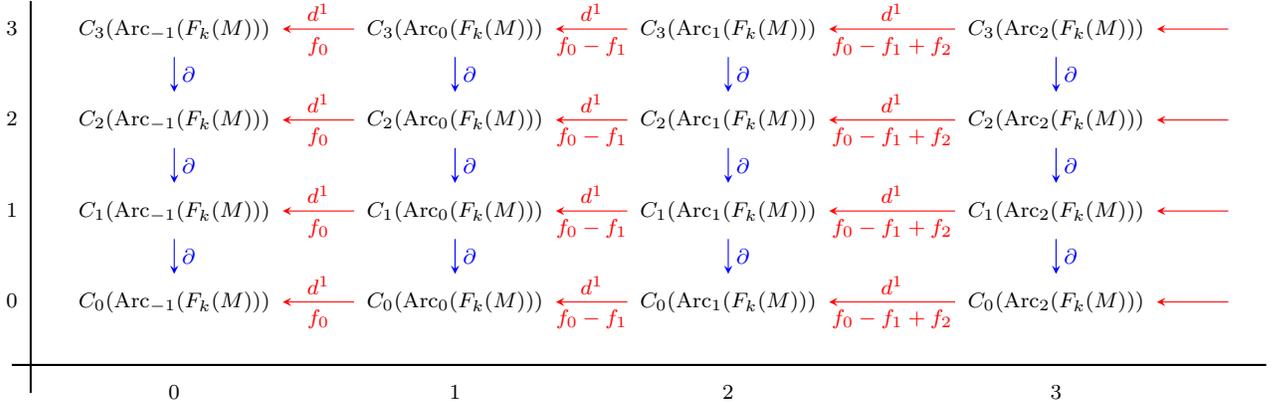
\begin{figure}[h!]    \centering \begin{tikzpicture} {\footnotesize
  \matrix (m) [matrix of math nodes,
    nodes in empty cells,nodes={minimum width=3ex,
    minimum height=5ex,outer sep=2pt},
    column sep=9ex,row sep=5ex, text height=1.5ex, text depth=0.25ex]{ 
 3    &[-4ex]   C_3(\Arc_{-1}(F_k(M)))&   C_3(\Arc_{0}(F_k(M))) &  C_3(\Arc_{1}(F_k(M))) &[6ex]  C_3(\Arc_{2}(F_k(M)))  &  \; \\  
 2    &[-4ex]   C_2(\Arc_{-1}(F_k(M))) &   C_2(\Arc_{0}(F_k(M)))&     C_2(\Arc_{1}(F_k(M))) &[6ex] C_2(\Arc_{2}(F_k(M)))  & \;  \\          
1     &[-4ex]  C_1(\Arc_{-1}(F_k(M)))  &  C_1(\Arc_{0}(F_k(M)))  &   C_1(\Arc_{1}(F_k(M))) &[6ex] C_1(\Arc_{2}(F_k(M))) &  \; \\             
 0     &[-4ex]  C_0(\Arc_{-1}(F_k(M))) &  C_0(\Arc_{0}(F_k(M)))   &   C_0(\Arc_{1}(F_k(M))) &[6ex]  C_0(\Arc_{2}(F_k(M))) & \; \\
&  0    &  1   &  2   &3&\\       }; 

 \draw[-stealth, red] (m-1-3.west) -- (m-1-2.east) node [midway,above] {$d^1$} node [midway,below] {$f_0$};
 \draw[-stealth, red] (m-2-3.west) -- (m-2-2.east) node [midway,above] {$d^1$} node [midway,below] {$f_0$};
 \draw[-stealth, red] (m-3-3.west) -- (m-3-2.east) node [midway,above] {$d^1$} node [midway,below] {$f_0$};
 \draw[-stealth, red] (m-4-3.west) -- (m-4-2.east) node [midway,above] {$d^1$} node [midway,below] {$f_0$};

 \draw[-stealth, red] (m-1-4.west) -- (m-1-3.east) node [midway,above] {$d^1$} node [midway,below] {$f_0-f_1$};
 \draw[-stealth, red] (m-2-4.west) -- (m-2-3.east) node [midway,above] {$d^1$} node [midway,below] {$f_0-f_1$};
 \draw[-stealth, red] (m-3-4.west) -- (m-3-3.east) node [midway,above] {$d^1$} node [midway,below] {$f_0-f_1$};
 \draw[-stealth, red] (m-4-4.west) -- (m-4-3.east) node [midway,above] {$d^1$} node [midway,below] {$f_0-f_1$}; 

 \draw[-stealth, red] (m-1-5.west) -- (m-1-4.east) node [midway,above] {$d^1$} node [midway,below] {$f_0-f_1+f_2$};
 \draw[-stealth, red] (m-2-5.west) -- (m-2-4.east) node [midway,above] {$d^1$} node [midway,below] {$f_0-f_1+f_2$};
 \draw[-stealth, red] (m-3-5.west) -- (m-3-4.east) node [midway,above] {$d^1$} node [midway,below] {$f_0-f_1+f_2$};
 \draw[-stealth, red] (m-4-5.west) -- (m-4-4.east) node [midway,above] {$d^1$} node [midway,below] {$f_0-f_1+f_2$};

 \draw[-stealth, red] (m-1-6.west) -- (m-1-5.east);
 \draw[-stealth, red] (m-2-6.west) -- (m-2-5.east);
 \draw[-stealth, red] (m-3-6.west) -- (m-3-5.east);
 \draw[-stealth, red] (m-4-6.west) -- (m-4-5.east);

 \draw[-stealth, blue] (m-1-2) -- (m-2-2) node [midway,right] {$\partial$};
 \draw[-stealth, blue] (m-2-2) -- (m-3-2) node [midway,right] {$\partial$};
 \draw[-stealth, blue] (m-3-2) -- (m-4-2) node [midway,right] {$\partial$};

 \draw[-stealth, blue] (m-1-3) -- (m-2-3) node [midway,right] {$\partial$};
 \draw[-stealth, blue] (m-2-3) -- (m-3-3) node [midway,right] {$\partial$};
 \draw[-stealth, blue] (m-3-3) -- (m-4-3) node [midway,right] {$\partial$};

 \draw[-stealth, blue] (m-1-4) -- (m-2-4) node [midway,right] {$\partial$};
 \draw[-stealth, blue] (m-2-4) -- (m-3-4) node [midway,right] {$\partial$};
 \draw[-stealth, blue] (m-3-4) -- (m-4-4) node [midway,right] {$\partial$};

 \draw[-stealth, blue] (m-1-5) -- (m-2-5) node [midway,right] {$\partial$};
 \draw[-stealth, blue] (m-2-5) -- (m-3-5) node [midway,right] {$\partial$};
 \draw[-stealth, blue] (m-3-5) -- (m-4-5) node [midway,right] {$\partial$};

\draw[thick] (m-1-1.north east) -- (m-5-1.east) ;
\draw[thick] (m-5-1.north) -- (m-5-6.north east) ;

}
\end{tikzpicture}
\caption{$E^0_{p,q}(k) \cong C_q(\Arc_p(F_k(M)))$. The $d^1$ differentials are equal to the alternating sum of the maps $f_i$ induced by the face maps.} \label{E0Page}
\end{figure}

 We will now describe some functoriality properties of the arc resolution spectral sequence and then make a calculation of some differentials emanating from the bottom row.

\begin{definition} \label{MapArc}
Let $M$ be the interior of a smooth $n$-manifold  $\overline{M}$ with an embedding $[0,1] \hookrightarrow \partial \overline{M}$. Choose an interval in the boundary of the half-closed disk $\R^{n-1} \times (-\infty,0]$. Fix an embedding 
$$\bar{e}: \overline M \sqcup (\R^{n-1} \times (-\infty,0]) \hookrightarrow \overline{M}$$ such that:
\begin{itemize}     
\item On the interior of the domain, $\bar{e}$ restricts to an embedding of $M \sqcup \R^n$ into $M$ such that $\bar{e}|_M$ is isotopic to the identity.
\item  The embedding $\overline e$ restricts to an embedding of the two boundary intervals of $\overline M$ and $\R^{n-1} \times (-\infty,0]$ into the boundary interval of $\overline M$. 
\end{itemize}
Then the embedding $\bar{e}$ induces a map of spaces $$e_{p,p'}:\Arc_p(F_S(\R^n)) \times \Arc_{p'}(F_T(M)) \m \Arc_{p+p'+1}(F_{S \bigsqcup T}(M))$$ as in Figure \ref{ArcEmbedding}.
\begin{figure}[!ht]    \centering
\labellist
 \hair 0pt
\pinlabel {\Large ,} at  360 100
\pinlabel {\Large $\longmapsto$} at  650 100
\pinlabel {\color{Maroon} $5$} at  452 112
\pinlabel {\color{Maroon} $5$} at  947  65
\pinlabel {\color{Maroon} $1$} at  877 132
\pinlabel {\color{Maroon} $1$} at  165 95
\pinlabel {\color{Maroon} $2$} at  105 137
\pinlabel {\color{Maroon} $2$} at  820 175
\pinlabel {\color{Maroon} $3$} at  100 40
\pinlabel {\color{Maroon} $3$} at  820 50
\pinlabel {\color{Maroon} $4$} at  145 50
\pinlabel {\color{Maroon} $4$} at  860 80
\endlabellist
  \centering
{\includegraphics[scale=.21]{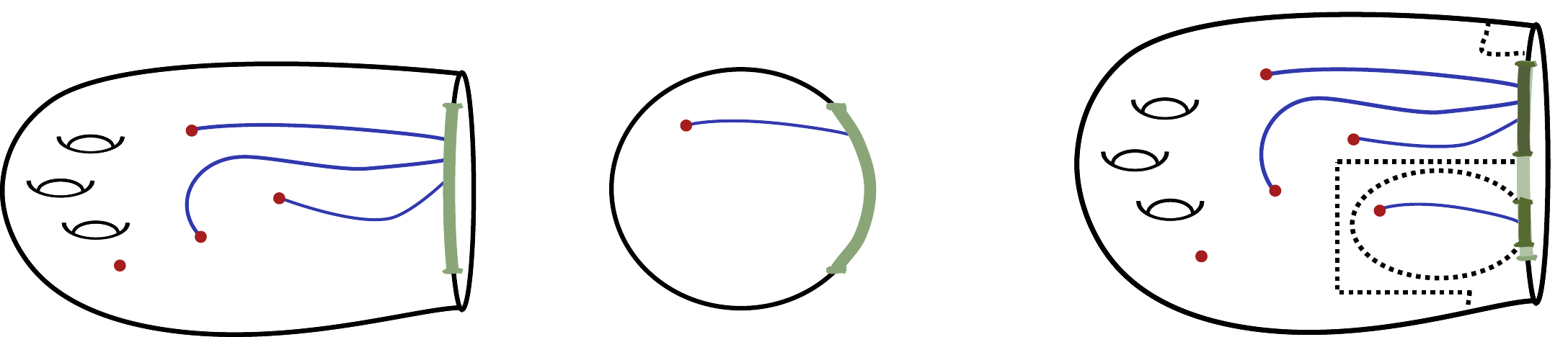}}
\caption{The map $e_{2,0}: \Arc_2(F_{\{1,2,3,4\}}(\R^n)) \times \Arc_{0}(F_{\{5\}}(M)) \m \Arc_{3}(F_{\{1,2,3,4,5\}}(M))$.}
\label{ArcEmbedding}
\end{figure}
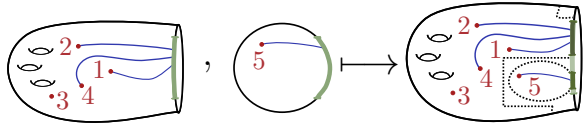  
\end{definition}

\begin{lemma} \label{LemmaLeibniz} The maps $e_{p,p'}$ of Definition \ref{MapArc} induce maps $$t^r:E^r_{p,q}[\R^n](S) \otimes E^r_{p'q'}[M](T) \m E^r_{p+p'+1,q+q'}[M](S \sqcup T)$$     These maps satisfy the following Leibniz rule with respect to the differentials; if $a \in E^r_{p,q}[\R^n](S)$ and $b \in E^r_{p'q'}[M](S)$, then
$$ d^r(t^r(a \otimes b)) = t^r(d^r(a) \otimes b) + (-1)^{p+q} t^r(a \otimes d^r(b)).$$ 
 \end{lemma}

\begin{proof}
Let $\mathcal F[M](S)_*$ denote the filtered chain complex given by filtering the double complex $E^0_{*,*}[M](S)$ in the simplicial direction. The maps $e_{p,p'}$ assemble to form a filtered chain map $$e_{p,p'*}: \mathcal F[\R^n](S)_* \otimes \mathcal F[M](T)_* \m \mathcal F[M](S \sqcup T)_{*+1}.$$ Filtered chain maps induce pairings of filtered chain complex spectral sequences which satisfy the Leibniz rule with respect to the differentials (see for example Helle \cite[Lemma 3.5.2]{GOH} or Massey \cite[Section 8]{MasseyExact}).

\end{proof}

In this subsection, we will not use the full strength of Lemma \ref{LemmaLeibniz} and will only use it to produce maps of spectral sequences. However, in the next two subsections, we will use the pairing to compute differentials.

\begin{convention}  We can produce chains in $E^0_{p,q}(k) \cong C_q(\Arc_p(F_k(M)))$ from $q$-parameter families of points in $\Arc_p(F_k(M))$. For example, Figure \ref{FigHomologyDemo} shows a map  $[0,1]^2 \to \Arc_0(F_5( M))$. Given any subdivision of the product $[0,1]^2$ into triangles, we can express this map as a linear combination of singular chains.  We interpret Figure \ref{FigHomologyDemo} to represent the associated chain in $C_2(\Arc_0(F_5 M ))$ or its homology class  in $H_2(\Arc_0(F_5 M ))$. 
\begin{figure}[!ht]    \centering
{\includegraphics[scale=.2]{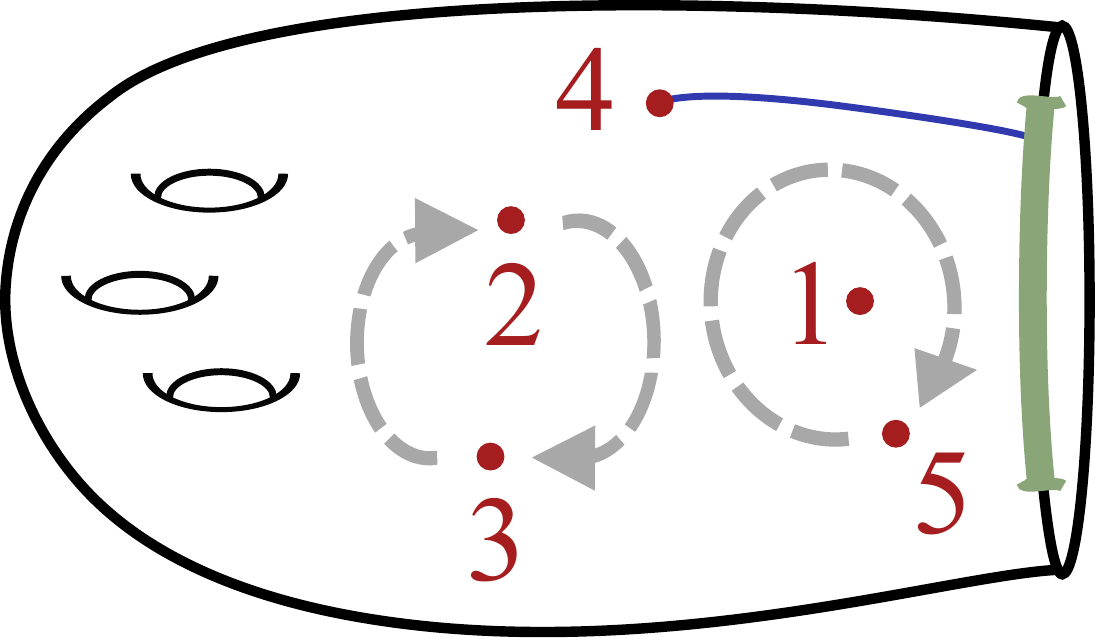}}
\caption{A map $[0,1]^2 \to \Arc_0(F_5 M)$. As the first factor $[0,1]$ ranges from $0$ to $1$, point 3 moves from bottom to top while simultaneously point 2 moves from top to bottom. As the second factor $[0,1]$ ranges from $0$ to $1$, point 5 moves in a closed loop around point 1.}
\label{FigHomologyDemo}
\end{figure}  
To view a map $s: [0,1]^q \to \Arc_p(F_k(M))$ as a sum of chains, we need to choose an order on the $q$ factors $e_i:[0,1] \to \Arc_p(F_k(M))$ of the domain. To compute its boundary, we use the formula 
$$ \partial s = \sum_j (-1)^{j+1} e_1 \times e_2 \times \cdots \times \partial(e_j) \times \cdots \times e_q \qquad \qquad  \text{where } \quad \partial e_i = e_i(1) - e_i(0) $$ 
For example, if we order the two singular $1$-simplices in Figure \ref{FigHomologyDemo} as they appear left to right, and observe that the second singular $1$-simplex is a cycle, we find that the boundary is the chain shown in Figure \ref{FigHomologyDemoBoundary}.
\begin{figure}[!ht]    \centering
{\includegraphics[scale=.2]{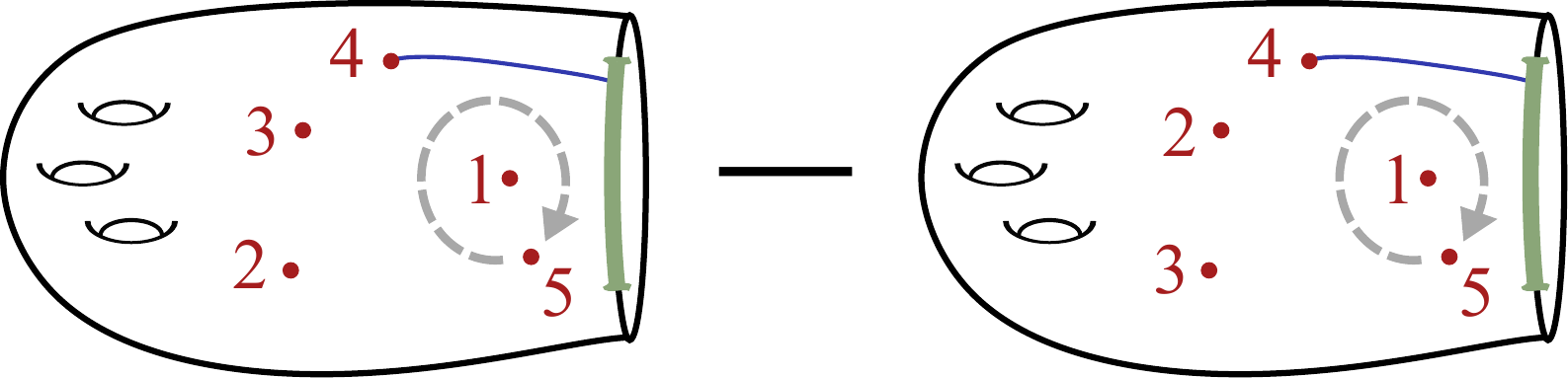}}
\caption{The boundary of Figure  \ref{FigHomologyDemo}.}
\label{FigHomologyDemoBoundary}
\end{figure}  
More generally, consider $e:[0,1] \to \Arc_p(F_k(M))$ and $y =\sum_\alpha m_\alpha \sigma_\alpha \in  C_{q-1}(\Arc_p(F_k(M)))$ with $\sigma_{\alpha}: \Delta^{q-1} \m \Arc_p(F_k(M)) $ and $m_\alpha \in \Z$.  Let $(e \times y) \in C_{q}(\Arc_p(F_k(M)))$ be a chain obtained by functorially subdividing $[0,1] \times \Delta^{q-1}$ into copies of $\Delta^q$ (e.g. see Hatcher \cite[Theorem 2.10]{hatcherbook}). To clarify how we orient the simplicies in the sudivision of $[0,1] \times \Delta^{q-1}$, note that the desired construction satisfies $$ \partial(e \times y) =\Big( (\partial e \times y) - (e \times \partial y) \Big) = \Big( (e(1)\times y )- (e(0) \times y) - (e \times \partial y) \Big).$$ For example, the product chain shown in Figure \ref{y3} has boundary given by the chain depicted in Figure  \ref{d(ey)}. 
\begin{figure}[!ht]   \centering
\begin{subfigure}{.35\textwidth}
\labellist
\large \hair 0pt
\pinlabel {\color{Fuchsia} $y$} at  236 122
\endlabellist
  \centering
{\includegraphics[scale=.19]{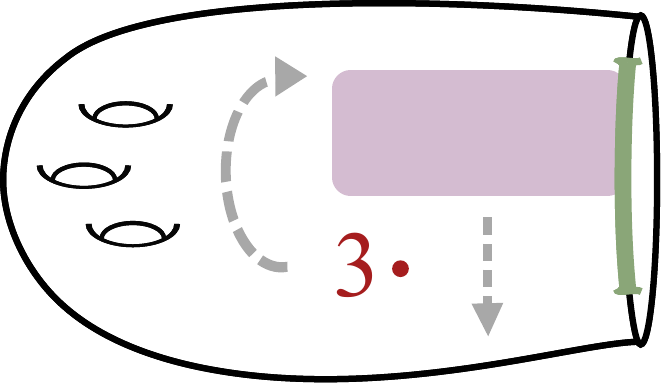}}
 \caption{(ordered) product of chains.}
  \label{y3}
\end{subfigure}%
\begin{subfigure}{.65\textwidth}
\labellist
\large \hair 0pt
\pinlabel {\color{Fuchsia} $y$} at  233 83
\pinlabel {\color{Fuchsia} $y$} at  679 117
\pinlabel {\color{Fuchsia} \footnotesize $\partial y$} at  1123 124
\endlabellist
  \centering
{\includegraphics[scale=.19]{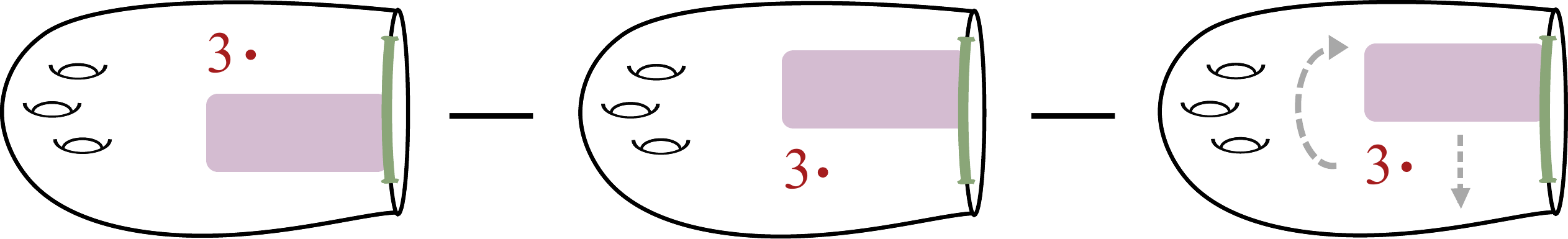}}
\caption{ boundary of a product of chains.}
  \label{d(ey)}
\end{subfigure}
\caption{The boundary of a product of chains.}  \label{FigSampleBoundary}
\end{figure}  
These conventions will feature in the computations carried out below and are important for determining signs. 
\end{convention}

The main result of this section is the values of the differentials computed in the following lemma. This result, combined with the Leibniz rule stated in Lemma \ref{LemmaLeibniz}, determines a large portion of the differentials in arc resolution spectral sequence.

\begin{lemma} \label{LemmaDifferentialsOnLie} Let $E^r_{p,q}(S)$ be the arc resolution spectral sequence and let $k=|S|$. Consider an element of Reutenauer's basis  for $\cL_S$ (Theorem \ref{ReutenauersBasis}) $$ L=[[[ \cdots[a_1, a_2], a_3], \ldots], a_{k-1}], a_k ] \in E^1_{k-1,0}(S).$$ Then $d^r(L)=0$ for $r<k$ and $d^k(L)$ is the image of the class $t_{\psi(  \cdots \psi(\psi(\psi(a_1, a_2), a_3), a_4), \cdots, a_k )}(y_0) %\in E^1_{-1,k-1}(S)$ 
$ in $E^k_{-1,k-1}(S)$. Here $y_0$ denotes the class of a point in $H_0(F_0(M))$.  

  \end{lemma}

 The image of the element $[[[1,2],3],4]\in E^1_{3,0}(4)$ is shown in Figure \ref{SampleDifferential}. 
\begin{figure}[!ht]    \centering
\includegraphics[scale=.19]{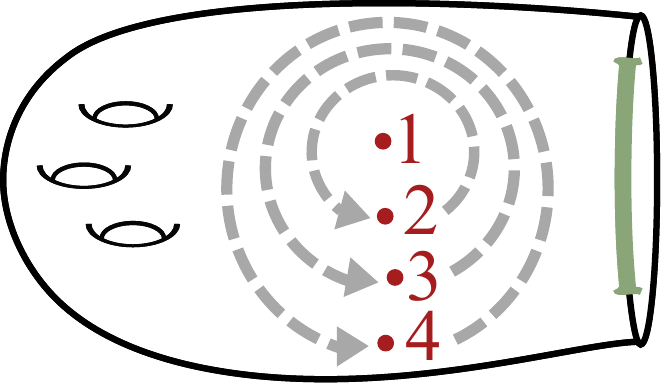}
\caption{$d^4\Big([[[1,2],3],4]\Big)$ The points labeled 2, 3, and 4 orbit counterclockwise around the point labeled 1 in concentric circles. }
\label{SampleDifferential}
\end{figure}  

From now on, we will simply write statements such as the above as $$d^k(L) = t_{\psi(  \cdots \psi(\psi(\psi(a_1, a_2), a_3), a_4), \cdots, a_k )}(y_0) $$ as we will implicitly identify elements that survive to later pages of spectral sequences with their images.

\begin{proof}[Proof of Lemma \ref{LemmaDifferentialsOnLie}] The statement of the theorem involves two numbers $k=|S|$ and $r$, the page of the spectral sequence. Our proof will involve a nested induction, first inducting on $k$ and then inducting on $r$. Throughout our inductive argument, any assumption we make will be understood to apply to all manifolds as opposed to just one particular fixed manifold. 

We begin with the base case $k=1$, $S=\{1\}$, where we observe that the $d^1$ differential maps the singleton word $1 \in E^1_{0,0}[M](S)$  to the class of the point in $H_0(\Arc_{-1}(F_{1}(M)))=H_0(F_{1}(M))$.  This result, shown in Figure  \ref{basecase}, follows from the description of the spectral sequence in Proposition \ref{GeometricRealizationSS}. 
\begin{figure}[!ht]    \centering
\labellist
 \hair 0pt 
\pinlabel {\LARGE  $\overset{d^1}{\longmapsto}$} at  430 110
\pinlabel {\color{Maroon} $1$} at  161 107
\pinlabel {\color{Maroon} $1$} at  702 107
\endlabellist
\includegraphics[scale=.18]{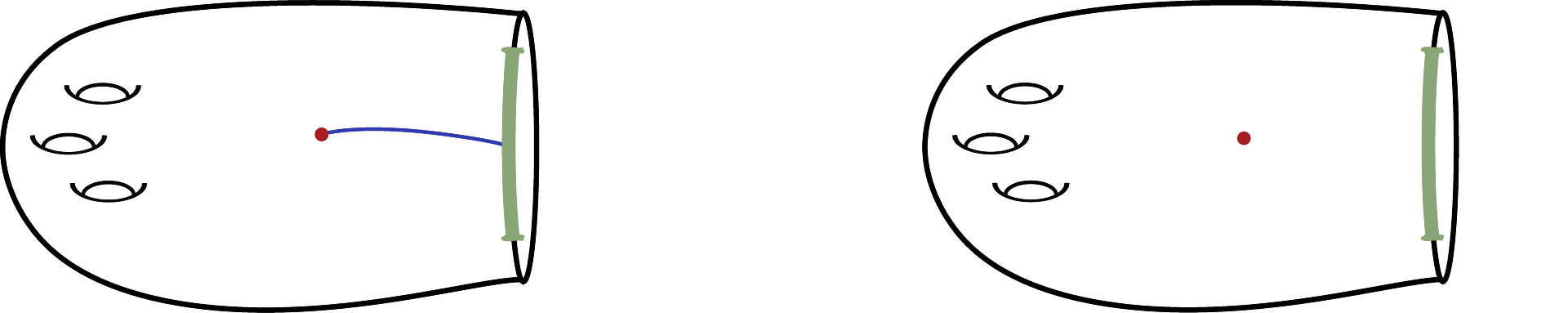}
\caption{$d^1(1)$.}
\label{basecase}
\end{figure}  

Now suppose that $k>1$, and let  $x_{k-1}$ be the class $[[ \cdots[a_1, a_2], a_3], \ldots], a_{k-1}] \in E^1_{k-2, 0}[M](S\backslash\{a_k\})$. Suppose by induction that $x_{k-1}$ survives to  $E^{k-1}_{k-2, 0}[M](S \backslash\{a_k\})$, and $$d^{k-1}(x_{k-1}) =t_{\psi(  \cdots \psi(\psi(\psi(a_1, a_2), a_3), a_4), \cdots, a_{k-1})}(y_0).$$  This implies that there exist chains $x_{k-2}, \ldots, x_1$ with $x_i \in E^0_{i-1, k-i-1}[M](S \backslash \{a_k \})$, such that $d^1(x_i)=(-1)^{i-1} \partial(x_{i-1}),$ as in Figure \ref{dnDifferential} (compare to Bott--Tu \cite[Formula 14.12, Page 164]{Bott&Tu}). 
\begin{figure}[h!]    \centering \begin{tikzpicture} { \footnotesize
  \matrix (m) [matrix of math nodes,
    nodes in empty cells,nodes={minimum width=3ex,
    minimum height=5ex,outer sep=2pt},
    column sep=6ex,row sep=3ex, text height=1.5ex, text depth=0.25ex]{         
 k-1    & \cdot &  \cdot&    \cdot &\cdots  &  \cdot &  \cdot&    \cdot & \;  \\          
k-2    & d^{k-1}(x_{k-1}) &  x_1&   \cdot &\cdots &   \cdot &  \cdot&    \cdot  &  \; \\             
 k-3     &\cdot&  -\partial x_1 & x_2  &\cdots&  \cdot &  \cdot&    \cdot  \; \\
   \vdots  & \vdots &  \vdots &  \vdots &  \ddots  & \vdots  & \vdots &  \vdots &  \\ 
 2 &  \cdot &  \cdot&    \cdot &\cdots  & x_{k-3} &  \cdot&    \cdot & \;  \\ 
 1 &   \cdot &  \cdot&    \cdot &\cdots     &  (-1)^{k-3} \partial x_{k-3}  &  x_{k-2}&   \cdot  & \;  \\ 
 0 &  \cdot &  \cdot&    \cdot &\cdots  &\cdot&  (-1)^{k-2} \partial x_{k-2} & x_{k-1} & \; \\ 
  \;\;\;& -1    &  0   &  1   & \cdots & k-4 & k-3 & k-2 \\       }; 

 \draw[-stealth, red] (m-2-3.west) -- (m-2-2.east) node [midway,above] {$d^1$} node [midway,below] {$f_0$};

 \draw[-stealth, red] (m-3-4.west) -- (m-3-3.east) node [midway,above] {$d^1$} node [midway,below] {$f_0-f_1$}; 
 
  \draw[-stealth, blue] (m-2-3) -- (m-3-3) node [midway,right] {$-\partial$};
 
  \draw[-stealth, blue] (m-3-4) -- (m-4-4) node [midway,right] {$\partial$};

  \draw[-stealth, red] (m-7-8.west) -- (m-7-7.east) node [midway,above] {$d^1$};

 \draw[-stealth, red] (m-6-7.west) -- (m-6-6.east) node [midway,above] {$d^1$}; 
 
  \draw[-stealth, red] (m-5-6.west) -- (m-5-5.east) node [midway,above] {$d^1$}; 

 \draw[-stealth, blue] (m-5-6) -- (m-6-6) node [midway,right] {$(-1)^{k-3} \partial$};
 
  \draw[-stealth, blue] (m-6-7) -- (m-7-7) node [midway,right] {$(-1)^{k-2}  \partial$};

 \draw[-stealth, green] (m-7-8) -- (m-2-2) node [midway, above right] {$d^k$};

\draw[thick] (m-1-1.north east) -- (m-1-1.north east|-m-8-1.east) ;
\draw[thick] (m-8-1.north) -- (m-8-8.north east) ;

}
\end{tikzpicture}
\caption{ Computing $d^{k-1}(x_{k-1}) = t_{\psi(  \cdots \psi(\psi(\psi(a_1, a_2), a_3), a_4), \cdots, a_{k-1})}(y_0)$. } \label{dnDifferential}
\end{figure}
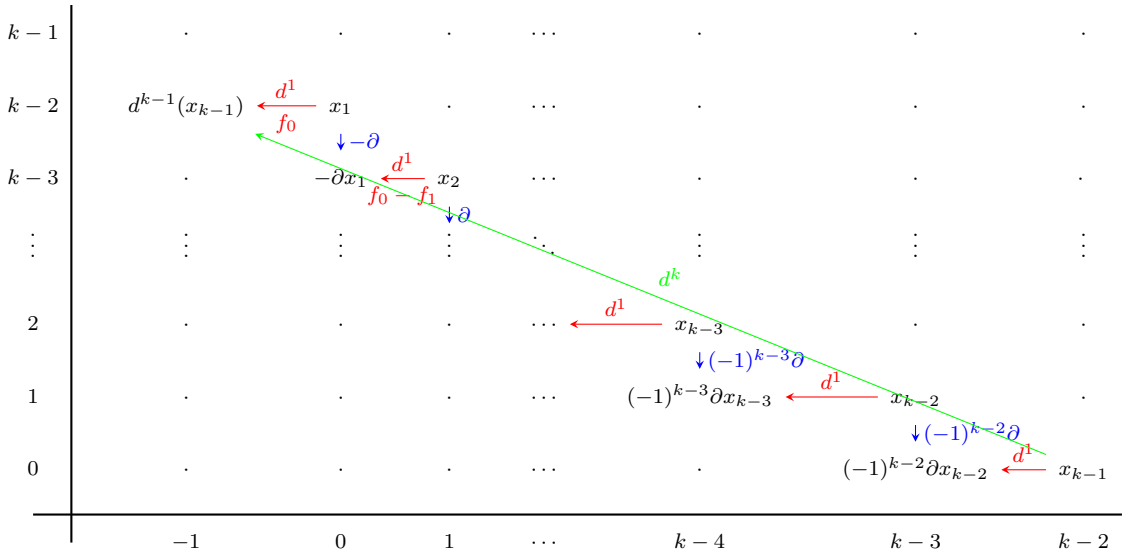

By plugging in the class of a point in $E^0_{-1,0}[\R^n](\{a_k\})$ and considering the map from Lemma \ref{LemmaLeibniz}, we get a map $E^r_{p,q}[M](S \backslash \{a_k\}) \m E^r_{p,q}[M](S)$.  We will use this to view the classes $x_i$ as elements of $E^0_{i-1, k-i-1}[M](S)$. Similarly, by plugging in the class of a point in $E^0_{-1,0}[M](\varnothing)$ and considering the map from Lemma \ref{LemmaLeibniz}, we get a map $E^r_{p,q}[\R^n](S ) \m E^r_{p,q}[M](S)$ which will allow us to associate classes in $E^r_{p,q}[\R^n](S )$ with classes in $E^r_{p,q}[M](S )$. The chain $x_i$ can be taken to be in the image of $E^0_{i-1, k-i-1}[\R^n](S) \m E^0_{i-1, k-i-1}[M](S)$. This uses our inductive assumptions applied to the case the manifold is $\R^n$ and the fact that $$L=[[[ \cdots[a_1, a_2], a_3], \ldots], a_{k-1}], a_k ] \in E^1_{k-1,0}[M](S)$$ is in the image of $E^1_{k-1,0}[\R^n](S)$. In other words,  $x_i$ can be represented as in Figure \ref{xBox}. Remember that the class $x_i$ is a chain on a space with $i$ arcs. 

\begin{figure}[!ht]    \centering
\labellist
 \hair 0pt 
\pinlabel {\color{Fuchsia} $x_{i}$} at  235 95
\endlabellist
\includegraphics[scale=.2]{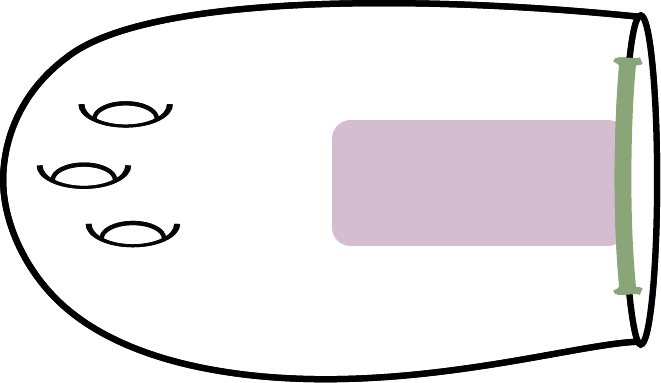}
\caption{The points of the chain $x_i$ can be taken to be in the shaded box.}
\label{xBox}
\end{figure}  

We may assume without loss of generality that the label $a_k$ is the letter $k$. Now consider the class $$L=[x_{k-1},k] \in E^1_{k-1,0}[M](S)$$ as shown in Figure \ref{xn}.

\begin{figure}[!ht]    \centering
\labellist
 \hair 0pt 
\pinlabel {\Large  $-(-1)^{k-1}$} [l] at  325 90
\pinlabel {\color{Maroon} $k$} at  161 137
\pinlabel {\color{Fuchsia} $x_{k-1}$} at  245 80
\pinlabel {\color{Maroon} $k$} at  766 59
\pinlabel {\color{Fuchsia} $x_{k-1}$} at  850 115
\endlabellist
\includegraphics[scale=.23]{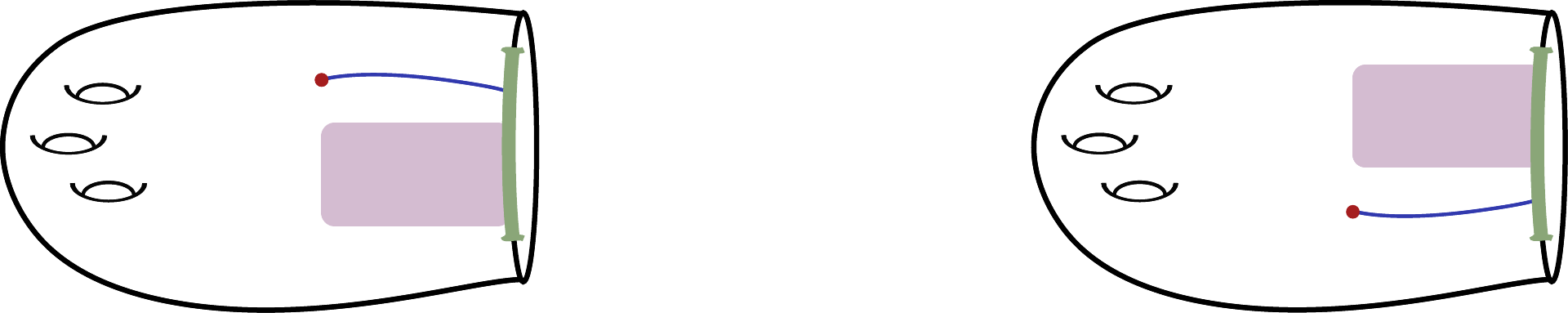}
\caption{A chain representing $[x_{k-1},k]$.}
\label{xn}
\end{figure}
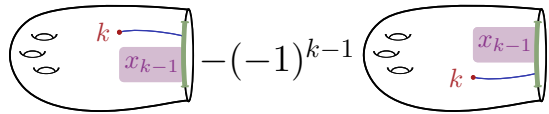  

Our goal is to show that $d^r([x_{k-1},k])=0$ for $r <k $, and to compute $d^k([x_{k-1},k])$. To do this, we will compute a zigzag of chains $\xi_i \in E^0_{i, k-i-1}[M](S)$ satisfying $(-1)^{i}\partial(\xi_{i-1})=d^1(\xi_{i})$, beginning with $\xi_{k-1}=L$. 
The image $$d^1([x_{k-1},k])=\sum (-1)^i f_i ( [x_{k-1},k])$$ is shown in Figure  \ref{fxn}. 
\begin{figure}[!ht]    \centering
\labellist
 \hair 0pt
\pinlabel { \Large  $-(-1)^{k-1}$} [l] at  773 90
\pinlabel {\color{Maroon} $k$} at  165 140
\pinlabel {\color{Maroon} $k$} at  606 139
\pinlabel {\color{Maroon} $k$} at  1230 60
\pinlabel {\color{Maroon} $k$} at  1682 63
\pinlabel {\color{Fuchsia} $x_{k-1}$} at  227 79
\pinlabel {\color{Fuchsia} $x_{k-1}$} at  1763 117
\pinlabel {\small \bf \color{Fuchsia} $d^1(x_{k-1})$} at  661 83
\pinlabel {\small \bf \color{Fuchsia} $d^1(x_{k-1})$} at 1279 119
\endlabellist
\includegraphics[scale=.23]{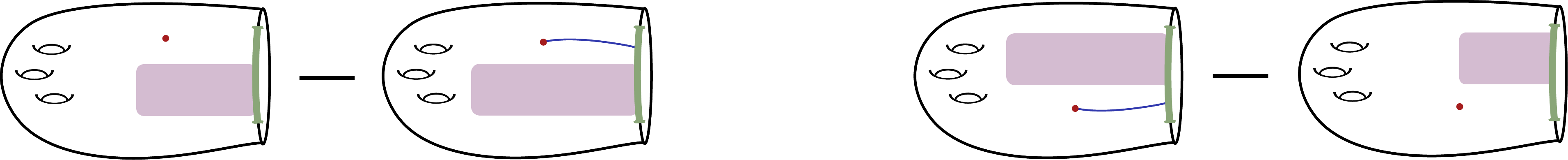}
\caption{A chain representing  $\sum (-1)^i f_i([x_{k-1},k])$.}
\label{fxn}
\end{figure}
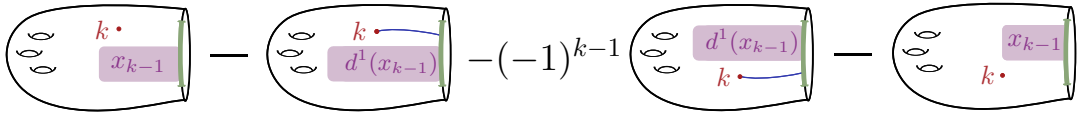  

Then $\sum (-1)^i  f_i ([x_{k-1},k])$ is equal to the boundary $(-1)^{k-1} \partial (\xi_{k-2})$, where $ \xi_{k-2}$ is the chain shown in Figure \ref{xin}. Recall that $x_{k-2}$ is defined such that $(-1)^{k-2}\partial (x_{k-2}) = d^1(x_{k-1})=\sum (-1)^i  f_i(x_{k-1})$, and that $\partial(x_{k-1})=0$. In Figure \ref{xin}, and in the images throughout this proof, we will order the simplicies with the simplex designated by the dotted line first, and the class $x_i$ in the shaded region second, so the boundary is computed as in Figure \ref{FigSampleBoundary}. We have shown that $d^1 ([x_{k-1},k])$ is zero in homology, and $[x_{k-1},k]$ survives to $E^2$.  
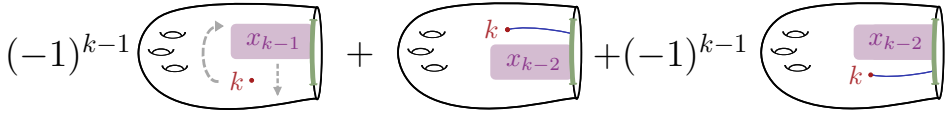
\begin{figure}[!ht]    \centering
\labellist
 \hair 0pt 
 \pinlabel {\Large  $(-1)^{k-1}$} [r] at  -10 95
\pinlabel {\Large  $+(-1)^{k-1}$} [l] at  780 95
\pinlabel {\Large  $+$}  at  381 95
\pinlabel {\color{Maroon} $k$} at  170 50
\pinlabel {\color{Maroon} $k$} at  606 139
\pinlabel {\color{Maroon} $k$} at  1230 60
\pinlabel {\color{Fuchsia} $x_{k-1}$} at  233 115
\pinlabel {\ \color{Fuchsia} $x_{k-2}$} at  671 79
\pinlabel { \color{Fuchsia} $x_{k-2}$} at 1299 115
\endlabellist
\includegraphics[scale=.23]{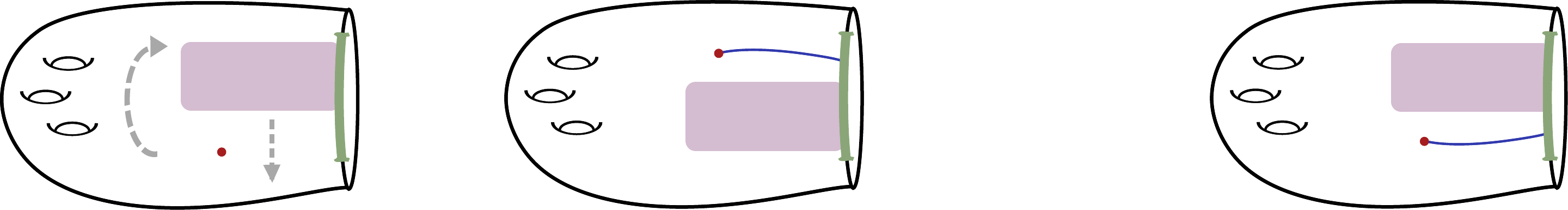}
\caption{The chain $\xi_{k-2}$.}
\label{xin}
\end{figure}  

We now prove by induction on $r$ that $d^{r-1}([x_{k-1},k])=0$ for $r<k$. Suppose by induction that $d^{r-1}([x_{k-1},k])$ is represented by the boundary $(-1)^{k-r-1} \partial(\xi_{k-r})$, where the chain $\xi_{k-r}$ is shown in Figure \ref{xin-Inductive}.
\begin{figure}[!ht]    \centering
\labellist
\hair 0pt
\pinlabel {\Large $+\;(-1)^{k-r+1}$} [l] at  870 86
\pinlabel {\Large  $+$} at  480 86
\pinlabel {\Large $(-1)^{k-r+1}$} [r] at  75 86
\pinlabel {\color{Maroon} $k$} at  270 50
\pinlabel {\color{Maroon} $k$} at  690 139
\pinlabel {\color{Maroon} $k$} at  1383 60
\pinlabel {\small \color{Fuchsia} $x_{k-r+1}$} at  330 115
\pinlabel {\ \color{Fuchsia} $x_{k-r}$} at  760 79
\pinlabel { \color{Fuchsia} $x_{k-r}$} at 1453 115
\endlabellist
\includegraphics[scale=.23]{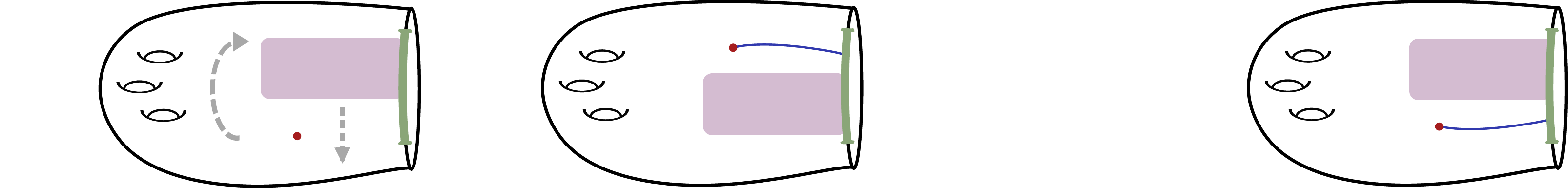}
\caption{The chain $\xi_{k-r}$.}
\label{xin-Inductive}
\end{figure}
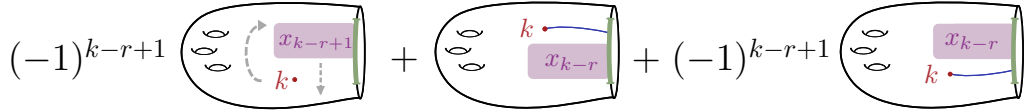  
 Then  $d^{r}([x_{k-1},k]) = \sum (-1)^i f_i(\xi_{k-r})$  is shown in Figure \ref{fxin}.
\begin{figure}[!ht]    \centering
\labellist
 \hair 0pt
\pinlabel {\Large  $-$} at  1037 330
\pinlabel { \Large  $+$} at  556 330
\pinlabel {\Large  $(-1)^{k-r+1}$} [l] at  -130 330
\pinlabel {\Large  $+\,\;(-1)^{k-r+1}$} [l] at  75 83
\pinlabel {\Large  $-$} at  835  83
\pinlabel {\color{Maroon} $k$} at  320 293
\pinlabel {\color{Maroon} $k$} at  803 376
\pinlabel {\color{Maroon} $k$} at  1279 376
\pinlabel {\color{Maroon} $k$} at  613 55
\pinlabel {\color{Maroon} $k$} at  1070 55 
\pinlabel {\tiny \color{Fuchsia} $d^1(x_{k-r+1})$} at  380 353
\pinlabel { \color{Fuchsia} $x_{k-r}$} at  874 316
\pinlabel {\color{Fuchsia} $x_{k-r}$} at 1140 113
\pinlabel {\small \color{Fuchsia} $d^1(x_{k-r})$} at  666 120
\pinlabel {\small \color{Fuchsia} $d^1(x_{k-r})$} at  1331 320
\endlabellist
\includegraphics[scale=.23]{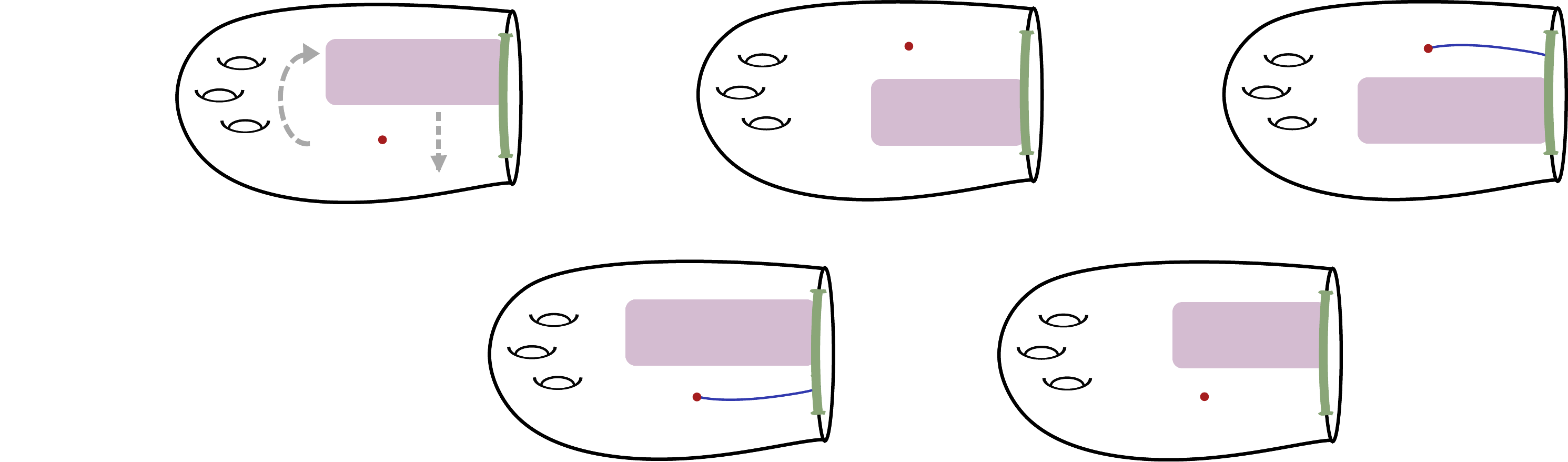}
\caption{The chain $d^{r}([x_{k-1},k]) = \sum (-1)^i f_i(\xi_{k-r})$ for $r \leq k-1$.}
\label{fxin}
\end{figure}  
  If $r \leq k-2$, then by inductive hypothesis there is a chain $x_{k-r-1}$ with $$(-1)^{k-r-1}\partial x_{k-r-1} =\sum (-1)^i f_i(x_{k-r}).$$ 
In this case, the chain $\xi_{k-r-1}$ in Figure  \ref{boundinghigher}  is such that $(-1)^{k-r-1} \partial(\xi_{k-r-1})$ equals the chain representing $d^{r}([x_{k-1},k])$ in Figure \ref{fxin}, and so $d^r([x_{k-1},k])=0$ on $E^r_{k-1-r,r}[M](S)$. By comparing Figure \ref{boundinghigher} to Figure \ref{xin-Inductive}, we see we have completed the inductive step in the induction on $r$. 
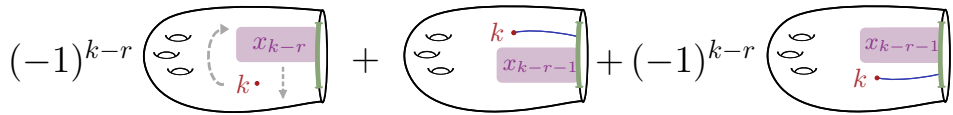
\begin{figure}[!ht]    \centering
\labellist
 \hair 0pt
\pinlabel {\Large  $+\,(-1)^{k-r}$} [l] at  983 90
\pinlabel {\Large  $+$} [l] at  565 90
\pinlabel {\Large  $(-1)^{k-r}$} at  85 90
\pinlabel {\color{Maroon} $k$} at  380 50
\pinlabel {\color{Maroon} $k$} at  815 139
\pinlabel {\color{Maroon} $k$} at  1437 60
\pinlabel {\color{Fuchsia} $x_{k-r}$} at  443 115
\pinlabel {\small \color{Fuchsia} $x_{k-r-1}$} at  890 79
\pinlabel {\small \color{Fuchsia} $x_{k-r-1}$} at 1510 115
\endlabellist
\includegraphics[scale=.23]{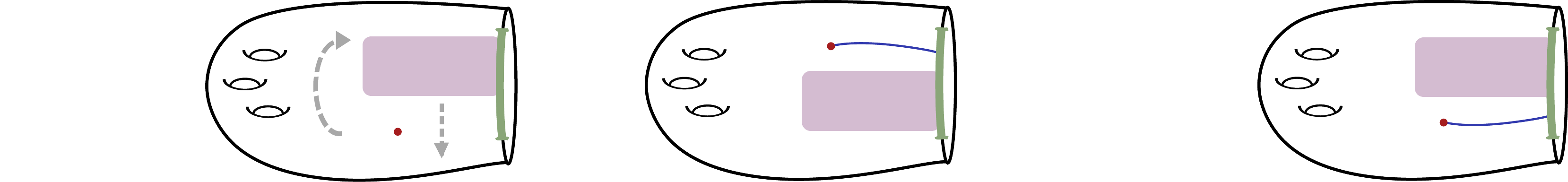}
\caption{A chain $\xi_{k-r-1}$ with $(-1)^{k-r-1}\partial(\xi_{k-r-1})  = \sum (-1)^i f_i(\xi_{k-r})$ for $r \leq k-2$.}
\label{boundinghigher}
\end{figure}  

Now consider Figure \ref{fxin} when $r = k-1$. There are no arcs attached in the chain $$d^1(x_1) = \sum (-1)^i  f_i(x_1)=f_0(x_1),$$ and by induction  $d^1(x_1)=d^{k-1}(x_{k-1})$ is a $\partial$-cycle. Hence the chain $\sum  (-1)^i f_i([x_{k-1},k])$ is the boundary of the chain in Figure \ref{boundingfxi-Lie}. Again, we conclude that $d^{k-1}([x_{k-1},k])=0$. 
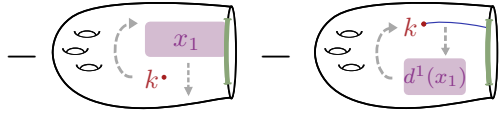
\begin{figure}[!ht]    \centering
\labellist
\hair 0pt
\pinlabel {\LARGE  $-$}  at  480 90
\pinlabel {\LARGE  $-$}  at  45 90
\pinlabel {\color{Maroon} $k$} at  267 51
\pinlabel {\color{Maroon} $k$} at  710 141
\pinlabel {\color{Fuchsia} $x_1$} at  327 115
\pinlabel {\footnotesize \color{Fuchsia} $d^1(x_1)$} at  753 55
\endlabellist
\includegraphics[scale=.23]{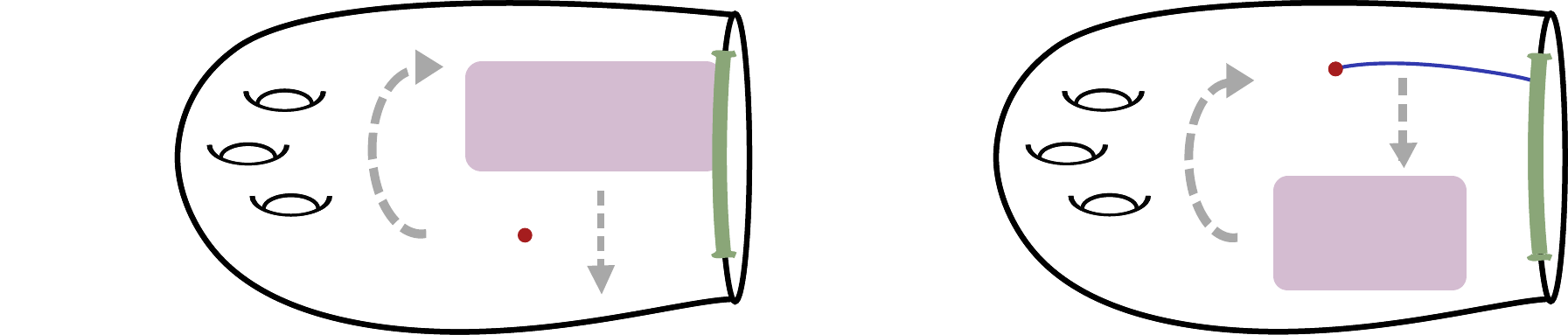}
\caption{A chain $\xi_0$ with boundary $-d^{k-1}([x_{k-1},k]) = -\sum (-1)^i f_i(\xi_{1})$.}
\label{boundingfxi-Lie}
\end{figure}  
We can compute $d^{k}([x_{k-1},k])$ by applying the map induced by the alternating sum of face maps to Figure \ref{boundingfxi-Lie}, with the result shown in Figure \ref{fBoundingfxi-Lie}. 
\begin{figure}[!ht]    \centering
\labellist
\hair 0pt
\pinlabel {\LARGE  $-$}  at  480 90
\pinlabel {\LARGE  $-$}  at  45 90
\pinlabel {\color{Maroon} $k$} at  267 51
\pinlabel {\color{Maroon} $k$} at  720 161
\pinlabel {\footnotesize  \color{Fuchsia} $d^1(x_1)$} at  312 115
\pinlabel {\footnotesize \color{Fuchsia} $d^1(x_1)$} at  753 55
\endlabellist
\includegraphics[scale=.23]{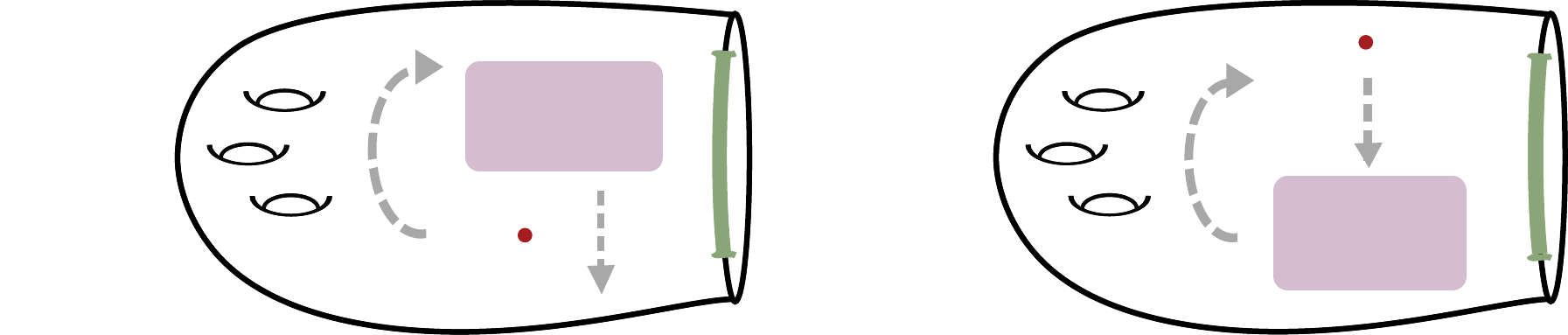}
\caption{The image $d^k([x_{k-1},k])$.}
\label{fBoundingfxi-Lie}
\end{figure}  
By construction: 
$$d^1(x_1)= d^{k-1}(x_{k-1})=t_{\psi(  \cdots \psi(\psi(\psi(a_1, a_2), a_3), a_4), \cdots, a_{k-1})}(y_0).$$ 
Hence the chain in Figure \ref{fBoundingfxi-Lie} is homologous to the chain in Figure \ref{finale}. In this figure we have negated the chain by reversing the direction of the arrow from clockwise to counterclockwise. 
\begin{figure}[!ht]    \centering
\labellist
 \hair 0pt
\pinlabel {\color{Maroon} $k$} at  187 27
\pinlabel {\tiny \color{Fuchsia} $d^{k-1}(x_{k-1})$} at  185 109
\endlabellist
\includegraphics[scale=.25]{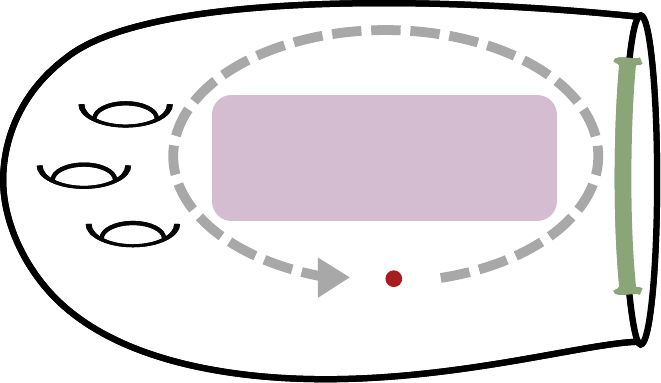}
\caption{The image $d^k([x_{k-1},k])$.}
\label{finale}
\end{figure}  
 Figure \ref{finale} concludes the induction on $k$, and the proof. 
\end{proof}

\subsection{Proof of secondary representation stability} \label{SectionMainProof}

In this subsection, we prove Theorem \ref{maintheorem},  secondary representation stability for the homology of configuration spaces. For this result we need to assume $R$ is a field of characteristic zero. The reason for this assumption is so that we can apply corollaries of Theorem \ref{NSSnoth}. Assuming that $R$ is a field also makes the formulation of Proposition \ref{E2filtration} easier, although workarounds do exist for general rings. We will also assume that the manifold $M$ has finite type. This implies that the homology groups of the ordered configuration spaces are finitely generated as abelian groups. The $\bigwedge \,(\Sym^2 R)$-modules which we will show  exhibit secondary representation stability are defined as follows.

\begin{definition}
Let $\W_i^M(S):=H_0^{\FI}\left(H_{\frac{|S|+i}{2}} \left( F (M) \right) \right)_{S}$. 
 Here we use the convention that fractional homology groups are $0$.  
\end{definition}

The collection of $\fS_k$-representations $\W_i^M(k)$ assemble to form a $\bigwedge\,(\Sym^2 R)$-module as follows. Let $(f, Q) \in \Hom_{FIM^+}(S, T)$ be a standard generator with  $f: S \to T$ an injective map and $Q=\{(x_1, y_1), \ldots, (x_d, y_d)  \}$ an oriented matching of the complement of the image. Let $f':S \m f(S)$ be the bijective map defined by $f$ and let $$f_*':H_{i + \frac{|S|}{2}} \left( F(M) \right)_S \m H_{i + \frac{|S|}{2}} \left( F(M) \right)_{f(S)}$$ be the map on homology induced by the FI-modules structure.  The element $(f,Q)$ acts on a homology class in $H_0^{\FI}(H_{i + |S|/2} \left( F(M) \right))_S$ by $t_{\psi(x_d,y_d)} \circ \ldots \circ  t_{\psi(x_1,y_1)}  \circ f_*'$, as shown in Figure \ref{Match}.  Although this map is defined on the homology of configuration spaces, it descends to a map on minimal generators $H_0^{\FI}(H_*(F(M)))$. The order of the composition factors $t_{\psi(x_i,y_i)}$ only affects the sign of the homology class: this sign is exactly what differentiates the category FIM$^+$ from the linearization of FIM.

To prove secondary representation stability we will need to better understand the algebraic structure on the $E^2$-page of the arc resolution spectral sequence. To that end, we now define a filtration on the top homology of the complex of injective words.

\begin{definition}
For $d$ and $b$ of the same parity, let $\T_d^b$ be the image of the natural map: $$\Ind^{\fS_d}_{\fS_b \times \fS_2 \times \ldots \times \fS_2} \Top _b\boxtimes \Lie_2 \boxtimes \ldots \boxtimes \Lie_2 \m \Top_d$$
\end{definition}

By Proposition \ref{EmbedTdb} below, we can identify the groups $\T_d^b$ with the groups: $$\Ind^{\fS_d}_{\fS_b \times \fS_2 \times \ldots \times \fS_2} \Top _b\boxtimes \Lie_2 \boxtimes \ldots \boxtimes \Lie_2.$$  For example, $\T_7^3$ is the span of the elements: 

$$  \Big\{ \quad [[a, b], c]\,[d,e]\,[ f,g], \quad  [[a, c], b]\,[d,e]\,[ f,g] \quad \Big\vert \quad \text{decompositions } [7] = \{ a, b, c\} \sqcup \{d,e\} \sqcup \{f,g\} \quad \Big\}. $$

\begin{proposition} \label{EmbedTdb} The map $\Ind^{\fS_d}_{\fS_b \times \fS_2 \times \ldots \times \fS_2} \Top _b\boxtimes \Lie_2 \boxtimes \ldots \boxtimes \Lie_2 \m \Top_d$ defining the group $\T_d^b$  is injective. 
\end{proposition}

\begin{proof}
We must show that the module $$\Ind^{\fS_d}_{\fS_b \times \fS_2 \times \ldots \times \fS_2} \Top _b\boxtimes \Lie_2 \boxtimes \ldots \boxtimes \Lie_2  \cong  \bigoplus_{ \substack{[d] = B \sqcup A_1 \sqcup A_2 \sqcup \cdots \sqcup A_{\frac{d-b}{2}}  \\ |B|=b, |A_i|=2} } \Top_B \otimes \Lie_{A_1} \otimes \cdots \otimes \Lie_{A_{\frac{d-b}{2}}} $$ 
injects into $\Top_d$. The summand indexed by the set decomposition $[d] = B \sqcup A_1 \sqcup A_2 \sqcup \cdots \sqcup A_{ \frac{d-b}{2} } $ embeds as the span of injective words: 
$$  \left\{ \quad L[a_1, b_1][a_2, b_2] \cdots \left[a_{\frac{d-b}{2}}, b_{\frac{d-b}{2}}\right]  \quad \mid \quad A_i=\{a_i, b_i \}, \quad L \in \Top_B \quad \right\}. $$ 
Given any element in the image of this summand -- viewed as a linear combination of injective words -- and given any word $w$ appearing as a term in this element, we can uniquely recover the decomposition $[d] = B \sqcup A_1 \sqcup A_2 \sqcup \cdots \sqcup A_{ \frac{d-b}{2} } $ by observing the order of the letters $[d]$ in $w$. Hence the intersection of the image of distinct summands is zero, and the map is injective as claimed. 
\end{proof}

\begin{proposition} \label{TbdSES}
There is a short exact sequence: $$0 \longrightarrow \T_d^b \longrightarrow \T_d^{b+2} \longrightarrow \Ind^{\fS_d}_{\fS_{b+2} \times \fS_2 \times \ldots \times \fS_2} (\Top_{b+2}/ \T_{b+2}^b) \boxtimes \Lie_2 \boxtimes \ldots \boxtimes \Lie_2 \longrightarrow 0.$$ 
\end{proposition}

\begin{proof} By Proposition \ref{EmbedTdb}, we may identify: 
\begin{align*} 
\T_d^b &\cong \Ind^{\fS_d}_{\fS_b \times \fS_2 \times \ldots \times \fS_2} \Top _b\boxtimes \Lie_2 \boxtimes \ldots \boxtimes \Lie_2  \\
& \cong \Ind^{\fS_d}_{\fS_{b+2} \times \fS_2 \times \ldots \times \fS_2}  \left( \Ind_{\fS_b \times \fS_2}^{\fS_{b+2}}  \Top _b\boxtimes \Lie_2 \right) \boxtimes \Lie_2  \boxtimes \ldots \boxtimes \Lie_2  \\ 
& \cong \Ind^{\fS_d}_{\fS_{b+2} \times \fS_2 \times \ldots \times \fS_2}  \left(\T_{b+2}^b \right) \boxtimes \Lie_2  \boxtimes \ldots \boxtimes \Lie_2  .
\\ \qquad \\ 
\text{ Moreover, }\T_d^{b+2} & \cong \Ind^{\fS_d}_{\fS_{b+2} \times \fS_2 \times \ldots \times \fS_2} \Top _{b+2}\boxtimes \Lie_2 \boxtimes \ldots \boxtimes \Lie_2. \end{align*}
\noindent Under these identifications, the map $\T_d^b \m \T_d^{b+2}$ is induced by the map $\T^b_{b+2} \m \T_{b+2}$. Since tensor product and induction are right-exact operations, from the short exact sequence 
$$ 0 \longrightarrow  \T_{b+2}^b \longrightarrow \Top_{b+2} \longrightarrow   (\Top_{b+2}/ \T_{b+2}^b ) \longrightarrow 0 \qquad \text{(exact by  Proposition \ref{EmbedTdb}) }$$ 
we can conclude that the sequence in question is exact at each point except possibly $\T_d^b$. But Proposition \ref{EmbedTdb} implies that the composition of the maps $ \T_d^b \longrightarrow \T_d^{b+2}  \longrightarrow \Top_d$ is injective.  This implies that the map $ \T_d^b \longrightarrow \T_d^{b+2}$ is injective, and we conclude that the sequence
$$0 \longrightarrow \T_d^b \longrightarrow \T_d^{b+2} \longrightarrow \Ind^{\fS_d}_{\fS_{b+2} \times \fS_2 \times \ldots \times \fS_2} (\Top_{b+2}/ \T_{b+2}^b) \boxtimes \Lie_2 \boxtimes \ldots \boxtimes \Lie_2 \longrightarrow 0$$
 is exact. 
\end{proof}

\begin{definition}
Let $E^r_{p,q}(k)$ denote entry $(p,q)$ on the $r$th page of the arc resolution spectral sequence for the set $[k]$. For $i \geq 0$, let $$A_j^i(k):= E^2_{2j-1,i-j+\lceil k/2 \rceil }(k).$$ The groups  $A_*^i(k)$  with the $d^2$ differential form a chain complex which we call the ``$i$th  even diagonal.'' For $i \geq 0$, let $$B_j^i(k):= E^2_{2j,i-j+\lceil k/2 \rceil }(k).$$ Call the chain complex $B_*^i(k)$ the ``$i$th odd diagonal.''
\end{definition}
Some examples of these complexes on $E^2_{p,q}(6)$ and are illustrated in Figure \ref{E2PageComplexes}.  $A_*^3(k)$ will always be the third (counting from $i=0$) diagonal above the triangle of zeroes, and similarly $B_*^1(k)$ is the first (counting from $i=0$) offset diagonal above the triangle of zeroes.

\begin{figure}[h!]    \centering \begin{tikzpicture} \tiny
  \matrix (m) [matrix of math nodes,
    nodes in empty cells,nodes={minimum width=3ex,
    minimum height=5ex,outer sep=2pt},
 column sep=3ex,row sep=3ex]{ 
 6    &  E^2_{-1, 6} & E^2_{0, 6}  & E^2_{1, 6}  & E^2_{2, 6}  &  E^2_{3, 6}  &  E^2_{4, 6} &  E^2_{5, 6}  &  E^2_{6, 6}  &  E^2_{7, 6}   \\  
 5    &  E^2_{-1, 5} & E^2_{0, 5}  & E^2_{1, 5}  & E^2_{2, 5}  &  E^2_{3, 5}  &  E^2_{4, 5}  &  E^2_{5, 5}  &  E^2_{6, 5}  &  E^2_{7, 5}   \\  
  4    &  E^2_{-1, 4} &  E^2_{0, 4}  & E^2_{1, 4}  & E^2_{2, 4}  &  E^2_{3, 4}  &  E^2_{4, 4}  &  E^2_{5, 4}  &  E^2_{6, 4}  &  E^2_{7, 4}   \\  
    3    &  E^2_{-1, 3} &  E^2_{0, 3}  & E^2_{1, 3}  & E^2_{2, 3}  &  E^2_{3, 3}  &  E^2_{4, 3}  &  E^2_{5, 3}  &  E^2_{6, 3}  &  E^2_{7, 3}   \\  
 2   &  0&  0 & E^2_{1, 2}  & E^2_{2, 2}  &  E^2_{3, 2}  &  E^2_{4, 2}  &  E^2_{5, 2}  &  E^2_{6, 2}  &  E^2_{7, 2}  &\qquad &\qquad & \\  
1   &  0 &  0 & 0  & 0 &  E^2_{3, 1}  &  E^2_{4, 1}  &  E^2_{5, 1}  &  E^2_{6, 1}  &  E^2_{7, 1} &&&  \\  
0    &  0 &  0  & 0 & 0 &  0  &  0  &  E^2_{5, 0}  &  E^2_{6, 0}  &  E^2_{7, 0} &&&   \\  
 \quad\strut &   -1  &  0  &  1  & 2 &3 & 4 & 5 & 6 & 7 & & & \\}; 

\draw[thick] (m-1-1.east) -- (m-8-1.east) ;
\draw[thick] (m-8-1.north) -- (m-8-10.north east) ;

  \draw[-stealth, red] (m-2-4) -- (m-1-2) ;
  \draw[-stealth, red] (m-3-6) -- (m-2-4) ;
\draw[-stealth, red] (m-4-8) -- (m-3-6) ;
\draw[-stealth, red] (m-5-10) -- (m-4-8) ;
\draw[-stealth, red] (m-6-12) -- (m-5-10) ;

  \draw[-stealth, ultra thick, blue] (m-3-3) -- (m-2-1) ;
  \draw[-stealth, ultra thick, blue] (m-4-5) -- (m-3-3) ;
  \draw[-stealth, ultra thick, blue] (m-5-7) -- (m-4-5) ;
\draw[-stealth, ultra thick,blue] (m-6-9) -- (m-5-7) ;
\draw[-stealth,  ultra thick, blue] (m-7-11) -- (m-6-9) ;

\end{tikzpicture}
\caption{The complexes  $A_*^3(6)$ (differentials in red) and $B_*^1(6)$ (in blue bold) on $E^2_{p,q}(6)$. } \label{E2PageComplexes} 
\end{figure}
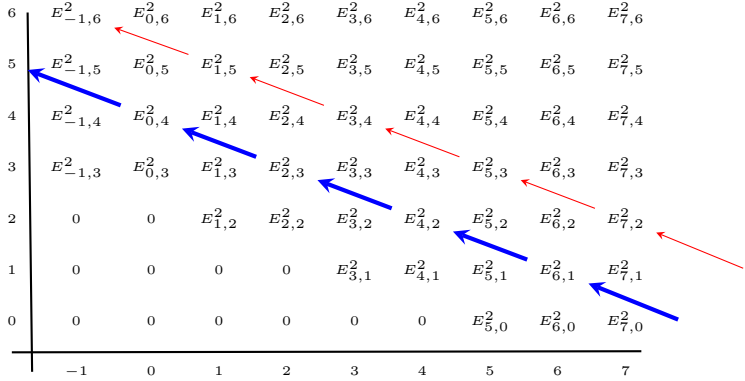

We now relate the chain complexes $A_*^i(k)$ and $B_*^i(k)$ to the chain complexes $\Inj^2_*(\W^M_j)$. Note that $A_*^i(k)$ and $B_*^i(k)$ are $0$ for $*<0$ in contrast to $\Inj^2_*(\W^M_j)$, which is potentially nonzero for $*=-1$. To simplify indexing in the following proposition, we introduce two $\bigwedge \,(\Sym^2 R)$-modules: 
$$ \qquad \mathcal V_i := \mathcal W^M_{2i} \oplus \mathcal W^M_{2i+1} \qquad \text{ and } \qquad \mathcal U_i := \mathcal W^M_{2i+1} \oplus \mathcal W^M_{2i+2}.$$   
Concretely,  
$$ \mathcal V_i(k) = H_0^{\FI} \left(H_{i+\left\lceil \frac{k}{2} \right\rceil } (F(M))\right)_k \qquad \text{and} \qquad   
\mathcal U_i(k) = H_0^{\FI} \left(H_{i+\left\lceil \frac{k+1}{2} \right\rceil } (F(M))\right)_k.$$

\begin{proposition} \label{E2filtration} Suppose $R$ is a field. The chain complex $A_*^i(k)$ has a filtration by chain complexes such that the filtration quotients are isomorphic to
$$ \Ind^{\fS_k}_{\fS_{2b} \times \fS_{k-2b} }   (\Top_{2b}/ \T_{2b}^{2b-2})  \boxtimes \Big( \Inj^2_{*-b-1 } \mathcal V_i \Big)_{k-2b} .$$

The chain complex $B_*^i(k)$ has a filtration such that the filtration differences are isomorphic to
$$ \Ind^{\fS_k}_{\fS_{2b+1} \times \fS_{k-2b-1} }   (\Top_{2b+1}/ \T_{2b+1}^{2b-1})  \boxtimes \Big( \Inj^2_{*-b-1} \; \mathcal U_i \Big)_{k-2b-1}. $$ These groups are chain complexes with the differential induced by the second factor and these isomorphism are isomorphisms of chain complexes.

\end{proposition}

\begin{proof}

By Proposition \ref{TbdSES}, there is a filtration of $\Top_d$ given by 
\begin{align*} 0 \hookrightarrow   \T^0_d\hookrightarrow  \cdots \hookrightarrow   \T^{d-4}_d   \hookrightarrow   \T^{d-2}_d  \ \hookrightarrow     \T^d_d = \Top_d &&\text{ when $d$ is even, and } \\ 
0 = \T^1_d \hookrightarrow   \T^3_d\hookrightarrow   \cdots \hookrightarrow   \T^{d-4}_d   \hookrightarrow   \T^{d-2}_d  \ \hookrightarrow     \T^d_d = \Top_d  && \text{ when $d$ is odd,}
\end{align*}
 whose quotients are the groups: 
$$  \frac{ \T_{d}^{b} }{ \T_{d}^{b-2}  } \cong \Ind^{\fS_d}_{\fS_{b} \times \fS_2 \times \ldots \times \fS_2} (\Top_{b}/ \T_{b}^{b-2}) \boxtimes \Lie_2 \boxtimes \ldots \boxtimes \Lie_2.$$

\noindent Since $R$ is a field, it follows that 
$$ E^2_{p,q}(k) \cong  \bigoplus_{\substack{[k] = P \sqcup R, \\ |P|=p+1}}  \Top_{p+1}(P) \otimes H_0^{\FI} (H_q( F(M)))_{k-p-1}
$$ 
is filtered by the modules 
$$ \bigoplus_{\substack{[k] = P \sqcup R, \\ |P|=p+1}}  \T^b_{p+1}(P) \otimes H_0^{\FI} (H_q(F(M)))_{k-p-1} \qquad \text{for $b \equiv p+1 \pmod2$} $$ 
with filtration quotients 
\begin{align*}
&\bigoplus_{\substack{[k] = P \sqcup R, \\ |P|=p+1}}  \Big( \Ind^{\fS_{p+1}}_{\fS_{b} \times \fS_2 \times \ldots \times \fS_2} (\Top_{b}/ \T_{b}^{b-2}) \boxtimes \Lie_2 \boxtimes \ldots \boxtimes \Lie_2 \Big)\otimes H_0^{\FI} (H_q(F(M)))_{k-p-1} \\
&= \Ind^{\fS_k}_{\fS_{b} \times \fS_2 \times \ldots \times \fS_2 \times \fS_{k-p-1}} (\Top_{b}/ \T_{b}^{b-2}) \boxtimes  \Lie_2 \boxtimes  \ldots \boxtimes  \Lie_2\boxtimes H_0^{\FI} (H_q (F(M)))_{k-p-1}\\
& = \Ind^{\fS_k}_{\fS_{b} \times \fS_{k-b} }   (\Top_{b}/ \T_{b}^{b-2})  \boxtimes \Big( \Ind^{\fS_{k-b}}_{ \fS_2 \times \ldots \times \fS_2 \times \fS_{k-p-1}} \Lie_2 \boxtimes \ldots \boxtimes \Lie_2 \boxtimes H_0^{\FI} (H_q( F(M)))_{k-p-1} \Big).
\end{align*}

\noindent This means that the filtration differences for $A_j^i(k)= E^2_{2j-1,i-j+\lceil k/2 \rceil }(k)$ are given by

\begin{align*}
&  \Ind^{\fS_k}_{\fS_{b} \times \fS_{k-b} }   (\Top_{b}/ \T_{b}^{b-2})  \boxtimes \Big( \Ind^{\fS_{k-b}}_{ \fS_2 \times \ldots \times \fS_2 \times \fS_{k-2j} } \Lie_2 \boxtimes \ldots \boxtimes \Lie_2 \boxtimes H_0^{\FI} (H_{i-j+\lceil k/2 \rceil } (F(M)))_{k-2j} \Big) \qquad \text{($b$ even)} \\ 
&=  \Ind^{\fS_k}_{\fS_{b} \times \fS_{k-b} }   (\Top_{b}/ \T_{b}^{b-2})  \boxtimes \Big( \Inj^2_{j-\frac{b}{2}-1 } \mathcal V_i \Big)_{k-b}. \end{align*}
Similarly, the filtration differences for  $B_j^i(k)= E^2_{2j,i-j+\lceil k/2 \rceil }(k)$ are given by
\begin{align*}
& \Ind^{\fS_k}_{\fS_{b} \times \fS_{k-b} }   (\Top_{b}/ \T_{b}^{b-2})  \boxtimes \Big( \Ind^{\fS_{k-b}}_{ \fS_2 \times \ldots \times \fS_2 \times \fS_{k-2j-1}} \Lie_2 \boxtimes \ldots \boxtimes \Lie_2 \boxtimes H_0^{\FI} (H_{i-j+\lceil k/2 \rceil} (F(M)))_{k-2j-1} \Big) \qquad \text{($b$ odd)} \\
&= \Ind^{\fS_k}_{\fS_{b} \times \fS_{k-b} }   (\Top_{b}/ \T_{b}^{b-2})  \boxtimes \Big( \Inj^2_{j-\frac{b}{2}-\frac12} \;  \mathcal U_i \Big)_{k-b}.
\end{align*}

For simplicity, we will reindex these filtrations by replacing odd values of $b$ with $2b+1$ and even values of $b$ with $2b$.

Let $\mathcal F_b(A_*^i)$ be the portion of the filtration of $A_*^i$ constructed above containing elements of the form $ \Ind^{\fS_k}_{\fS_{2b} \times \fS_{k-2b} }   \Top_{2b}  \boxtimes ( \Inj^2_{j-b-1 } \mathcal V_i )_{k-2b}$ and similarly define  $\mathcal F_b(B_*^i)$. We have constructed filtrations on the groups $A_*^i$ and $B_*^i$, and now it remains to check that these are filtrations of chain complexes. We must also verify that the boundary maps on the filtration quotients induced by the $d^2$ differential in the spectral sequence agree with the boundary maps of $$ \Ind^{\fS_k}_{\fS_{2b} \times \fS_{k-2b} }   (\Top_{2b}/ \T_{2b}^{2b-2})  \boxtimes( \Inj^2_{*} \mathcal V_i )_{k-2b} \quad \text{and} \quad \Ind^{\fS_k}_{\fS_{2b+1} \times \fS_{k-2b-1} }   (\Top_{2b+1}/ \T_{2b+1}^{2b-1})  \boxtimes ( \Inj^2_{*} \; \mathcal U_i )_{k-2b-1} .$$ 

First we will show that the subgroups $\mathcal F_b(A_*^i)$ are in fact subchain complexes. An element of $\mathcal F_b(A_j^i)(k)$  can be written as a sum of elements of the form: $$t \boxtimes l_0 \boxtimes \ldots \boxtimes l_{j-b} \boxtimes v \qquad \text{ with $t \in \Top_{2b}$, $l_q \in \Lie_2$ and $v \in H_0^{\FI}(H_{i-j+\lceil \frac{k-1}{2} \rceil}(F(M)))_{2j-1}$.}$$
 Moreover, we may assume that $t$ is a product of Lie polynomials (and not  a linear combination of products of Lie polynomials). 
By the Leibniz rule (Lemma \ref{LemmaLeibniz}) and our calculations of the differentials in the arc resolution spectral sequence from  Lemma \ref{LemmaDifferentialsOnLie}, it follows that the differential $$d^2(t \boxtimes l_0 \boxtimes \ldots \boxtimes l_{j-b} \boxtimes v)$$ is given by a signed sum of terms which remove one $\Lie_2$ factor and then stabilize $v$ by the appropriate Browder operation. Note there is no nonzero term involving applying $d_2$ to $v$ since $v$ corresponds to an element of the $-1$st column of the arc resolution spectral sequence. Since all of the terms in the sum are in $\mathcal F_b(A_{j-1}^i)(k)$, this establishes that that $\mathcal F_b(A_*^i)$ is a filtration of chain complexes. There are two types of terms in the signed sum, the first involves removing an $\Lie_2$ factor from $t$ and the second involve deleting one of the $l_q$ factors. The portion of the sum involving terms of the second type is exactly the boundary map of the chain complex $$\Ind^{\fS_k}_{\fS_{2b} \times \fS_{k-2b} }   \Top_{2b}  \boxtimes ( \Inj^2_{*} \; \mathcal V_i )_{k-2b}.$$ Thus, it suffices to show that the portion of the sum involving terms of the first type are in $\mathcal F_{b-1}(A_{j-1}^i)(k)$, and hence zero in the quotient. But removing an $\Lie_2$ factor from $t$ yields a Lie polynomial that is two letters shorter and hence in $\Top_{2b-2}$.  A similar argument works for $B_*^i$. 
\end{proof}

The following result shows that vanishing on the $E^3$-page of the arc resolution spectral sequence implies secondary representation stability.

\begin{proposition} \label{H_0ForFIPairs} There are isomorphisms of FB-modules:  
$$H_0(A_*^i) \cong H_0^{\FIM^+}(\mathcal \cV_i).$$ 
There are isomorphisms of symmetric group representations:
$$H_0 (A_*^{i})(2k) \cong H_0^{\FIM^+}(\mathcal W_{2i}^M)(2k) \qquad  \text{ and } \qquad  H_0 (A_*^{i})(2k+1) \cong H_0^{\FIM^+}(\mathcal W_{2i+1}^M)(2k+1).$$

\end{proposition}

\begin{proof} Let $\mathcal Q$ be a $\bigwedge \,(\Sym^2 R)$-module. In analogy to Proposition \ref{minusonehomology}, there is an isomorphism $$H_{-1}(\Inj^2_*(\mathcal Q)) \cong H_0^{\FIM^+}(\mathcal Q).$$
 By Proposition \ref{E2filtration}, the map $\Inj^2_{*-1}(\mathcal V_i) \m A_*^i$ is an isomorphism for $* =0,1$. Thus, $$H_{-1}(\Inj^2_*(\mathcal V_i)) \cong H_0(A_*^i).$$ 

 The second pair of isomorphisms follow from the fact that $\mathcal V_i \cong \mathcal W_{2i}^M \oplus  \mathcal W_{2i+1}^M$. 
\end{proof}

We now prove the main theorem: if $M$ is a finite type noncompact connected manifold, and $R$ is a field of characteristic zero, then $\W_i^M$ is a finitely generated $\bigwedge \,(\Sym^2 R)$-module. For convenience, we will also assume $M$ is smooth but see Remark \ref{topologicalmanifolds} for a discussion of the case of general topological manifolds.

\begin{proof}[Proof of Theorem \ref{maintheorem}] 

We will prove by induction that the  $\bigwedge \,(\Sym^2 R)$-modules $\W^M_i$ are finitely generated. Because the homology of configuration spaces of manifolds with finite type homology is finitely generated (as abelian groups), it suffices to show that these $\bigwedge \,(\Sym^2 R)$-modules have finite generation degree.

Before we proceed, we make some preliminary observations. By combining Proposition \ref{connected2}, Proposition \ref{E2filtration}, and considering the spectral sequence associated to a filtered chain complex, we conclude the following results: 
\begin{enumerate}
\item[(a)] If $\cU_i$ is  finitely generated for $i \leq d$, then for all $j$, 
$$H_j(B^i_*(k)) \cong 0 \qquad \text{ for all $i \leq d$, and all $k$ sufficiently large (depending on $i$, $j$, and $M$).} $$ 
\item[(b)]   If $\cV_i$ is  finitely generated for $i \leq d$,  then for all $j$, 
$$H_j(A^i_*(k)) \cong 0 \qquad \text{for all $i \leq d$, and all $k$ sufficiently large (depending on $i$, $j$, and $M$).} $$ 
\end{enumerate}

The proof of Theorem \ref{ConfigSpaceRepStable} implies that $\W_i^M$ vanishes for $i$ strictly negative. This will be the base case of the following two-step induction argument. In the first part of the induction argument, we will assume that $\W_i^M$ is finitely generated for $i \leq 2m$ and then prove that $\W^M_{2m+1}$ is finitely generated. This induction hypothesis is equivalent to the statement that $\cU_i$ and $\cV_i$ are finitely generated for $i \leq m-1$. By Proposition \ref{H_0ForFIPairs},  the conclusion is equivalent to the statement that $H_0(A_*^{m}(k)) \cong 0$ for sufficiently large odd $k$. In the second part of the induction argument, we will assume that $\W_i^M$ is finitely generated for $i \leq 2m+1$ and then prove that $\W_{2m+2}^M$ is finitely generated. This induction hypothesis is equivalent to the statement that $\cV_i$ is finitely generated for $i \leq m$ and $\cU_i$ is finitely generated for $i \leq m-1$, and the conclusion is equivalent to the statement that $H_0(A_*^{m+1}(k)) \cong 0$ for sufficiently large even $k$.

\paragraph*{First induction step:}
Assume $\W_i^M$ is finitely generated for $i \leq 2m$. Our goal is to show $H_0(A_*^{m}(k)) \cong 0$ for odd $k$ sufficiently large. By definition, when $k$ is odd,  $$A_0^m(k) \cong E^2_{-1,m+\frac{k+1}{2} }(k) \qquad \text{and} \qquad H_0(A_*^m(k)) \cong E^3_{-1,m+\frac{k+1}{2} }(k).$$ 

\noindent Since the connectivity of the arc resolution is $(k-1)$ by Proposition \ref{ArcF}, we know that for large $k$, $$E^\infty_{-1,m+\frac{k+1}{2}}(k)  \cong  0.$$ There are no differentials out of $E^r_{-1,m+\frac{k+1}{2}}(k)$ so to prove that $H_0(A_*^i(k))$ vanishes we will show that there are no nonzero differentials $d^r$ into this group for $r>2$. The domains of such differentials are $E^r_{-1+r,m-r+1+\frac{k+1}{2} }(k)$  for $r>2$. When $r\geq 2m+4$, the proof of Theorem  \ref{ConfigSpaceRepStable} implies $E^2_{-1+r,m-r+1+\frac{k+1}{2} }(k) \cong 0$. The differentials are shown in the case $m=1, k=7$ in Figure \ref{E3VanishesOdd}. 

\begin{figure}[h!]    \centering \begin{tikzpicture} \scriptsize
  \matrix (m) [matrix of math nodes,
    nodes in empty cells,nodes={minimum width=3ex,
    minimum height=5ex,outer sep=2pt},
 column sep=3ex,row sep=3ex]{ 
5    &  E^2_{-1, 5} &  E^2_{0, 5}  & E^2_{1, 5}  & E^2_{2, 5}  &  E^2_{3, 5}  &  E^2_{4, 5}  &  E^2_{5, 5}  &  E^2_{6, 5}  &  E^2_{7, 5} &  E^2_{8, 5}    \\  
  4    &  E^2_{-1, 4} &  E^2_{0, 4}  & E^2_{1, 4}  & E^2_{2, 4}  &  E^2_{3, 4}  &  E^2_{4, 4}  &  E^2_{5, 4}  &  E^2_{6, 4}  &  E^2_{7, 4}&  E^2_{8, 4}     \\  
    3    &  0 &  E^2_{0, 3}  & E^2_{1, 3}  & E^2_{2, 3}  &  E^2_{3, 3}  &  E^2_{4, 3}  &  E^2_{5, 3}  &  E^2_{6, 3}  &  E^2_{7, 3}   &  E^2_{8, 3} \\  
 2   &  0&  0 & 0 & E^2_{2, 2}  &  E^2_{3, 2}  &  E^2_{4, 2}  &  E^2_{5, 2}  &  E^2_{6, 2}  &  E^2_{7, 2} &  E^2_{8, 2} &\qquad &\qquad & \\  
1   &  0 &  0 & 0  & 0 &  0 &  E^2_{4, 1}  &  E^2_{5, 1}  &  E^2_{6, 1}  &  E^2_{7, 1}&  E^2_{8, 1}  &&&  \\  
0    &  0 &  0  & 0 & 0 &  0  &  0  &  0 &  E^2_{6, 0}  &  E^2_{7, 0} &E^2_{8, 0}  &&&   \\  
 \quad\strut &   -1  &  0  &  1  & 2 &3 & 4 & 5 & 6 & 7 & 8 & & & \\}; 

\draw[thick] (m-1-1.east) -- (m-7-1.east) ;
\draw[thick] (m-7-1.north) -- (m-7-11.north east) ;

  \draw[-stealth, ultra thick, red] (m-2-4) -- (m-1-2) ;
  \draw[-stealth,  ultra thick, red] (m-3-6) -- (m-2-4) ;
\draw[-stealth, ultra thick, red] (m-4-8) -- (m-3-6) ;
\draw[-stealth, ultra thick, red] (m-5-10) -- (m-4-8) ;
\draw[-stealth, ultra thick, red] (m-6-12) -- (m-5-10) ;

  \draw[-stealth, green] (m-3-5) -- (m-1-2) ;
    \draw[-stealth, green] (m-4-6) -- (m-1-2) ;
    \draw[-stealth, green] (m-5-7) -- (m-1-2) ;
      \draw[-stealth, green] (m-6-8) -- (m-1-2) ;

\end{tikzpicture}
\caption{The complex  $A_*^1(7)$ and the differentials $d^3$, $d^4$, $d^5$, and $d^6$ (shown in green). } \label{E3VanishesOdd} 
\end{figure}
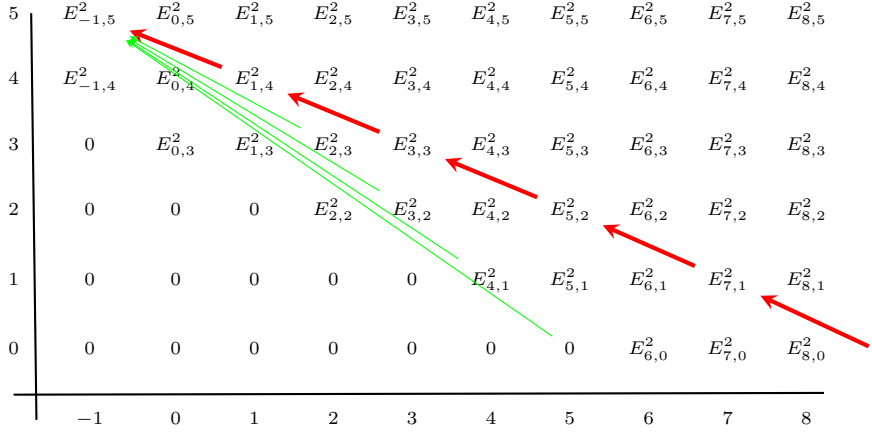  

For $2<r < 2m+4$,  the groups $E^3_{-1+r,m-r+1+\frac{k+1}{2} }(k)$ are of the form $H_j(A^i_*(k))$ for $i \leq m-1$ and  $2 \leq j \leq m+1$ or of the form $H_j(B^i_*(k))$  for $ i \leq m$ and  $1 \leq j \leq m+1$. Thus by observations (a) and (b), there is some uniform bound $K \in \Z$ such that these groups are all zero for $k>K$.  Hence for large $k$ and $r>2$, there cannot be nonzero differentials with codomain $E^r_{-1,m+\frac{k+1}{2} }(k)$. It follows that, for $k$ sufficiently large and odd,  $$H_0(A_*^m(k)) \cong E^3_{-1,m+\frac{k+1}{2} }(k) \cong E^\infty_{-1,m+\frac{k+1}{2} }(k) \cong 0.$$ This establishes the first induction step.

\paragraph*{Second induction step:}
Assume $\W_i^M$ is finitely generated for $i \leq 2m+1$. Our goal is to show $H_0(A_*^{m+1}(k)) \cong 0$ for even $k$ sufficiently large. When $k$ is even, 
$$A_0^{m+1}(k) \cong E^2_{-1,m+1+\frac{k}{2} }(k) \qquad \text{and} \qquad H_0(A_*^{m+1}(k)) \cong E^3_{-1,m+1+\frac{k}{2} }(k).$$ 
Again, it suffices to show that $E^3_{-1,m+1+\frac{k}{2} }(k)$ is not the target of any nonzero differentials $d^r$, $r>2$, once $k$ is sufficiently large. The domains of the only possibly nonzero differentials are  $E^r_{-1+r,m-r+2+\frac{k}{2} }(k)$  for $2<r<2m+5$.  But  for $2<r<2m+5$ the groups $E^3_{-1+r,m-r+2+\frac{k}{2} }(k)$ are one of $H_j(A^i_*(k))$ for $i \leq m$ and $2 \leq j \leq m+2$,  or $H_j(B^i_*(k))$  for $ i \leq m+1$ and $1 \leq  j \leq m+1$. By observations (a) and (b) these groups vanish for large even $k$. This establishes the second induction step and the theorem. 
 \end{proof}

\begin{remark}
To prove secondary representation stability, one only needs information about $d^2$-differentials in the arc resolution spectral sequence. This information can be obtained by using the Leibnitz rule if one can compute $d^2: E^2_{1,0}[\R^n](2)  \m E^2_{-1,1}[\R^n](2) $. This can be computed in the following alternative way. The map is surjective since $E^\infty_{-1,1}[\R^n](2) \cong 0$. Since $H_1(F_1( \R^n )) \cong 0$, $E^2_{-1,1}[\R^n](2) \cong H_1(F_2(\R^n))$. One can then compute this differential from the fact that $F_2(\R^n) \simeq S^{n-1}$ with $H_{n-1}(F_2(\R^n))$ generated by the Browder operation applied to two copies of the generator of $H_0(F_1(\R^n))$. We included calculations of higher differentials since they will be needed in the next subsection to establish an improved stable range in higher dimensions and are suggestive of even higher order stability for surfaces.

\end{remark}

\subsection{Improved range in higher dimensions} \label{SectionHigherDimensions}
 
Although Theorem \ref{maintheorem} holds for manifolds of dimension $n \geq 3$, the result in higher dimensions is degenerate: the homology operation $\psi$ is zero for $n \geq 3$, and the isomorphism of Corollary \ref{corSecondaryCentral} is also the zero map. Thus, in high dimensions secondary representation stability manifests itself as an improved range for representation stability. We begin by showing that the arc resolution spectral sequence collapses at the $E^2$-page if $\dim M >2$. In this subsection, we work with integral coefficients.

\begin{proposition}\label{collapse}
If $M$ is a noncompact connected smooth manifold of dimension $n \geq 3$, the arc resolution spectral sequence collapses at the $E^2$-page.
\end{proposition}

\begin{proof} Since $F_2(\R^n) \simeq S^{n-1}$, the class $\psi(1,2) \in H_1(F_2(\R^n))$ is zero for $n\geq 3$. Since Browder operations are bilinear, the iterated product  $\psi(1,\psi(2,\ldots , \psi(k-1,k) \ldots ) $ vanishes in $H_{k-1}(F_k(\R^n))$. In particular, $t_{\psi(1,\psi(2,\ldots , \psi(k-1,k) \ldots ) }$ is the zero map. By Lemma \ref{LemmaDifferentialsOnLie}, for $L=[1,\ldots[k-1,k],\ldots] \in E^1_{k-1,0}(S)$, the differential $d^r(L)$ vanishes for $r<k$ and $$d^k(L)= t_{\psi(1,\psi(\ldots ,\psi(k-1,k)\ldots))}(y_0)$$ where $y_0$ is the class of a point in $H_0(F_0(M))$. Thus $d^k(L)=0$ as well. For degree reasons, for $r>k$ the codomain of the differential $d^r$ is zero, and so $d^r(L)=0$. 
 
 Now consider $T \in E^1_{k-1,0}(S)$. By Theorem \ref{TopHomology}, $T=L_1 \boxtimes \ldots \boxtimes L_m$ with $L_i$ Lie polynomials. By Lemma \ref{LemmaLeibniz}, $d^r(T)$ is a signed sum of products of $d^r(L_i)$. These terms vanish by the above paragraph so $d^r(T)=0$ for all $r \geq 2$.

 Consider $r \geq 2$ and assume by induction that we have shown that $d^t=0$ for all $2 \leq t < r$. Thus $E^2_{p,q} \cong E^r_{p,q}$. In the language of Lemma \ref{LemmaLeibniz}, Proposition \ref{GeometricRealizationSS} implies that $E^2_{p,q}[M](S)$ is generated by classes of the form $t^r(T \otimes \alpha)$ with $T \in E^1_{p,0}[\R^n](U)$ and $\alpha \in E_{-1,q}^2[M](S \setminus U)$ for $U \subseteq S$ a subset of size $p+1$. Then $d^r(T)=0$ by the above paragraph and $d^r(\alpha)=0$ since $\alpha$ is in the $-1$st column. Thus, $d^r(t^r(T \otimes \alpha))=0$ and we have shown that $d^r=0$. The claim follows by induction. 
\end{proof}

Using Proposition \ref{collapse}, we can prove an improved stable range for the homology of configuration spaces of higher-dimensional manifolds. This result was proven by Church--Ellenberg--Farb for noncompact connected orientable manifolds \cite[Theorem 6.4.3]{CEF}, and we extend their result to all noncompact connected manifolds. 

\begin{theorem} \label{ImprovedRangeTheorem} Let $M$ be a noncompact connected smooth manifold of dimension at least three. Then $\deg  H_0^{\FI}(H_i(F(M); \Z)) \leq i$.
\end{theorem}

\begin{proof}
Consider the spectral sequence described in Proposition \ref{GeometricRealizationSS}. We proved in Proposition \ref{ArcF} that $E^\infty_{p,q}(S)  \cong  0$ for $p+q \leq |S|-2$. Proposition \ref{collapse} implies that $E^\infty_{p,q}(S) \cong E^2_{p,q}(S)$ for all $p$ and $q$. Since $H_0^{\FI}(H_i(F(M)))_{S} \cong E^2_{-1,i}(S)$, the claim follows. 
\end{proof}

\subsection{Conjectures and calculations} \label{secConj}
 
In this subsection, we make several conjectures. We give evidence for some of these conjectures by proving them in special cases. 

\subsubsection*{Higher-dimensional manifolds and higher order stability}

We begin with some questions concerning configuration spaces. 

\begin{question} \label{QuestionHigherManifolds}   \begin{enumerate}
\item[(a)] Is there a notion of tertiary and higher order representation stability that is present in the homology of configuration spaces?
\item[(b)] What is the stable range for secondary representation stability?
\item[(c)] Is there any form of nontrivial secondary representation stability for configuration spaces of higher-dimensional manifolds?
\end{enumerate} 
\end{question} 

We suggest a conjectural answer to all three questions. Its statement requires the following definition.
 
\begin{definition} Define the following twisted (skew-)commutative algebras: 
$$ \mathfrak L^n_d := \left\{ \begin{array}{ll} \Sym \, H_{(d-1)(n-1)}(F_d(\R^n)), &  \text{ $(d-1)(n-1)$ even} \\ 
\bigwedge H_{(d-1)(n-1)}(F_d(\R^n)), &  \text{ $(d-1)(n-1)$ odd.} 
\end{array} \right. $$ 
\end{definition}

Note that $H_{(d-1)(n-1)}(F_d(\R^n))$ is $\Lie_d$ when $(d-1)(n-1)$ is odd. For $(d-1)(n-1)$ even, there is a similar description except with different signs. For $d=1$ and $n$ arbitrary, $\mathfrak L_d^n$-modules are precisely FI-modules. For $d=2$, these $\mathfrak L_d^n$-modules are modules over $\bigwedge (\mathrm{Sym}^2 R)$ if $n$ is even and modules over $ \mathrm{Sym}(\bigwedge^2 R)$ if $n$ is odd.

For a noncompact $n$-manifold $M$, the embedding $\R^n \sqcup M \hookrightarrow M$ induces maps 
$$H_{(d-1)(n-1)}(F_d(\R^n)) \otimes H_i(F_k(M)) \longrightarrow H_{i+(n-1)(d-1)}(F_{k+d}(M)).$$ For $d=1$, this gives the FI-module structure on $H_i(F(M))$ and for $d=n=2$, this gives the $\bigwedge \,(\Sym^2 R)$-module structure on $\W^M_i$. In general these embeddings induce $\mathfrak L^n_d$-module structures on the groups $\W[d]^M_i(S)$ defined as follows. 
 
\begin{definition} Let $M$ be a noncompact connected manifold of dimension $n$ and let 
\[ \W[d]^M_i(S): = H_0^{\mathfrak L^n_{d-1}}\left( \ldots \left( H_0^{\mathfrak L^n_{1}}\left( H_{\frac{(n-1)(d-1)|S|+i}{d}} \left( F(M) ;R\right) \right) \right) \ldots \right)_S \]
\end{definition}

Note that we use the $\mathfrak L^n_d$-module structure on $\W[d]^M_i$ to define $\W[d+1]^M_i(S)$. For $d=1$, $\W[d]^M_i$ is just the FI-module $H_i(F(M))$. For $d=n=2$, $\W[d]^M_i$ is the $\bigwedge \,(\Sym^2 R)$-module $\W^M_i$. We conjecture that these modules have finite generation degree when $M$ is sufficiently highly connected, and we conjecture an explicit stable range.
 
\begin{conjecture}
Let $M$ be a noncompact manifold of dimension $n \geq 2$. If $M$ is $q$-connected with $q \geq \left \lfloor \frac{(n-1)(d-1)}{d} \right \rfloor $, then $H_0^{\mathfrak L_d^n}\left( \W[d]^M_i\right)(S) \cong 0$ for $$|S|>\max \left( \frac{i(d^2+d)}{n-1},\frac{i d}{q d -(n-1)(d-1)} \right).$$

\label{highdimconj}
\end{conjecture}
 
The above conjecture can be interpreted as three separate conjectures, addressing the three parts of Question \ref{QuestionHigherManifolds}.  Note that $\left \lfloor  ((n-1)(d-1))/d \right \rfloor =0$ when $n=2$ and thus for surfaces we are only assuming that the manifold is connected. Our heuristic for assuming that the manifold needs to be $\left \lfloor((n-1)(d-1))/d \right \rfloor $-connected is to bound the \emph{slope} of certain homology classes, that is, the  ratio of homological degree to the number of moving points. This condition seems to ensure that the slope of all homology classes in the configuration space that ``come from the topology of the manifold'' is higher than those coming from $H_{d-1}(F_d(\R^n))$. As support for Conjecture \ref{highdimconj}, we will next prove the result in the case that the manifold is $\R^n$. From now on, we work with integral coefficients. 

\subsubsection*{Configurations of (punctured) Euclidean space}

Cohen  \cite[Chapter III]{CLM} proved that the homology groups $H_*(F_k(\R^n))$ are the submodule of the free graded commutative algebra on the free graded lie algebra on the set $[k]$ such that in each product of brackets, every element of $[k]$ appears exactly once. The bracket is the $E_n$-Browder operation $\psi^n$ and the product is $\bullet$ (see Theorem \ref{PropGerstenhaber}). For example, a typical element of  $H_{3(n-1)}(F_6(\R^n))$ is $2 \bullet \psi^n(1,4)  \bullet \psi^n(3,\psi^n(5,6))$.

\begin{proposition}\label{calculationRn}
 Conjecture \ref{highdimconj} holds for $M=\R^n$ with integral coefficients. Specifically, for $n>1$, $$H_0^{\mathfrak L_d^n}\left( \W[d]^{\R^n}_i\right)(S) \cong 0 \qquad \text{ for all $|S|>\frac{i(d^2+d)}{n-1}.$} $$
\end{proposition}

\begin{proof}

Cohen's description of $H_*(F_k(\R^n))$ allows us to compute the groups $\W[d]^{\R^n}_i$ explicitly. The $\mathcal L_{d}^n$-module structure on $\W[d]^{\R^n }_i$ is induced by stabilizing by $(d-1)$ nested $E_n$-Browder operations. Thus, the $\mathcal L_{d}^n$ generators $H_0^{\mathfrak L_d^n}\left( \W[d]^{\R^n}_i\right)(S)$ are spanned by products of $d$ or more nested Browder operations. The ratio of homological degree to number of points for these classes is at least $\frac{(n-1)(d) }{d+1}$; see Theorem \ref{PropGerstenhaber}. The group $H_0^{\mathfrak L_d^n}\left( \W[d]^{\R^n}_i\right)(S)$ is defined as a subquotient (and in fact, for $M=\R^n$, is a submodule) of the homology group $H_{i+\frac{(n-1)(d-1)|S|}{d}  } (F_S(\R^n))$. Thus, elements of $H_0^{\mathfrak L_d^n}\left( \W[d]^{\R^n}_i\right)(S)$ have a ratio of homological degree to number of points given by $\frac{i}{|S|}+\frac{(n-1)(d-1)}{d}$. If 
$$\frac{i}{|S|}+\frac{(n-1)(d-1)}{d} < \frac{(n-1)(d) }{d+1},$$
then the set of abelian group generators for $H_0^{\mathfrak L_d^n}\left( \W[d]^{\R^n}_i\right)(S)$ is empty and so $H_0^{\mathfrak L_d^n}\left( \W[d]^{\R^n}_i\right)(S)$ vanishes in the indicated range.  
\end{proof}

Cohen's calculation completely determine the modules $\W^{\R^2}_i$. We describe the case of $i=0$ in detail. 

\begin{proposition} \label{PropWR20}
 For the plane $M = \R^2$, $\W^{\R^2}_0 \cong M^{\FIM^+}(0).$ Notably, $\W^{\R^2}_0(2k) =  H_0^{\FI}(H_k(F(\R^2)))_{2k}$ is a rank--$\displaystyle \left( \frac{(2k)!}{k!2^k}\right)$ free module. 
\end{proposition}

\begin{proof}

For any $k>0$ the generators $H_0^{\FI} (H_i(F(\R^2)))_{k}$ can be identified with a subgroup of the free abelian group $H_i(F_{k}(\R^2))$ where the products of iterated brackets have no degree-0 singleton factors.  In particular, $H_0^{\FI} (H_k(F(\R^2)))_{2k}$ has a basis indexed by the set of perfect matchings on $[2k]$, where the matching $\{a_i, b_i\}_{i=1}^k$ corresponds (up to sign) to the homology class $\psi(a_1, b_1)\bullet \psi(a_2, b_2) \bullet \cdots \bullet \psi(a_k, b_k) $, as in Figure \ref{H3F6Disk}.
\begin{figure}[!ht]    \centering
\includegraphics[scale=.18]{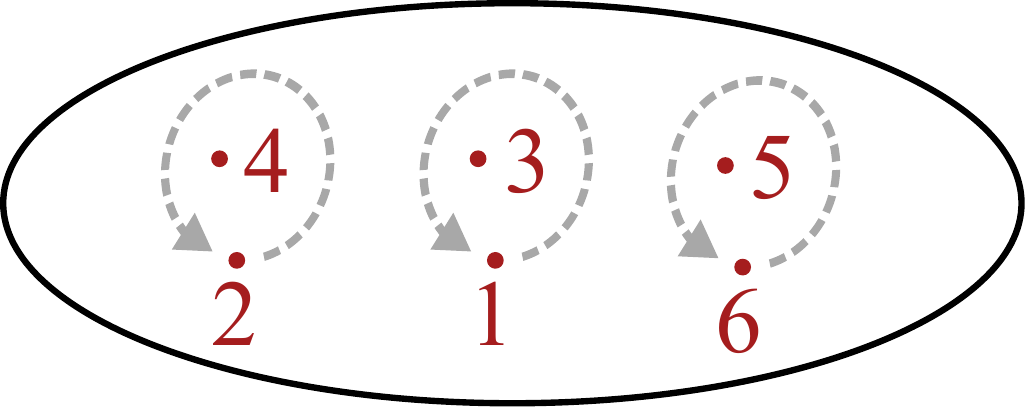}
\caption{The basis element for $H_0^{\FI} (H_3(F(\R^2)))_6$ corresponding the matching $\left\{ \{4,2\}, \{3,1\}, \{5,6\} \right\}$. }
\label{H3F6Disk}
\end{figure}  
This description follows from the work of Cohen. As an $\fS_{2k}$--representation, the group $H_0^{\FI} (H_k(F(\R^2)))_{2k}$ is precisely $M^{\FIM^+}(0)_{2k}$. Since $t_{\psi(a,b)}$ is the operation $x \mapsto  \psi(a,b) \bullet x$, this identification is compatible with the $\FIM^+$ action. 
\end{proof}
 
 The decomposition of the $\fS_{2k}$--representation $H_0^{\FI}(H_k(F(\R^2);\Q))_{2k}$ into irreducible constituents of is given explicitly in Proposition \ref{PropDecomposingM(0)}.
 
\begin{remark}
The methods used to prove Proposition \ref{PropWR20} can be used for other calculations. For example, if $M$ is a punctured $2$-disk, then $\W_0^M \cong M^{\FIM^+}(0)$ and $\W_1^M \cong M^{\FIM^+}(H_1(M)) \bigoplus M^{\FIM^+}(\Lie_3)$.
\end{remark}

 \subsubsection*{Computing $\W^M_0$ for some surfaces $M$}

\begin{proposition} \label{PropMobiusGenus} Let $M$ be a connected surface. If $M$ is not orientable or of genus greater than zero, then $$ \W^M_0(0) \cong \Z \qquad \text{and} \qquad \W^M_0(2i) \cong 0 \quad \text{for $i>0$.} $$ \end{proposition}
 
\begin{proof}
By definition, $\W^M_0(0)= H_0^{\FI} \left( H_0 (F(M)) \right)_0 = H_0 (F_0(M)).$  As claimed, this is isomorphic to $\Z$ for any connected manifold $M$. 

To prove the vanishing of $\W^M_0(2i) = H_0^{\FI} \left( H_i(F(M)) \right)_{2i}$ for $i > 0$ we will first show, by assembling known results, that the map 
$$H_0^{\FI}(H_1(F(\R^2)))_2 \longrightarrow H_0^{\FI}(H_1(F(M)))_2$$ 
induced by embedding $\R^2 \hookrightarrow M$ is zero. We begin with the case that $M$ is nonorientable. Let $\mathcal M$ denote the M\"obius strip.  Since $\mathcal M$ is open, $H_1(F(\mathcal M))$ is an FI$\sharp$-module. By Church--Ellenberg--Farb \cite[Theorem 4.1.5]{CEF} (here Theorem \ref{4.1.5}) it follows that 
\[ H_1(F_2(\mathcal M)) \cong \bigoplus_{\ell =0}^2 \Ind_{\fS_\ell \times \fS_{2-\ell} }^{\fS_2} H_0^{\FI }(H_1(F(\mathcal M)))_\ell \boxtimes \Z, \] with $\Z$ the trivial  $ \fS_{2-\ell}$-representation.
Since $H_1(F_1(\mathcal M)) \cong H_1(\mathcal M) \cong \Z $ and $H_1(F_0(\mathcal M)) \cong 0$, 
the component of  $H_1(F_2(\mathcal M))$ generated in degrees $\ell=0,1$ is isomorphic to 
$$ \Ind_{\fS_1 \times \fS_1}^{\fS_2} H_0^{\FI }(H_1(F(\mathcal M)))_1 \boxtimes \Z \cong \Ind_{\fS_1 \times \fS_1}^{\fS_2} \Z \cong \Z^2,$$
the canonical $\fS_2$ permutation representation. By Theorem \ref{4.1.5} this component is a direct summand of $H_1(F_2(\mathcal M))$. Wang  \cite[Lemma 1.6]{Wang} showed that $H^1(F_2(\mathcal M)) \cong \Z^2$ and that $H^2(F_2(\mathcal M)) \cong \Z$ and hence is torsion free. We deduce that $H_1(F_2(\mathcal M)) \cong \Z^2$ consists entirely of its $\ell=1$ component: the component $H_0^{\FI }(H_1(F(\mathcal M)))_2$ generated in degree $\ell=2$ is zero. Hence the map 
$$H_0^{\FI}(H_1(F(\R^2)))_2 \longrightarrow H_0^{\FI}(H_1(F(\mathcal M)))_2$$ 
is zero. For a general noncompact nonorientable surface $M$, the map $$H_0^{\FI}(H_1(F(\R^2)))_2 \longrightarrow H_0^{\FI}(H_1(F(M)))_2$$ factors through $H_0^{\FI}(H_1(F(\mathcal M)))_2$, and so is zero. 

From a presentation of $\pi_1 (F_k(M))$ for $M$ a noncompact, orientable positive genus surface (for example, Bellingeri \cite[Theorem 6.1]{Bellingeri}) we see that the map $H_1(F_2(\R^2)) \m H_1(F_2(M))$ is zero even before passing to the quotient module  of minimal generators. 

Consider the arc resolution spectral sequence described in Section \ref{SectionArcRes}. By the proof of Theorem \ref{ConfigSpaceRepStable}, the domain of any differentials $d^r_{p,q}$ with codomain $E^r_{-1,i}(2i)$ are zero for $r>2$. Since $E^r_{p,q}(2i)=0$ for $p<-1$, there are no nontrivial differentials out of the group $E^r_{-1,i}$. High connectivity of the arc resolution (Proposition \ref{ArcF}) implies that $E^\infty_{-1,i}(2i) \cong 0$ for $i>0$. Thus the differentials
$d^2 : E^r_{1,i-1}(2i) \to E^r_{-1,i}(2i)$
 are surjective for $i>0$. Equivalently, the maps 
 $$\Ind^{\fS_{2i}}_{\fS_{2i-2} \times \fS_2} \W_0^M(2i-2) \boxtimes H_1(F_2(\R^2)) \longrightarrow \W_0^M(2i) $$ 
 surject for all $i >0$. We have shown that the map $H_0^{\FI}(H_1(F(\R^2)))_2 \longrightarrow H_0^{\FI}(H_1(F(M)))_2$ is zero if $M$ is not orientable or has positive genus, and so for $i=1$, 
 $$\Ind^{\fS_{2}}_{\fS_{0} \times \fS_2} \W_0^M(0) \boxtimes H_1(F_2(\R^2)) \longrightarrow \W_0^M(2) $$  is the zero map. Since it is also surjective, $\W_0^M(2) \cong 0$. The claim for higher $i$ then follows inductively, using the fact that only the zero group can be the surjective image of the zero group. 
\end{proof}
 
In Proposition \ref{PropMobiusGenus} we saw that for nonorientable or positive genus surfaces $M$, $$ \W^M_0(2i) = H_0^{\FI} \left( H_i(F(M))\right)_{2i}=0 \qquad \text{for $i>0$}, $$ and this gives the following small improvement on known stable ranges.
 
\begin{corollary}
Let $M$ be a connected noncompact manifold which is not a (possibly punctured) $2$--disk, and let $i>0$. Then $H_i(F(M)) $ is generated in degree $\leq 2i-1$.
\end{corollary}

\subsubsection*{The combinatorics of FIM$^+$-modules}

There has been considerable recent success in characterizing the structure of finitely generated modules over the category FI and certain relatives, and these results suggest a number of questions about what ``representation stability" should mean for modules over FIM$^+$. In Proposition \ref{PropDecomposingM(d)}, we describe the decomposition of free FIM$^+$-modules over $\Q$ into irreducible $\fS_k$-representations, using a calculation of $M^{\FIM^+}(0)$ stated in Proposition \ref{PropDecomposingM(0)}. In Question \ref{QuestionRepStability}, we pose some questions about the structure of finitely generated rational FIM$^+$-modules. 

Let $\fB_k \cong \fS_2 \wr \fS_k \subseteq \fS_{2k}$ denote the signed permutation group on $k$ letters, the Coxeter group in type $B_k/C_k$. Let $V_{(1^k, \varnothing)}$ denote the $1$-dimensional rational $\fB_k$-representation pulled back from the sign $\fS_k$-representation under the natural surjection $\fB_k \to \fS_k$. There are isomorphisms of $\fS_{2k}$-representations:
$$M^{\FIM^+}(0)_{2k} \cong \Ind_{\fB_k}^{\fS_{2k}} V_{(1^k, \varnothing)} .$$ 
The decompositions of these induced representations are described explicitly by Stembridge \cite[Page 7]{StembridgeNotes}, a result which he attributes to Littlewood. These decompositions are as follows. 

\begin{proposition}[{Littlewood, Stembridge \cite[Page 7]{StembridgeNotes}}] \label{PropDecomposingM(0)} There are isomorphisms of $\fS_{2k}$-representations:
$$ M^{\FIM^+}(0)_{2k} \cong \bigoplus_{\lambda \in D_{2k} } V_{\lambda}.$$ 
Here $V_{\lambda}$ the irreducible $\fS_{2k}$-representation associated to the partition $\lambda$. A partition $\lambda \vdash 2k$ is in $D_{2k}$ if and only if it has the following symmetry: when the associated young diagram (in English notation) is cut into two along the staircase shown in Figure \ref{Staircase}, 
\begin{figure}[!ht]    \centering
\includegraphics[scale=.8]{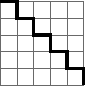}
\caption{Staircase dividing young diagrams into two skew subdiagrams.}
\label{Staircase}
\end{figure}  
then the resultant two skew subdiagrams are symmetric under reflection in the line of slope $-1$. 
\end{proposition} 
Figure \ref{PartitionsOf6} illustrates this symmetry in the case $2k=6$. 
\begin{figure}[!ht]    \centering
\begin{subfigure}{.45\textwidth}
  \centering
{\includegraphics[scale=2]{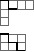}}
 \caption{Partitions of $6$ contained in $D_{6}$}
\end{subfigure} %
\begin{subfigure}{.45\textwidth}
  \centering
{\includegraphics[scale=2]{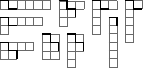}}
\caption{ Partitions of $6$ not in $D_6$.}
\end{subfigure}
\caption{The set $D_6$.}
\label{PartitionsOf6}
\end{figure}  
Notably, identifying a partition in $D_{2k}$ with one of its skew subdiagrams puts $D_{2k}$  in bijection with \emph{strict} partitions of $k$, that is, partitions with distinct parts.  This computation of  $M^{\FIM^+}(0)$  allows us to use the Littlewood-Richardson rule to compute the decomposition of any free rational $\FIM^+$-module. 

\begin{proposition} \label{PropDecomposingM(d)} 
Given a rational $\fS_d$--representation $W$, the associated free $\FIM^+$-module 
$$M^{\FIM^+}(W)  \cong M^{\FIM^+}(d) \otimes_{\Q[\fS_d]} W$$ 
has the following decomposition: 
\begin{align*}
M^{\FIM^+}(W)_k \cong \left\{ \begin{array}{ll} \Ind^{\fS_k}_{\fS_d \boxtimes \fS_{k-d} } W \boxtimes M^{\FIM^+}(0)_{k-d},  & k \equiv d \pmod{2} \\ 
 0 , & k \not\equiv d \pmod{2} \\ 
\end{array} \right. 
\end{align*}
\end{proposition}

\begin{example} For example, the first five nonzero components of the rational module $M^{\FIM^+}(1)$ decompose as follows:
 \begin{align*}
M^{\FIM^+}(1)_1 & \cong V_{\Y{1}}   \qquad  \qquad 
M^{\FIM^+}(1)_3  \cong  V_{\Y{3}} \oplus V_{\Y{2,1}}  \qquad \qquad 
M^{\FIM^+}(1)_5  \cong V_{\Y{4,1}} \oplus V_{\Y{3,1,1}} \oplus V_{\Y{3,2}}    \\   
M^{\FIM^+}(1)_7 & \cong V_{\Y{5,1,1}}  \oplus V_{\Y{4,3}} \oplus V_{\Y{4,2,1}}  \oplus V_{\Y{4,1,1,1}}   \oplus V_{\Y{3,3,1}}    \\
M^{\FIM^+}(1)_9 & \cong V_{\Y{6,1,1,1}} \oplus  V_{\Y{5,3,1}} \oplus V_{\Y{5,2,1,1}} \oplus V_{\Y{5,1,1,1,1}} 
\oplus  V_{\Y{4,4,1}}  \oplus  V_{\Y{4,3,2}}  \oplus  V_{\Y{4,3,1,1}} 
\end{align*} 
\end{example}

In analogy to the other categories and (skew-)tca's that have been studied under the scope of ``representation stability," we pose the following questions.

 \begin{question} \label{QuestionRepStability}  
 What constraints does finite generation put on the irreducible representations appearing in a rational $\FIM^+$-module?  Given a finitely generated rational $\FIM^+$-module $\cV$, is there some operation on Young diagrams for constructing $\cV_{k+1}$ from  $\cV_k$ in the stable range in the spirit of Church--Farb's \emph{multiplicity stability} \cite[Definition 1.1]{CF}? Does $\cV_k$ even determine $\cV_{k+1}$ for large $k$?
 
\end{question}

\subsubsection*{Algebraic finiteness properties for twisted (skew-)commutative algebras}

Conjecture \ref{highdimconj} suggests the following purely algebraic questions.
 
\begin{question}
Let $R$ be a Noetherian ring. If $\mathcal W$ is a finitely generated $\mathfrak L_d^n$-module and $\mathcal V$ is a $\mathfrak L_d^n$-submodule, is $\mathcal V$ necessarily finitely generated? \label{conjNoth}
\end{question}
 
The main theorem of Church--Ellenberg--Farb--Nagpal \cite{CEFN} shows that the answer is yes for $d=1$. For $R$ a field of characteristic zero, the answer is yes for $d=2$. This is due to Nagpal--Sam--Snowden who address the case when $n$ is odd \cite[Theorem 1.1]{SSN1} and the case when $n$ is even \cite[Theorem 1.1]{NSSskew}. The following question generalizes Church--Ellenberg's results in \cite{CE} to the case of $d>1$. For a (skew-)tca $\mathcal A$, let $H_i^{\mathcal A}$ denote the ith left derived functor of $H_0^{\mathcal A}$.
 
\begin{question} Is there a function $f:\N_0 \times \N_0 \times \N_0 \m \N_0$ such that for all $\mathfrak L_d^n$-modules $W$ with $\deg H_0^{\mathfrak L_d^n}(W)=g$ and $\deg H_1^{\mathfrak L_d^n}(W)=r$, then $\deg H_i^{\mathfrak L_d^n}(W) \leq f(g,r,i)$? \label{higherCE}
\end{question}

An affirmative answer to Question \ref{higherCE} for $d=n=2$ would allow us to prove a quantitative version of Theorem \ref{maintheorem}. An affirmative answer to either of these two questions for $d>2$ seems relevant to establishing tertiary and higher order representation stability, though more ideas appear to be needed.

{\footnotesize
\bibliographystyle{amsalpha}
\bibliography{FI}
}

\end{document}